\documentclass[11pt,a4paper, reqno]{article}
\usepackage[dvipsnames]{xcolor}

\usepackage[utf8]{inputenc}
\usepackage{amsfonts}
\usepackage{amssymb}
\usepackage{enumerate}

\newenvironment{enumeratea}{\begin{enumerate}[\upshape (a)]}{\end{enumerate}}

\usepackage{graphicx}
\usepackage{mathtools}
\usepackage{amsmath, amsthm}
\usepackage{tikz}
\usepackage{color}
\usepackage[affil-it]{authblk}
\usepackage[T1]{fontenc}
\usepackage[sort,numbers]{natbib}
\usepackage{bbm}
\usepackage[a4paper, top = 1in, bottom = 1in, left = 1in, right = 1in]{geometry}
\usepackage{amsmath}
\usepackage{amsfonts}
\usepackage{amssymb}
\usepackage{bbm}
\usepackage{mathrsfs}
\usepackage[utf8]{inputenc}
\usepackage[english]{babel}
\usepackage{hyperref}
\usepackage{relsize}
\usepackage{accents}
\usepackage[noabbrev]{cleveref}
\usepackage{appendix}
\usepackage{lipsum}
\usepackage[color = white]{todonotes}

\newcommand\blfootnote[1]{%
  \begingroup
  \renewcommand\thefootnote{}\footnote{#1}%
  \addtocounter{footnote}{-1}%
  \endgroup
}

\makeatletter
  \@addtoreset{chapter}{part}
  \@addtoreset{@ppsaveapp}{part}
\makeatother

\long\def\/*#1*/{}

\usepackage{palatino}

\makeatletter
\newsavebox\myboxA
\newsavebox\myboxB
\newlength\mylenA

\newcommand*\xoverline[2][0.75]{%
    \sbox{\myboxA}{$\m@th#2$}%
    \setbox\myboxB\null
    \ht\myboxB=\ht\myboxA%
    \dp\myboxB=\dp\myboxA%
    \wd\myboxB=#1\wd\myboxA
    \sbox\myboxB{$\m@th\overline{\copy\myboxB}$}
    \setlength\mylenA{\the\wd\myboxA}
    \addtolength\mylenA{-\the\wd\myboxB}%
    \ifdim\wd\myboxB<\wd\myboxA%
       \rlap{\hskip 0.5\mylenA\usebox\myboxB}{\usebox\myboxA}%
    \else
        \hskip -0.5\mylenA\rlap{\usebox\myboxA}{\hskip 0.5\mylenA\usebox\myboxB}%
    \fi}
\makeatother

\usepackage{environ}
\NewEnviron{eq}{%
\begin{equation}\begin{split}
  \BODY
\end{split}\end{equation}
}

\usepackage{hyperref}
\hypersetup{
  colorlinks,
  linkcolor=blue,
  citecolor=black,
  filecolor=black,
  urlcolor=black,
  pdftitle={},
  pdfauthor={S. Bhamidi, S. Dhara, R. van der Hofstad},
  pdfcreator={S. Dhara},
  pdfsubject={},
  pdfkeywords={}
}
\numberwithin{equation}{section}

\def\sss{\scriptscriptstyle}

\newcommand{\var}[1]{\ensuremath{\mathrm{Var}\left(#1\right)}}

\newcommand{\floor}[1]{\ensuremath{\left\lfloor #1 \right\rfloor}}
\newcommand{\ind}[1]{\ensuremath{\mathbbm{1}_{\left\{#1\right\}}}}
\newcommand{\pto}{\ensuremath{\xrightarrow{\mathbbm{P}}}}
\newcommand{\dto}{\ensuremath{\xrightarrow{d}}}

\newcommand{\PR}{\ensuremath{\mathbbm{P}}}
\newcommand{\E}{\ensuremath{\mathbbm{E}}}

\newcommand{\Z}{\ensuremath{\mathbbm{Z}}}
\newcommand{\e}{\ensuremath{\mathrm{e}}}

\newcommand{\OP}{\ensuremath{O_{\sss\PR}}}
\newcommand{\oP}{\ensuremath{o_{\sss\PR}}}
\newcommand{\thetaP}{\ensuremath{\Theta_{\sss\PR}}}

\newcommand{\shortarrow}{{\sss\downarrow}}

\newcommand{\dif}{\mathrm{d}}
\newcommand{\bld}[1]{\boldsymbol{#1}}

\newcommand{\GRG}{\mathrm{GRG}_n(\bld{w})}

\newcommand{\NR}{\mathrm{NR}_n(\bld{w})}
\newcommand{\CL}{\mathrm{CL}_n(\bld{w})}
\newcommand{\NRp}{\mathrm{NR}_n(\bld{w},\perc)}

\newcommand{\rGRG}{\mathrm{GRG}_n}
\newcommand{\rNR}{\mathrm{NR}_n}
\newcommand{\rCL}{\mathrm{CL}_n}
\newcommand{\bw}{\bld{w}}

\newcommand{\sC}{\mathscr{C}}

\newcommand{\cGinf}{\mathscr{G}_{\sss \infty}}
\newcommand{\sW}{\mathscr{W}}
\newcommand{\sCi}{\mathscr{C}_{\sss (i)}}

\newcommand{\ber}{\mathrm{Bernoulli}}
\newcommand{\poi}{\mathrm{Poisson}}
\newcommand{\dTV}{\mathrm{d}_{\sss \mathrm{TV}}}
\newcommand{\cf}{c_{\sss \mathrm{F}}}
\newcommand{\dst}{\mathrm{d}}

\newcommand{\sB}{\mathscr{B}}
\newcommand{\cS}{\mathcal{S}}
\newcommand{\spn}{\mathrm{Span}_a}
\newcommand{\rspn}{\mathrm{Span}}

\newtheorem{theorem}{Theorem}

\newtheorem{lemma}[theorem]{Lemma}
\newtheorem{proposition}[theorem]{Proposition}
\newtheorem{corollary}[theorem]{Corollary}
\newtheorem{assumption}[theorem]{Assumption}
\newtheorem{remark}[theorem]{Remark}
\newtheorem{fact}[theorem]{Fact}
\newtheorem{defn}[theorem]{Definition}

\numberwithin{theorem}{section}

\newcommand{\erdos}{Erd\H{o}s-R\'enyi }

\def\qed{ \hfill $\blacksquare$}

\newcommand{\cA}{\mathcal{A}}\newcommand{\cB}{\mathcal{B}}\newcommand{\cC}{\mathcal{C}}
\newcommand{\cE}{\mathcal{E}}

\newcommand{\cG}{\mathcal{G}}\newcommand{\cI}{\mathcal{I}}

\newcommand{\cN}{\mathcal{N}}
\newcommand{\cR}{\mathcal{R}}
\newcommand{\cW}{\mathcal{W}}\newcommand{\cX}{\mathcal{X}}

\newcommand{\mvtheta}{\boldsymbol{\theta}}

\newcommand{\bZ}{\mathbb{Z}}

\newcommand{\xx}{\boldsymbol{x}} 
\newcommand{\yy}{\boldsymbol{y}}

\newcommand{\set}[1]{\left\{#1\right\}}

\newcommand{\Var}{\mathrm{Var}}
\newcommand{\cT}{\mathcal{T}}
\newcommand{\sCa}{\mathscr{C}^a}

\newcommand{\eqn}[1]{\begin{equation} #1 \end{equation}}
\newcommand{\eqan}[1]{\begin{align} #1 \end{align}}

\newcommand{\vep}{\varepsilon}

\newcommand{\nn}{\nonumber}

\newcommand{\conn}{\longleftrightarrow}
\newcommand{\Poi}{{\sf Poisson}}
\newcommand{\indic}[1]{\mathbbm{1}_{\{#1\}}}
\newcommand{\indicwo}[1]{\mathbbm{1}_{#1}}
\newcommand{\seta}{\eta_s}
\newcommand{\per}{\pi}
\newcommand{\perc}{\pi_c}
\newcommand{\percl}{\pi_c(\lambda)}
\newcommand{\percn}{\pi_n}
\newcommand{\bfT}{\mathbf{T}}

\begin{document}
\title{Multiscale genesis of a tiny giant for\\
percolation
on scale-free random graphs}
\author[Bhamidi, Dhara, van der Hofstad]{Shankar Bhamidi$^1$, Souvik Dhara$^{2}$, Remco van der Hofstad$^3$}
  \date{\today}
 \maketitle
 \blfootnote{\emph{Emails:} 
 \href{mailto:bhamidi@email.unc.edu}{bhamidi@email.unc.edu},
 \href{mailto:sdhara@mit.edu}{sdhara@mit.edu},
 \href{mailto:r.w.v.d.hofstad@tue.nl}{r.w.v.d.hofstad@tue.nl}} 
\blfootnote{$^1$Department of Statistics and Operations Research,  University of North Carolina}
\blfootnote{$^2$Department of Mathematics, Massachusetts Institute of Technology}
\blfootnote{$^3$Department of Mathematics and Computer Science, Eindhoven University of Technology}

\begin{abstract}
We study the critical behavior for percolation on inhomogeneous random networks on~$n$ vertices, where the weights of the vertices follow a power-law distribution with exponent $\tau \in (2,3)$. 
Such networks, often referred to as \emph{scale-free networks}, exhibit critical behavior when the percolation probability tends to zero at an appropriate rate, as $n\to\infty$. 
We identify the critical window for a host of scale-free random graph models such as the Norros-Reittu model, Chung-Lu model and generalized random graphs.
Surprisingly, there exists a finite time inside the critical window, after which, we see a sudden emergence of a tiny giant component.  This is a novel behavior which is in contrast with the critical behavior in other known universality classes with $\tau \in (3,4)$ and $\tau >4$.

Precisely, for edge-retention probabilities $\percn = \lambda  n^{-(3-\tau)/2}$, there is an explicitly computable $\lambda_c>0$ such that the critical window is of the form $\lambda \in (0,\lambda_c),$ where the largest clusters have size of order $n^{\beta}$ with $\beta=(\tau^2-4\tau+5)/[2(\tau-1)]\in[\sqrt{2}-1, \tfrac{1}{2})$ and have non-degenerate scaling limits, 
while in the supercritical regime $\lambda > \lambda_c$, a unique 
`tiny giant' component of size $\sqrt{n}$ emerges.
For $\lambda \in (0,\lambda_c),$ 
the scaling limit of the maximum component sizes can be described in terms of components of a one-dimensional inhomogeneous percolation model on $\Z_+$ studied in a seminal work by Durrett and Kesten~\cite{DK90}. 
For $\lambda>\lambda_c$, we prove that the sudden emergence of the tiny giant is caused by a phase transition inside a smaller core of vertices of weight~$\Omega(\sqrt{n})$.
\end{abstract}

\section{Introduction}
\subsection{Background}
Percolation phase transitions are one of the foundational tenets in the application of probabilistic combinatorics to areas ranging from statistical physics to social dynamics \cite{Gri99}.
At the simplest level, one starts with a base (potentially random) graph. For a parameter $\pi$, each edge in the graph is retained with probability $\pi$ and deleted with probability $1-\pi$,  independently across edges. The first questions of interest is understanding the emergence of a giant connected component as one increases the value of $\pi$, and identifying critical values of this parameter where abrupt changes in the connectivity occurs. 
These question arise as building blocks for more complex interacting particle systems  
 e.g.\ in the study of epidemics,  condensed matter theory, robustness of networks such as the Internet when the edges of the underlying network experience random failure \cite{Bar16,Newman-book,DGM08}.

Unlike phase transition on infinite graphs such as lattices, there is typically no unique value for phase transition in large but finite graphs. Instead, there is an interval of $\pi$-values, often referred to as the \emph{critical window}, where this structural transition in the component sizes takes place. To fix ideas, let us recall classical results for percolation on complete graphs with~$n$ vertices and $\pi = c/n$ or Erd\H{o}s-R\'enyi random graphs $\mathrm{ER}(n,c/n)$. 
It is well known that the critical window is given by $c(\lambda)=1+\lambda n^{-1/3}$ for $-\infty<\lambda<\infty$ \cite{JLR00,A97,JKLP93}, i.e., if $\sCi$ denotes the $i$-th largest component, then 
\begin{enumeratea}
	\item If $\lambda=\lambda_n \to -\infty$ and $|\lambda_n| = o(n^{1/3})$, then $\frac{|\sCi|}{2\lambda_n^2 n^{2/3}\log |\lambda_n|} \pto 1$ for all $i\geq 1$. 
	\item If $\lambda=\lambda_n \to +\infty$ and $\lambda_n = o(n^{1/3})$, then $\frac{|\sC_{\sss (1)}|}{2|\lambda_n| n^{2/3}} \pto 1$ and $\frac{|\sCi|}{2\lambda_n^2 n^{2/3}\log |\lambda_n|} \pto 1$ for all $i\geq 2$. 
	\item Inside the critical window when $\lambda$ is fixed, $n^{-2/3}|\sCi|$ converges in distribution to non-degenerate strictly positive random variables whose distribution depends on $\lambda$.
\end{enumeratea}
Thus the largest component sizes concentrate outside the critical window, whereas they yield  non-degenerate  scaling limits in the critical window which \emph{sensitively} depends on the \emph{precise} location in the scaling window given by $\lambda$. 
Starting with the pioneering work by  Janson, Knuth, {\L}uczak, Pittel \cite{JKLP93} and  Aldous~\cite{A97},  the study of critical behavior has inspired an enormous literature with several  scaling-limit results showing qualitatively similiar behavior as in \erdos random graph for largest component sizes \cite{AP00,BHL10,DHLS15,Jo10,NP10a,R12} and their metric structure \cite{ABG09,BBSX14,BDW20}, as well as qualitatively different behavior for component sizes \cite{BHL12,Jo10,DHLS16,AL98} and their metric structure \cite{BHS15,BDHS17,CG20,BDW20}.
See \cite[Chapter 1]{Dha18} for a detailed literature overview.

\subsection{Overview of our contributions}
In this paper, we prove a new type of phase transition phenomenon in the emergence of maximally connected components in certain random graphs, and develop techniques in probabilistic combinatorics necessiated by such models. The starting point is random network models with power-law degree distributions with exponent $\tau \in (2,3)$.  These models are enormously popular in applications owing to empirical observations that many real world systems (World-Wide Web, social networks, protein interaction networks \cite{Bar16}) seem to exhibit 
qualitative properties similar to such models. Mathematically these models turn out to be significantly challenging (as will be further evident below), since they contain extremal degree vertices \emph{at many different scales} 
which play crucial and central roles in the connectivity pattern at specific phases of the percolation process. 
In this context, our main contributions include:  

\paragraph{New universality class:} This paper considers a number of major families of scale-free random graph models with degree exponent $\tau \in (2,3)$ related to Aldous's multiplicative coalescent~\cite{A97}; these include models such as the {\em Norros-Reittu model}, the {\em Chung-Lu model}, and the {\em generalized random graph} (see Section~\ref{sec-RG-mod} for more details). This class of models has turned out to be central in understanding universality phenomenon for critical random graphs in the sense that, once these models have been understood, a host of other canonical random graph models can all be proven to have the same asymptotic behavior in the critical regime, see e.g.  \cite{BBSX14,BDHS17,BHS15,BBW12}. We show that the critical window for percolation on scale-free random graphs will be given by 
	\begin{eq}
	\label{pc-def-NR}
	\perc (\lambda) = \lambda n^{- (3-\tau)/2}, \qquad \lambda\in (0,\lambda_c),
	\end{eq}
for some explicitly computable model dependent critical time $\lambda_c$. 
Thus, surprisingly, 
the critical window is given by a bounded interval $\lambda \in (0,\lambda_c)$. 
In other words, if we look at the coalescence of the critical components as the percolation parameter transitions through the critical window, the components evolve in a non-trivial manner only up to a finite time $\lambda_c$, after which all of the critical components suddenly coalesce with each other. 
This is in contrast with Aldous' multiplicative coalescent, where the coalescence happens over an infinite length window. 
This phenomenon had not been predicted by the extensive investigation via numerical approaches on these models carried out in areas such as statistical physics and condensed matter theory.  

\paragraph{Multiscale emergence of connectivity and technical novelty:} Analyzing the critical regime of models in this class present significant technical challenges as standard techniques based on exploration processes or differential equations cannot be implemented; rather one needs to carefully understand the contribution of extremal degree vertices or \emph{hubs} of different scales contributing to connectivity at each value of $\pi_c(\cdot)$. More precisely: 
\begin{enumeratea}
\item {\bf Critical scaling window:} For $\lambda\in (0, \lambda_c)$, we show that the maximal component sizes scale like $n^{\beta}$ with $\beta = \frac{1}{\tau-1} - \frac{3-\tau}{2}\in[\sqrt{2}-1, \frac{1}{2})$,
and the rescaled vector of ordered component sizes converges to a non-degenerate random vector in $\ell^2$-topology. 
The distributional asymptotics can be derived in terms of an inhomogeneous percolation model on $\Z_+$, which represents the core connectivity structure between the hubs. In this regime,  connectivity emerges owing to interconnections between \emph{macro-hubs}, namely maximal degree vertices (with weights $n^{1/(\tau-1)}$). However note that with $\pi_c(\lambda)\to 0$, these macro-hubs cannot be directly connected; rather (with positive probability) they are connected via \emph{two step paths} thorugh intermediete scale \emph{meso-hubs} of weight $n^{(\tau-2)/(\tau-1)}$.  This interconnected structure forms the \emph{core} of the critical components, and we use path-counting techniques to show that the 1-neighborhood of the core spans the critical components (see Figure~\ref{fig:exploration}). 
The core can be coupled with a one-dimensional inhomogeneous percolation model on $\Z_+$, which was studied in a seminal work of Durrett and Kesten~\cite{DK90} and in follow-up work 
by Zhang~\cite{zhang1991power}.
\item {\bf Supercritical regime:} For $\lambda > \lambda_c$, instead, we show that there is a \emph{unique giant component} of size $\sqrt{n}\gg n^\beta$, and the size of the rescaled giant component concentrates. In this case we show that the graph restricted to a special set of vertices of weight at least $\sqrt{n}$ can be approximated by a well-behaved inhomogeneous random graph in the spirit of Bollob\'as, Janson and Riordan \cite{BJR07}.
A small giant component (of size $\sqrt{n}$) appears inside this restricted set precisely when  $\lambda>\lambda_c$. 
This forms the core of the giant component in the whole graph, and again the 1-neighborhood of the core spans the giant component (and, in fact, the core itself is, in size, negligible to this 1-neighborhood). Analyzing the resulting random structure requires several delicate estimates of multi-type branching process as well as a careful topological analysis of paths exiting and returning to these special class of vertices. 
\end{enumeratea}

\section{Main results}
\label{sec-mod-res}

\subsection{Preliminaries: notation, convergence and topologies} 
\label{sec:notation}
To describe the main results of this paper, we need some definitions and notations. 
We use $\xrightarrow{\sss\PR}$ and $\xrightarrow{\sss d}$ to denote convergence in probability and in distribution respectively.
The topology needed for convergence in distribution will be specified unless clear from the context. 
We use the Bachmann-Landau notation $O(\cdot)$, $o(\cdot)$, $\Theta(\cdot)$ for large $n$ asymptotics of real numbers.
For two real sequences $(a_n)_{n\geq 1},(b_n)_{n\geq 1} $, write $a_n \asymp b_n$ for $a_n/b_n = (1+o(1))$.
 A sequence of events $(\mathcal{E}_n)_{n\geq 1}$ is said to occur with high probability~(whp) with respect to the associated sequence of probability measures $(\mathbbm{P}_n)_{n\geq 1}$  if $\mathbbm{P}_n\big( \mathcal{E}_n \big) \to 1$. 
For two sequences of real-valued random variables $(X_n)_{n\geq 1}$ and $(Y_n)_{n\geq 1}$, write $X_n = O_{\sss\mathbbm{P}}(Y_n)$ if  $ ( |X_n|/|Y_n| )_{n \geq 1} $ is a tight sequence; $X_n =o_{\sss\mathbbm{P}}(Y_n)$ when $X_n/Y_n  \xrightarrow{\sss\PR} 0 $; $X_n =\thetaP(Y_n)$ if both $X_n=\OP(Y_n) $ and $Y_n=\OP(X_n)$. 
We use $C,C',C_1,C_2$ etc as generic notation for positive constants whose value can change from line to line.   Fix $\tau \in (2,3)$.  Throughout this paper, we denote 
	\begin{equation}\label{eqn:notation-const}
 	\alpha= 1/(\tau-1),\qquad \rho=(\tau-2)/(\tau-1),\qquad \eta=(3-\tau)/(\tau-1), \qquad \seta = (3-\tau)/2.
	\end{equation}
For $p> 0$, let $\ell^p$ denote the collection of sequences $ \xx= (x_1, x_2, x_3, ...)$ with $\sum_{i=1}^{\infty} |x_{i}|^p < \infty$. Equip this space with the $p$-norm metric $d(\xx, \yy)= \big( \sum_{i=1}^{\infty} |x_i-y_i|^p \big)^{1/p}$. Let $\ell^p_{\shortarrow} \subset \ell^p$ be the collection of sequences $\xx$ with $x_i\geq 0$ for all $i$ and the elements of the sequence arranged in non-increasing order.

\subsection{Scale-free random graph models}
\label{sec-RG-mod}
We now describe the main models studied in this paper. Given a set of weights $\bw=(w_i)_{i\in [n]}$ on the vertex set $[n]$, the Poissonian random graph or Norros-Reittu model \cite{NR06}, denoted by $\NR$, is generated by creating an edge between vertex $i$ and $j$ independently with probability
	\begin{equation}
	\label{eq:p-ij-NR-defn}
	p_{ij} = p_{ij}^{\sss \rm NR}:=1-\e^{-w_iw_j/\ell_n},
	\end{equation}
where $\ell_n = \sum_{i\in [n]}w_i$ denotes the total weight. 
Our results for the critical window  
hold more generally, for example, for the Chung-Lu Model~\cite{CL02,CL02b} (denoted by $\CL$) with 
	\eqn{
	\label{CL-pij}
	p_{ij}^{\sss\rm CL}: = \min \{w_iw_j/\ell_n,1\},
	}
and the generalized random graph model~\cite{BDM06} (denoted by $\GRG$)  with 
	\eqn{
	\label{GRG-pij}
	p_{ij}^{\sss \rm GRG}: = \frac{w_iw_j}{\ell_n+w_iw_j}. 
	}
The final model has the property that, conditionally on the degree sequence~$\bld{d}=(d_i)_{i\in[n]}$, the law of the obtained random graph is the same as that of a uniformly chosen graph from the space of all simple graphs with degree distribution $\bld{d}$ (cf.~\cite[Theorem 6.15]{RGCN1}).

The percolated graph $\rNR(\bw,\per)$ is obtained by keeping each edge of the graph independently with probability $\per$. This deletion process is also independent of the randomization of the graph. Naturally, the behavior of $\rNR(\bw)$, and thus of $\rNR(\bw,\per)$ depends sensitively on the choice of vertex weights. The following choice of vertex weights will give rise to scale-free random graphs:
	\begin{assumption}[Scale-free weight structure]
	\label{assumption-NR}
	\normalfont
	For some $\tau \in (2,3)$, consider the distribution function~$F$ satisfying $[1-F](w) = Cw^{-(\tau-1)}$ for some $C>0$, and let $w_i = [1-F]^{-1}(i/n)$. 
	\end{assumption} 
In this setting, if $W_n$ denotes the weight of a vertex chosen uniformly at random, then $W_n$ will satisfy an asymptotic power-law in the sense that for any $w> 0$,  $\PR(W_n>w)\to Cw^{-(\tau-1)},$ and, as a result, the asymptotic weight distribution will have the same exponent~$\tau$ (see \cite[Chapter 6]{Hof17}) resulting in a scale-free random graph. Further,
		\begin{equation}
		\E[W_n] = \frac{1}{n}	\sum_{i\in [n]}w_i \to \mu = \E[W],
		\end{equation}
and, for all $i\in[n]$,
	\eqn{
	\label{theta(i)-def}
	n^{-\alpha}w_i  = \cf i^{-\alpha},
	}
for some constant $\cf >0$. Throughout $\cf$ will denote the special constant appearing above.

\subsection{Results}
\label{sec-NR-res} 
We start by describing our results for the barely subcritical regime, then the critical window, and end with the super-critical regime. To explicitly describe limit constants, we will phrase the results with respect to the Norros-Reittu model deferring statements to other models to Theorem \ref{thm:extensions}. The phase transition is described in terms of functionals of the relevant {components}. Let $(|\sCi(\per)|)_{i\geq 1}$ be the component sizes of $\rNR(\bw,\per)$, arranged in non-increasing order (breaking ties arbitrarily). Further, let $(W_{\sss (i)}(\per))_{i\geq 1}$ denote the corresponding weight of these clusters, i.e.,
	\eqn{
	\label{weight-cluster}
	W_{\sss (i)}(\per)=\sum_{j\in \sCi(\per)} w_j.
	}
The phase transitions will be described in terms of these two functionals.  

\subsubsection{Behavior in the barely sub-critical regime} 
\label{sec-GRG-sub-res}
Recall the constants related to the degree exponent in  \eqref{eqn:notation-const}.
\begin{theorem}[Subcritical regime for  $\mathrm{NR}_n(\bw, \percn)$]
\label{thm:barely-subcrit-single-edge}
 Suppose that $\bw$ satisfies {\rm Assumption~\ref{assumption-NR}}, and consider $\rNR(\bw,\per)$ with  $\percn=\lambda_n n^{-\seta}$ with $\lambda_n =o(1)$, and $\percn \gg n^{-\alpha}$. 
Then, for any fixed $i\geq 1$, as $n\to\infty$, 
	\begin{equation}
	\frac{|\sCi(\percn)|}{n^{\alpha}\percn} \pto \cf i^{-\alpha},\quad \text{and}\quad \frac{W_{\sss (i)}(\percn)}{n^{\alpha}} \pto \cf i^{-\alpha}.
	\end{equation}
\end{theorem}
Theorem~\ref{thm:barely-subcrit-single-edge} implies that the largest percolation clusters  with $\percn \ll n^{-\seta}$ are the clusters of the hubs, i.e., the vertices with the largest weights ($w_i = \Theta(n^{\alpha})$). Further, the hubs with high probability lie in disjoint components. Since, after percolation, the number of neighbors of hub $i$ is close to $\percn w_i\approx n^{\alpha}\percn \cf i^{-\alpha}$, these largest clusters consist mostly of the hubs with their immediate neighbors. In particular, since the largest cluster sizes {\em concentrate}, we are not in the critical window when $\pi_n \ll n^{-\eta_s}$.

\subsubsection{Behavior in the critical window} 
\label{sec-NR-crit-res}
As discussed in the introduction, this critical window consists of $\percn=\lambda n^{-\seta}$ for some explicit bounded interval of~$\lambda$. In particular, such values of $\percn$ are much larger than the values considered in the previous section. We will see that there is a surprising phase transition in $\lambda$, occurring at a finite positive value $\lambda_c$. Below $\lambda_c$, the scaling limits of the largest connected components have {\em non-degenerate} scaling limits, and any two hubs are in the same component with asymptotic probabilities strictly bounded between 0 and 1. 
Recall~\eqref{eqn:notation-const}, and define
	\begin{equation}
	\label{eq:scaling-window}
	\perc =\percl:= \lambda n^{-\seta}, 
	\quad \text{for }\lambda \in (0,\lambda_c),
	\end{equation} 
where $\lambda_c$ is 
given by 
\begin{gather}
    \lambda_c := \sqrt{\frac{\eta}{4B_{\alpha}}}= \frac{\cf^{-1/\alpha}}{2}\sqrt{\frac{(3-\tau)\mu^{1/\alpha}}{A_\alpha}}, \label{eqn:lambc-def}\\
     A_\alpha:= \int_0^\infty \frac{1-\e^{-z}}{z^{1/\alpha}} dz, \quad B_\alpha:= \frac{\cf^{2/\alpha} A_\alpha}{\alpha\mu^{1/\alpha}}. \label{eqn:A-B-alpha-def}
\end{gather}
\begin{theorem}[Critical regime for  $\mathrm{NR}_n(\bw, \percn)$] 
\label{thm:main-crit}
Suppose that $\bw$ satisfies {\rm Assumption~\ref{assumption-NR}}, and consider $\rNR(\bw,\percn)$ with $\percn=\percl$ for $\lambda\in (0,\lambda_c)$, with $\lambda_c$ as in \eqref{eqn:lambc-def}.  Then, as $n\to \infty$,
	\begin{equation}
	(n^{\alpha}\perc)^{-1} (|\sCi(\percl)|)_{i\geq 1} \dto (\sW_{\sss (i)}^{\infty}(\lambda))_{i\geq 1} \quad \text{and} \quad n^{-\alpha} (W_{\sss (i)} (\percl))_{i\geq 1} \dto (\sW_{\sss (i)}^{\infty}(\lambda))_{i\geq 1} 
	\end{equation} 
with respect to the $\ell^2_{\shortarrow}$-topology and $\ell^2$-topology respectively. 
The limiting random variables $(\sW_{\sss \mathrm{(I)}}^{\infty}(\lambda))_{i\geq 1}$ are non-degenerate and described in {\rm Definition~\ref{defn:g-lambda}} below.
\end{theorem}
The non-degenerate scaling limit of the component sizes, as well as their weights is the hallmark of critical behavior.
To define the limiting variables in Theorem~\ref{thm:main-crit}, we need the following infinite weighted random graph which belongs to a general class of models studied by Durrett and Kesten in \cite{DK90}. 

\begin{defn}[Limiting variables] 
\label{defn:g-lambda} \normalfont
Fix vertex set $\Z_+$ and let vertex $i\in \Z_+$ have weight $\theta_i:=\cf i^{-\alpha}\mu^{-1}$. Consider the random multi-graph $\cGinf(\lambda)$ on $\Z_+$ where vertices $i$ and $j$ are joined independently by Poisson$(\lambda_{ij})$ many edges with $\lambda_{ij}$ given by 
	\begin{equation}\label{defn:lambda-ij}
	\lambda_{ij}:=\lambda^2\int_0^\infty \Theta_i(x) \Theta_j(x) \dif x, \quad \text{where}\quad
	\Theta_i(x):= 1-\e^{-\cf\theta_i x^{-\alpha}}.
	\end{equation}
For $i \geq 1$, let $\sW_{\sss (i)}^{\infty}(\lambda)$ denote the $i$-th largest element of the set 
    \[\bigg\{\sum_{i\in \sC}\theta_i\colon \sC \text{ is a connected component}\bigg\},
    \]
which is well-defined when $(\sW_{\sss (i)}^{\infty}(\lambda))_{i\geq 1}\in \ell^2_{\shortarrow}$ almost surely.
 \end{defn}

We will see that asymptotically there are Poisson($\lambda_{ij}$) many {\em two-step} paths between macro-hubs $i$ and $j$ via intermediete meso-scale hubs of size $\Theta(n^\rho)$ in $\rNR(\bw, \perc(\lambda))$, for $i,j$ fixed as $n\rightarrow \infty$. These two-step paths between hubs form the backbone of the largest connected components. The connectivity structure of these two-step connections undergoes a phase transition, as we next explain. The following result implies that the limiting object is well-defined for $\lambda \in (0, \lambda_c]$, and undergoes a phase transition at $\lambda=\lambda_c$:

\begin{proposition}[Phase transition for the limiting model]
\label{prop-limit-as-finite}
\begin{enumerate}[(a)]
	\item For $\lambda\leq \lambda_c$, $ (\sW_{\sss (i)}^{\infty}(\lambda))_{i\geq 1}$ is in $\ell^2_{\shortarrow}$ almost surely.
	\item For $\lambda > \lambda_c$, $\cGinf(\lambda)$ is connected almost surely, in particular $\sW_{\sss (1)}^{\infty}(\lambda)=\infty$ and $\sW_{\sss (2)}^{\infty}(\lambda)=0$ almost surely.
\end{enumerate} 
\end{proposition}

\subsubsection{Behavior in the supercritical regime} 
\label{sec-GRG-super-res}
Let us now consider percolation with probability $\percn = \lambda n^{-\seta}$ for $\lambda>\lambda_c$. 
Since $\cGinf(\lambda)$ represents the connectivity structure between the hubs, Proposition~\ref{prop-limit-as-finite}~(b) suggests that (1) the hubs are in the same component whp, (2) the largest connected component after $\lambda_c$ is much larger than the components before $\lambda_c$.
Our result next result shows that  in fact a \emph{unique} giant component of size $\sqrt{n}$ 
appears in the graph, and the size of this giant component concentrates.
Moreover, the giant component is {\em unique} in the sense that the second largest component is of a smaller order. 
To describe the limiting size of the giant component, fix $a>0$, and 
define 
\begin{eq}\label{defn:zeta-lambda}
\zeta_a^\lambda := \lambda\int_0^a  \cf u^{-\alpha} \rho_a^\lambda(u)\dif u,
\end{eq}
where $\rho_a^\lambda: (0,a] \mapsto [0,1]$ is the maximum solution to the fixed point equation 
\begin{eq}\label{defn:surv-prob}
\rho_a^\lambda(u)=1-\e^{-\lambda \int_0^a \kappa(u,v)\rho_a^\lambda(v)\dif v}, \quad \text{with }\quad  \kappa(u,v) := 1-\e^{-\cf^{2} (uv)^{-\alpha}/\mu}. 
\end{eq}
In Proposition~\ref{prop-large-a-original}, we will see that $\zeta^{\lambda} = \lim_{a\to\infty} \zeta_a^\lambda$ exists and $\zeta^\lambda \in (0,\infty)$, whenever $\lambda>\lambda_c$. 
We now state our result for the emergence of the giant component for $\lambda>\lambda_c$:
\begin{theorem}[$\sqrt{n}$-asymptotics of size and uniqueness giant]
\label{thm:supcrit-bd}
Suppose that $\bw$ satisfies {\rm Assumption~\ref{assumption-NR}}, and consider $\rNR(\bw,\percn)$ with $\percn=\lambda n^{-\seta}$ for some $\lambda>\lambda_c$. 
Then, as $n\to\infty$, 
\begin{eq}
n^{-1/2}|\sC_{\sss (1)}(\percn)|\pto \zeta^\lambda, \quad \text{and} \quad n^{-1/2}|\sC_{\sss (2)}(\percn)|\pto 0,
\end{eq}
where $\zeta^{\lambda} = \lim_{a\to\infty} \zeta_a^\lambda$ with $\zeta_a^\lambda$ given by \eqref{defn:zeta-lambda}. 
Further, $\{v\colon w_v\geq  n^{1/2+\delta}\}\subseteq \sC_{\sss (1)}(\percn)$ whp for every $\delta>0.$
\end{theorem}

\subsection{Discussion}
\label{sec:discussion}
In this section, we discuss some insights to our results, extensions and open problems.
\paragraph{Critical window for other rank-1 models. } 
Our results for the subcritical regime, and the critical window 
hold more generally for the Chung-Lu Model $\CL$ and the generalized random graph model $\GRG$ described in \eqref{CL-pij} and \eqref{GRG-pij}. To state this formally, define 
\begin{eq}\label{eq:A-CL-GRG}
A_{\alpha}^{\sss \mathrm{CL}} = \int_0^\infty \min\{1,z\}z^{-1/\alpha} \dif z, \quad A_{\alpha}^{\sss \mathrm{GRG}} = \int_0^\infty \frac{z^{1-1/\alpha}}{1+z} \dif z,
\end{eq}
and define $B_\alpha^{\sss \mathrm{CL}}$, $B_\alpha^{\sss \mathrm{GRG}}$, and the critical values $\lambda_c^{\sss \mathrm{CL}}$ and $\lambda_c^{\sss \mathrm{GRG}}$ identically as in \eqref{eqn:lambc-def} and  \eqref{eqn:A-B-alpha-def} with the above choices of $A_{\alpha}^{\sss \mathrm{CL}}$ and $A_{\alpha}^{\sss \mathrm{GRG}}$ respectively. 
To define the limiting object, let 
\begin{eq}\label{eq:theta-i-CL-GRG}
\Theta_i^{\sss \mathrm{CL}}(x) = \min\Big\{\frac{\cf^2i^{-\alpha}x^{-\alpha}}{\mu},1\Big\}, \quad \Theta_i^{\sss \mathrm{GRG}}(x)= \frac{\cf^2 i^{-\alpha}x^{-\alpha}}{\mu+\cf^2 i^{-\alpha}x^{-\alpha}}.
\end{eq}
Denote the graph $\rCL(\bw,p)$, $\rGRG(\bw,p)$  obtained by independently keeping each edge of the graph $\CL$ and $\GRG$ respectively. 
\begin{theorem}[Extensions to other rank-1 models] 
	\label{thm:extensions}
Under {\rm Assumption~\ref{assumption-NR}}, {\rm Theorems~\ref{thm:barely-subcrit-single-edge},~\ref{thm:main-crit}} hold for $\rCL(\bw,\percl)$ and $\rGRG(\bw,\percl)$ with $\lambda_c$ replaced by 
$\lambda_c^{\sss \mathrm{CL}}$ and $\lambda_c^{\sss \mathrm{GRG}}$ defined below \eqref{eq:A-CL-GRG} respectively, and the scaling limits given by {\rm Definition~\ref{defn:g-lambda}} with $\Theta_i(x)$ replaced by $\Theta_i^{\sss \mathrm{CL}}(x)$ and $ \Theta_i^{\sss \mathrm{GRG}}(x)$ defined in \eqref{eq:theta-i-CL-GRG}, respectively. 
\end{theorem}
The proof of Theorem~\ref{thm:extensions} only requires minor adaptations on the proofs of {\rm Theorems~\ref{thm:barely-subcrit-single-edge},~\ref{thm:main-crit}}. We point out the key modifications in  Remarks~\ref{rem:related-1},~\ref{rem:related-2} and skip redoing the whole proof for Theorem~\ref{thm:extensions}. We also believe that a result analogous to Theorem~\ref{thm:supcrit-bd} holds for the giant component in $\rCL(\bw,\percn)$ and $\rGRG(\bw,\percn)$ with
\begin{eq}\label{eq:kappa-CL-GRG}
\kappa^{\sss \mathrm{CL}}(u,v) := \min\Big\{\frac{\cf^2(uv)^{-\alpha}}{\mu},1\Big\}, \quad \kappa^{\sss \mathrm{GRG}}(u,v) := \frac{\cf^2 (uv)^{-\alpha}}{\mu+\cf^2 (uv)^{-\alpha}}. 
\end{eq}
However, since the proof of Theorem~\ref{thm:supcrit-bd} is extremely delicate, we leave this as an open question. 

\paragraph*{When does the single-edge constraint matter?} In concurrent works \cite{DH21,DHL19}, we study percolation  on scale-free networks around criticality for models that allow for multi-edges such as the configuration model~\cite{DHL19} and the Norros-Reittu model~\cite{DH21},  where in the latter model, the number of edges between vertices $i$ and $j$ is $\Poi(w_iw_j/\ell_n)$. It turns out that a giant emerges in these multi-edge models when 
\begin{eq}
	\label{def-pc-CM}
	\percn = \lambda_n n^{- (3-\tau)/(\tau-1)}, \quad \text{where } \lambda_n \to \infty.
\end{eq}
Thus, the emergence of giant happens in multi-edge models for much smaller $\percn$ values. Interestingly, when $\percn = \lambda n^{-\seta}$ and $\lambda>\lambda_c$, both the single-edge and multi-edge version of the Norros-Reittu model contain a giant component of size $\sqrt{n}$, but the description of their asymptotic sizes are vastly different. In fact, we  believe that the asymptotic proportions are strictly different although we do not prove it in this article. On the other hand, if $\percn = \lambda_n n^{-\seta}$ with $\lambda_n \to \infty$, then the giants in both the single and multi-edge Norros-Reittu model turn out to have the same size~\cite{DH21}. Such differences in multi-edge versus single-edge settings are  absent in the $\tau>3$ settings.

\paragraph*{Critical windows: emergence of hub connectivity.}
The critical window changes due to the single-edge constraint as noted in the previous paragraph.
However, there are some common features. First, the component sizes are of the order $n^{\alpha} \percl$ in both the regimes. 
This is due to the fact that the main contribution to the component sizes comes from hubs and their direct neighbors.
Second, in both cases, the critical window is the regime in which hubs start getting connected. 
More precisely, the critical window is given by those values of $\pi$ such that, for any fixed $i, j\geq 1$,
	\begin{eq}
	\liminf_{n\to\infty} \PR(i, j\text{ are in the same component in the }\pi \text{-percolated graph} ) \in (0,1).
	\end{eq}
For multi-edge models, hubs are connected {\em directly} with strictly positive probability, while under the single-edge constraint, 
hubs are connected with positive probability via intermediate vertices of degree $\Theta (n^{\rho})$. In the barely subcritical regime, instead, all the hubs are in different components.
Hubs start forming the critical components as connection probability $\pi$ varies over the critical window.
Finally in the barely super-critical regime the giant component is formed, and this giant contains all the hubs. 
This feature is also observed in the $\tau \in (3,4)$ case~\cite{BHL12}. 
However, the distinction between $\tau \in (3,4)$ and $\tau\in (2,3)$ is that, for $\tau \in (3,4)$, the paths between the hubs have lengths that grow with $n$, namely as $n^{(\tau-3)/(\tau-1)}$.

\paragraph*{Open problems.} 
We believe that the results proved in this paper are universal for percolation problems on a host of scale-free random graph models. For example, we believe that our results carry over to the setting of {\em uniform} random graphs with a given degree distribution, for which the probability that hubs $i$ and $j$ are connected is close to $d_id_j/(\ell_n+d_id_j)$ (see, e.g., \cite{GHSS18}). 
Further, we believe that similar results apply to {\em site} percolation on scale-free random graphs, irrespective of whether the model has a single-edge constraint or not. In fact, for site percolation, the clusters for the erased and normal configuration models are identical, so that also their scaling behavior is expected to be identical.

The previous discussion suggests that the typical distances in large critical components are quite small, and it would be of interest to describe their distributions in more detail. Further, it would be of interest to derive the scaling limit of the {\em diameter} of the large critical components. 
Finally, we show that the critical window in the single-edge case is $\percn= \lambda n^{-(3-\tau)/2}$ for $\lambda\in (0,\lambda_c)$, which does not include the critical point $\lambda_c$. However, $\lambda_c$ {\em is} included in the critical case for the limiting graph in Proposition \ref{prop-limit-as-finite}. This raises the question what happens for $\percn=n^{-(3-\tau)/2}\lambda_c(1+\vep_n)$ for $\vep_n=o(1)$. Is barely supercritical behavior then observed, or does a second type of critical behavior emerge? We leave this as an interesting open question.

\subsection{Proof outline}
\paragraph*{The critical window (Section \ref{sec:proof-GRG}).}
\begin{figure}
    \centering
    \includegraphics[scale = 0.045]{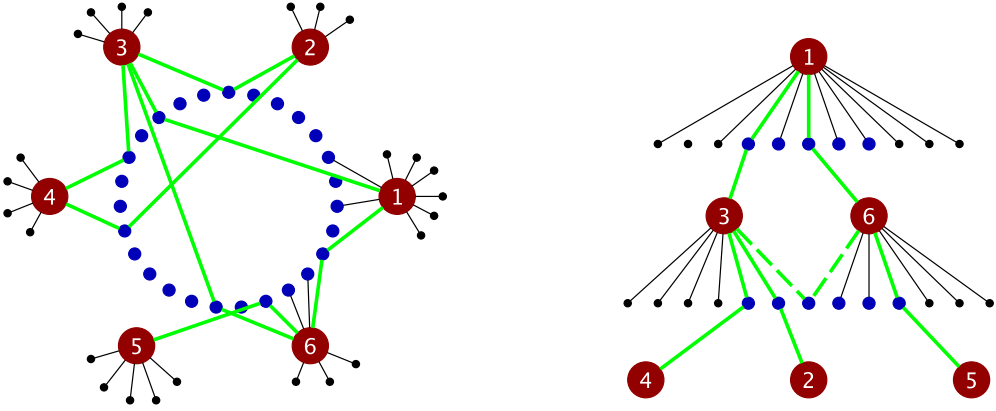}
    \caption{Visualization of the core structure of components and the exploration of the neighborhood. Red vertices indicate hubs with $w_i = \Theta(n^{\alpha})$, and blue vertices having $w_i = \Theta(n^{\rho})$ are intermediate vertices that connect hubs via~two-step paths (indicated by green edges).
    }
    \label{fig:exploration}
\end{figure} 
The key idea is that the largest critical components  correspond to \emph{connected components} containing macro-hubs (maximal weight vertices).
However since $\pi_n\to 0$ any two macro-hubs cannot be directly connected in the large network limit, rather these have non-trivial probability of being  connected via~a two-step path passing through meso-scale intermediete hubs of weight $w_i = \Theta (n^{\rho})$. 
 In fact, we can couple the hubs and these two-step connections to the infinite graph $\cGinf (\lambda)$ as in Definition \ref{defn:g-lambda} in total variation distance (see Proposition~\ref{prop:hub-connection} below).
Next, we show that the primary contribution to the component sizes come from the one-neighborhood of the subgraph consisting of hubs and their two-step connections. 
This is reflected in the fact that, when we  explore the graph starting from hubs in a breadth-first manner, we see an alternating structure with the hubs appearing in the even generations, and the odd generations consisting of vertices having weight $\Theta (n^{\alpha})$, see Figure~\ref{fig:exploration} and Proposition~\ref{prop:tail-particular-component} below. The main technique here is to use appropriate path-counting techniques (see Proposition~\ref{lem:geometric-decay} below).
Finally, we conclude the proof of Theorem~\ref{thm:main-crit} by showing that the vector of component sizes is tight in $\ell^2_{\shortarrow}$ when $\lambda<\lambda_c$ (see Proposition~\ref{prop-tight-l2}). 
The phase transition at $\lambda=\lambda_c$ is exemplified in Proposition~\ref{prop-limit-as-finite}, as   $\cGinf (\lambda)$ becomes connected for $\lambda>\lambda_c$.

\paragraph*{Supercritical regime (Sections~\ref{sec:small-giant} and \ref{sec-size-tiny-giant}).} 
The key observation is that the core of the giant component can be identified by looking at a special set of vertices $V$ consisting of vertices with $w_i = \Omega(\sqrt{n})$. Note that these vertices are  present only in the $\tau\in (2,3)$ regime (for $\tau>3$, the maximum weight is $o(\sqrt{n})$). 
Now, the subgraph restricted to $V$ is an inhomogeneous random graph with kernel approximately equal to $\kappa$ given by \eqref{defn:surv-prob}.
Using  general results from inhomogeneous random graphs \cite{BJR07}, this allows us to conclude that the graph restricted to $V$ exhibits a {\em phase transition}, and a unique giant component of approximate size $|V| \zeta^{\lambda}$  appears for some $\lambda>\lambda_1$, where $\lambda_1$ is given by the inverse of norm of a suitable integral operator. Thus, a \emph{giant component} appears inside $V$ precisely after $\lambda_1$. 
This constitutes the core of vertices, and the 1-neighborhood of this tiny giant spans almost the entire giant component. 
The quantity $\zeta^\lambda$ in \eqref{defn:zeta-lambda} should be interpreted as the size of the 1-neighborhood of the small giant. 
Thus, we see two structural transitions occurring at $\lambda = \lambda_c$ and at $\lambda = \lambda_1$. 
These values have rather different origins, namely $\lambda_c$ arises as the connectivity threshold for an inhomogeneous percolation on the integers in Proposition~\ref{prop-limit-as-finite},
and $\lambda_1$ as the critical value of an appropriate inhomogeneous random graph, 
described in terms of an operator of some branching process.
However, an explicit computation shows that, in fact, $\lambda_c = \lambda_1$, see Lemma \ref{lem-lambdacs-equal} below.

\section{Proofs for subcritical and critical regimes}
\label{sec:proof-GRG}
We begin with the proof of the critical regime,  starting in Section~\ref{sec:limit-object-finite} by proving  Proposition~\ref{prop-limit-as-finite} and in particular  showing that the asserted limiting object is finite.
In Section~\ref{sec-hub-conn}, we set up technical ingredients to study the connectivity structure between macro-hubs.
In Section~\ref{sec:key-ingredients}, we derive path-counting estimates, which are used in Section~\ref{sec:negligible-contributions} to show that if we start exploring the components containing hubs, then the total number of vertices at {\em even} distances is negligible, and the total number of vertices at \emph{large and odd distances is also} negligible (the same estimates will be useful in the sub-critical regime, which explains why we start with the critical regime first). 
This allows us to compute the size of the components containing hubs in Section~\ref{sec:component-with-hubs}. We  conclude the proof of Theorem~\ref{thm:main-crit} in Section~\ref{sec:l2-tightness} by showing that the vector of component sizes is tight in $\ell^2_{\shortarrow}$.  The subcritical regime is analyzed in Section \ref{sec-proof-barely-sub-GRG} where using path counting techniques we show that the largest components are essentially stars with hubs as centers.

\subsection{Finiteness of the limiting object for \texorpdfstring{$\lambda\leq \lambda_c$}{TEXT}: proof of Proposition \ref{prop-limit-as-finite}}
\label{sec:limit-object-finite}
Recall $\lambda_c$ from \eqref{eqn:lambc-def}, and the  constants $A_\alpha,B_\alpha$ from \eqref{eqn:A-B-alpha-def}.  
Define the symmetric function $h:(0,\infty)^2\to(0,\infty)$ by 
	\begin{equation}
	\label{eqn:hij-def}
	h(x,y):= { B_\alpha}{(x\wedge y)^{-(1-\alpha)}(x\vee y)^{-\alpha}}.
	\end{equation} 
Note that $h(i,j)$ is perfectly homogeneous of exponent $-1$, i.e., $h(tx,ty) = t^{-1}h(x,y)$ for all $x,y,t>0$.
Analogous to $\cGinf(\lambda)$ in Definition \ref{defn:g-lambda}, consider the following random graph which belongs to a general class of models studied by Durrett and Kesten~\cite{DK90}:
\begin{defn}[Inhomogeneous percolation model] \normalfont \label{defn:DK-model}
Consider the random graph $\cG_{\sss \mathrm{DK}}(\lambda)$ on $\Z_+$ where vertices $i$ and $j$ are joined with probability $\min\{\lambda^2 h(i,j), 1\}$, independently across edges.  
\end{defn}
\begin{theorem}[Previous results for $\cG_{\sss \mathrm{DK}}(\lambda)$, \cite{DK90,zhang1991power}]
\label{thm:DK-known-results}
	\begin{enumerate}[(a)]
		\item By \cite{DK90}, the random graph $\cG_{\sss \mathrm{DK}}(\lambda)$ is connected almost surely for $\lambda > \lambda_c$. 
		\item 
		For $i,j\in \bZ_+$, write $\PR_{\lambda, {\sss \mathrm{DK}}}(i\leftrightarrow j)$ for the probability that $i,j$ are connected by some path in $\cG_{\sss \mathrm{DK}}(\lambda)$. 
		By \cite{zhang1991power}, there exists $c_1<\infty$ such that for $\lambda=\lambda_c$, and for any $1\leq i<j$,
		\[\PR_{\lambda_c, {\sss \mathrm{DK}}}(i\leftrightarrow j) \leq c_1 \log(i\vee 2)/\sqrt{ij}.  \]   
	\end{enumerate}
\end{theorem}
\noindent To see that the $\lambda_c$ in \eqref{eqn:lambc-def} gives the same 
critical value as \cite[(1.5)]{DK90}, we compute 
\begin{eq}\label{crit-vallue-deduction}
\bigg[\int_0^\infty \frac{h(1,y)}{\sqrt{y}}\dif y\bigg]^{-1} = \frac{1}{B_\alpha} \bigg[\int_0^1\frac{\dif y}{y^{\frac{1}{2}+1-\alpha}}+\int_1^\infty\frac{\dif y}{y^{\frac{1}{2}+\alpha}}\bigg]^{-1} = \frac{2\alpha -1}{4B_\alpha}  = \frac{\eta}{4B_\alpha}, 
\end{eq}
where we have used that $\alpha\in (1/2,1)$. 
The square root in \eqref{eqn:lambc-def} is due to the fact that we have used $\lambda^2$ in Definition~\ref{defn:DK-model} instead of $\lambda$ as in \cite{DK90}. The factor $\lambda^2$ arises for us, since we deal with two-step paths. We next discuss an extension where the connection probabilities are {\em asymptotically} equal to $h(i,j)$ also proven  in {\cite[Extension (a)]{DK90}}.

\begin{corollary}[Extension to asymptotic edge probabilities]\label{rem:extentsion-DK}
Consider the graph $\cG'_{\sss \mathrm{DK}}(\lambda)$ constructed by keeping an edge between $i$ and $j$ independently with probability $r(i,j)$, and let   
\begin{eq}
\lim_{i\to\infty}\lim_{j\to\infty}{r(i,j)}/{ \lambda^2h(i,j)} = 1.
\end{eq}
 Then, $\cG'_{\sss \mathrm{DK}}(\lambda)$ is connected almost surely if $\lambda > \lambda_c$. 
\end{corollary}
We next state the following lemma which allows us to compare the connection probabilities in Definitions~\ref{defn:g-lambda}~and~\ref{defn:DK-model}. 
Let $p_{\infty}(i,j):=1-\e^{-\lambda_{ij}}$, with $\lambda_{ij}$ as in \eqref{defn:lambda-ij}, be the probability that there is an edge between $i,j$ in $\cGinf (\lambda)$ in Definition~\ref{defn:g-lambda}.
\begin{lemma}[Asymptotics of two-step probabilities]
\label{lem:pij-h-rel}
For all $i,j\in \Z_+$, $p_{\infty}(i,j)\leq \lambda^2 h(i,j)$. Further,
	\eqn{ \label{lambda-ij-ratio}
	\
	\lim_{i\to\infty}\lim_{j\to \infty}{p_{\infty}(i,j)}/{\lambda ^2 h(i,j)} =1.
	}
Consequently, $\cGinf(\lambda)$ is almost surely connected for $\lambda>\lambda_c$. 
\end{lemma}
\begin{proof} Without loss of generality, let $i<j$.
We first show the first assertion on domination. 
Using $1-\e^{-x} \leq x$ for all $x>0$ twice, as well as \eqref{defn:lambda-ij}, we note that 
\begin{eq}\label{ub-rij}
	p_{\infty}(i,j) \leq \lambda_{ij} = \lambda^2 \int_0^\infty \Theta_i(x)\Theta_j(x)\dif x\leq \frac{\lambda^2 \cf^2j^{-\alpha} }{\mu}  \int_0^\infty \Big(1- \e^{-\frac{\cf^2 }{\mu} i^{-\alpha} x^{-\alpha}}\Big) x^{-\alpha}\dif x. 
\end{eq}
Substituting $z = (i^{-\alpha}x^{-\alpha}) \cf^2/\mu$ with $x = \cf^{2/\alpha} \mu^{-1/\alpha} z^{-1/\alpha} i^{-1}$ and $\dif z = -\alpha \cf^2\mu^{-1} i^{-\alpha} x^{-\alpha-1} \dif x$,  
\begin{eq}
	p_{\infty}(i,j) \leq \frac{\lambda^2}{j^\alpha} \int_0^\infty (1-\e^{-z}) \frac{\cf^{2/\alpha}}{\mu^{1/\alpha}} \frac{\dif z}{\alpha i^{1-\alpha} z^{1/\alpha}} = \frac{\lambda^2\cf^{2/\alpha}}{\alpha\mu^{1/\alpha}} \frac{1}{i^{1-\alpha} j^\alpha} \int_0^\infty \frac{1-\e^{-z}}{z^{1/\alpha}} \dif z, 
\end{eq}
and thus using \eqref{eqn:A-B-alpha-def}, it follows that $p_{\infty}(i,j)\leq \lambda^2 h(i,j)$. 
For the second assertion, note that  $\lim_{i\to\infty}\lim_{j\to\infty}\lambda_{ij} = 0$. Thus, we can use the same calculation as above, together with the fact that $\lim_{x\to 0}[1-\e^{-x}]/x = 1$ to conclude \eqref{lambda-ij-ratio}.
\end{proof}
\begin{remark}[Related rank-one models] \label{rem:related-1}
\normalfont
If we replace $\Theta_i(x)$ by $\Theta_i^{\sss \mathrm{CL}}(x)$ and $\Theta_i^{\sss \mathrm{GRG}}(x)$ from \eqref{eq:theta-i-CL-GRG} respectively, then 
Lemma~\ref{lem:pij-h-rel} holds with $h(i,j)$ replaced by $h^{\sss \mathrm{CL}}(i,j) = (B_{\alpha}^{\sss \mathrm{CL}}/B_\alpha )h(x,y)$ and $h^{\sss \mathrm{GRG}}(i,j) = (B_{\alpha}^{\sss \mathrm{GRG}}/B_\alpha )$ respectively from \eqref{eq:theta-i-CL-GRG}. 
\end{remark}

\begin{proof}[Proof of Proposition~\ref{prop-limit-as-finite}] 
Recall that $\theta_i = \cf i^{-\alpha} \mu^{-1}$. 
Using Corollary~\ref{rem:extentsion-DK}, together with Theorem~\ref{thm:DK-known-results} and Lemma \ref{lem:pij-h-rel}, the graph $\cGinf(\lambda)$ is almost surely connected for $\lambda>\lambda_c$.
Thus, Proposition~\ref{prop-limit-as-finite}~(b) follows from Theorem \ref{thm:DK-known-results}~(a) and the fact that $\sum_{i=1}^{\infty} \theta_i =\infty$. 
Next, by the upper bound in Lemma~\ref{lem:pij-h-rel} and monotonicity in $\lambda$ for the connection probabilities, it is enough to show Proposition~\ref{prop-limit-as-finite}~(a) for $\cG_{\sss \mathrm{DK}}(\lambda_c)$.
To this end, let $\cC(j)$ denote the component of vertex~$j$ in $\cG_{\sss \mathrm{DK}}(\lambda_c)$. Define 
	\begin{equation}
	\label{C-leq-def}
        \cC_{\sss \leq}(j) = \begin{cases}
        \cC(j) &\text{if }j= \min\{i\colon i\in \cC(j)\},\\
        \varnothing & \text{otherwise.}
        \end{cases}
	\end{equation}
Then it is enough to show that
	\begin{equation}
	\label{eqn:ldown-ents}
	L:=	\E\bigg(\sum_{j=1}^\infty \Big(\sum_{i\in \cC_{\sss \leq}(j)}\theta_i\Big)^2\bigg)< \infty.
	\end{equation}
Expanding the above, we obtain
\begin{eq}
L\leq \sum_{j=1}^\infty\theta_j^2 +2\sum_{i_1> i_2\geq j}\theta_{i_1}\theta_{i_2} 
\PR_{\lambda_c, {\sss \mathrm{DK}}}(i_1\leftrightarrow i_2 \text{ in  } [j,\infty), i_1, i_2\in \cC(j)):=||\mvtheta||_2 +2L_1. 
\end{eq}
Here the event $\{i_1\leftrightarrow i_2 \text{ in  } [j,\infty)\}$ is the event that there exists a path in $\cG_{\sss \mathrm{DK}}(\lambda_c)$ from  $i_1$ and $i_2$ with intermediate vertices in  $[j,\infty)$. 
Since $\mvtheta\in \ell^2_{\shortarrow}$, it is enough to show that $L_1<\infty$. 
Splitting into cases depending on whether $i_2=j$ or $i_2>j$, we get $L_1=L_2+L_3$, where
	\begin{equation}
	\label{eqn:fe3}
	L_2:=\sum_{j=1}^\infty \sum_{i_1>j} \theta_{i_1} \theta_j\PR_{\lambda_c, {\sss \mathrm{DK}}}(i_1\leftrightarrow j \text{ in  } [j,\infty))\leq \sum_{j=1}^{\infty} \frac{1}{j^{\alpha}}\sum_{i>j} \frac{1}{i^{\alpha}}	
	\frac{c_1\log{j}}{\sqrt{ij}}<\infty,
	\end{equation}
where the second inequality follows from Theorem~\ref{thm:DK-known-results}~(b) and the last inequality uses $\alpha> \tfrac{1}{2}$. The final term to bound is $L_3$.
For any $i_1, i_2 > j$ write $\set{i_1\leftrightarrow i_2}_j$ for the event $\set{i_1\leftrightarrow i_2 \text{ in  } [j,\infty)}$ in $\cG_{\sss \mathrm{DK}}(\lambda_c)$.
Next note that 
	\eqn{
	\{i_1\leftrightarrow i_2 \text{ in  } [j,\infty), i_1, i_2\in \cC(j)\}
	\subseteq \bigcup_{z\geq j} \big[\set{z\leftrightarrow j}_j\circ \set{i_1\leftrightarrow z}_j \circ \set{i_2\leftrightarrow z}_j\big],
	}
where $\set{z\leftrightarrow j}_j\circ \set{i_1\leftrightarrow z}_j \circ \set{i_2\leftrightarrow z}_j$ denotes the event that the implied connections are realized using {\em disjoint} sets of edges.
The union bound combined with the BK-inequality \cite[Theorem 3.3]{BK85} implies that, for fixed $i_1>i_2>j$,
	\begin{eq}
	\label{eqn:bk}
	&\PR_{\lambda_c, {\sss \mathrm{DK}}}(i_1\leftrightarrow i_2 \text{ in  } [j,\infty), i_1, i_2\in \cC(j))\\
	&\quad \leq \sum_{z\geq j} \PR_{\lambda_c,{\sss \mathrm{DK}}}(\set{j\leftrightarrow z}_j) \PR_{\lambda_c,	
	{\sss \mathrm{DK}}}(\set{z\leftrightarrow i_1}_j) \PR_{\lambda_c,{\sss \mathrm{DK}}}(\set{z\leftrightarrow i_2}_j)\\
	&\quad \leq \sum_{z\geq j} c_1^3 \frac{\log(j\vee 2)\log(i_1\vee z)\log(i_2\vee z)}{\sqrt{i_1i_2j^3z^3}}\leq C\frac{\log(j\vee 2)\log(i_1\vee 2)\log(i_2\vee 2)}{\sqrt{i_1i_2j^3}},
	\end{eq}
where we have once again used Theorem \ref{thm:DK-known-results}(b) for the final inequality. Thus, 
	\eqn{
	L_3\leq C \sum_{j=1}^\infty \sum_{i_1>i_2>j} \theta_{i_1} \theta_{i_2} \frac{\log(j\vee 2)\log(i_1\vee 2)\log(i_2\vee 2)}{\sqrt{i_1i_2j^3}}.
	}
Together with the statement that $\theta_i=\cf i^{-\alpha}\mu^{-1}$ with $\alpha >\tfrac{1}{2}$, we obtain that $L_3 <\infty$. This completes the proof of \eqref{eqn:ldown-ents} and hence Proposition~\ref{prop-limit-as-finite}~(a).
\end{proof}

\subsection{Connectivity structure between hubs}
\label{sec-hub-conn} 
In this section, we estimate the connection probabilities between macro-hubs. 
Recall $p_{ij}$ from~\eqref{eq:p-ij-NR-defn}. 
Henceforth, in this section we simply write $\perc$ for $\percl$.
For any $i\neq j$, let $X_{ij}$ denote the number of paths of length 2 from $i$ to $j$. 
For $v\notin \{i,j\}$, let $\xi_{ij}(v)$ denote the indicator that $\{i,v\}$ and $\{v,j\}$ create edges in $\NRp$.
Thus,
\begin{eq}\label{eq:distn-Xij}
X_{ij} = \sum_{v\neq i,j} \xi_{ij}(v), \quad \text{with } \xi_{ij}(v) \sim \ber\big(\perc^2p_{iv}p_{vj} \big),  \text{ independently.}
\end{eq}
\begin{proposition}[Hub connectivity]\label{prop:hub-connection} 
For each fixed $i,j\geq 1$, 
\begin{eq}\label{eq:X-ij-asymp}
\lim_{n\to\infty}\dTV(X_{ij},P_{ij}) = 0, \quad \text{where }P_{ij} \sim \poi(\lambda_{ij}),
\end{eq}and $\dTV(\cdot,\cdot)$ denotes the total variation distance. 
Moreover, for any fixed $K\geq 1$, $(X_{ij})_{1\leq i<j\leq K}$ are asymptotically independent. 
\end{proposition}
Before embarking on the proof of Proposition~\ref{prop:hub-connection}, we describe moment estimates on the weights  $\bw$. Recall that $\ell_n = \sum_{i\in [n]} w_i$, and $a_n \asymp b_n$ denotes that $a_n=b_n(1+o(1))$.

\begin{lemma}[Moment estimates]\label{lem:order-estimates}
Under {\rm Assumption~\ref{assumption-NR}},
for any fixed $a>0$, 
\begin{eq}
\#\{k: w_k \geq a\ell_n/w_i\} &\asymp  n \bigg(\frac{\cf w_i}{a \ell_n}\bigg)^{\tau -1},\qquad \sum_{w_k>a\ell_n/w_i}w_k \asymp \frac{\cf^{\tau-1} n}{1-\alpha}\bigg(\frac{w_i}{a\ell_n}\bigg)^{\tau-2}, \\
&\sum_{k: w_{k}\leq a\ell_n/w_{i}}w_k^2 \asymp \frac{\cf^{\tau-1} n}{2\alpha-1}\bigg(\frac{a\ell_n}{w_i}\bigg)^{3-\tau}, 
\end{eq} 
where the approximations are uniform over $i\in[n]$.
\end{lemma}
\begin{proof}
The first approximation follows from  \eqref{theta(i)-def} by noting that 
\begin{eq}\label{eq:larger-than-elln-order}
w_k \geq \frac{a\ell_n}{w_i} \iff \cf\bigg(\frac{ n}{k}\bigg)^{\alpha} \geq   \frac{a\ell_n}{w_i} \iff k\leq   n \bigg(\frac{\cf w_i}{a\ell_n}\bigg)^{\tau -1}.
\end{eq}
Moreover, 
\begin{eq}
\sum_{k:w_k>a\ell_n/w_i}w_k &= \cf n^\alpha \sum_{k<  n(\cf w_i/a\ell_n)^{\tau-1}} k^{-\alpha} \\
&\asymp \frac{\cf n^{\alpha}}{1-\alpha} \bigg( n\Big(\frac{\cf w_i}{a
\ell_n}\Big)^{\tau-1}\bigg)^{1-\alpha} \asymp  \frac{\cf^{\tau-1}  n}{1-\alpha}\bigg(\frac{w_i}{a\ell_n}\bigg)^{\tau-2},
\end{eq}
and
\begin{eq}
\sum_{k: w_{k}\leq a\ell_n/w_{i}}w_k^2 &= \cf^{2}n^{2\alpha} \sum_{k\geq  n(\cf w_i/a \ell_n)^{\tau-1}} k^{-2\alpha} 
\\&
\asymp\frac{\cf^{2} n^{2\alpha}}{2\alpha-1} \bigg(\frac{a\ell_n}{\cf w_i}\bigg)^{3-\tau}  n^{1-2\alpha}\asymp \frac{\cf^{\tau-1} n}{2\alpha-1}\bigg(\frac{a\ell_n}{w_i}\bigg)^{3-\tau},
\end{eq}
where the approximations are uniform over $i\in [n]$.
Thus, the proof follows. 
\end{proof}

\begin{proof}[Proof of Poisson approximation in  Proposition~\ref{prop:hub-connection}]
We first prove the Poisson approximation in \eqref{eq:X-ij-asymp}, followed by the asserted asymptotic independence. Fix $\delta>0$. Recall $\rho = 1-\alpha$.
We start by splitting the sum in \eqref{eq:distn-Xij} over three sets $\{v: w_v < \delta n^{\rho}\}$, $\{v: \delta n^{\rho}\leq w_v \leq \delta^{-1} n^{\rho}\}$ and $\{v: w_v > \delta^{-1} n^{\rho}\}$. 
Let us denote these three partial sums by $X_{ij}^{\sss \rm\mathrm{(I)}}(\delta)$, $X_{ij}^{\sss \rm\mathrm{(II)}}(\delta)$ and $X_{ij}^{\sss \rm\mathrm{(III)}}(\delta)$ respectively. 
Now, using Lemma~\ref{lem:order-estimates}, 
\begin{eq} \label{hub-connection-term-1}
\E[X_{ij}^{\sss \rm\mathrm{(I)}}(\delta)] &\leq \frac{w_iw_j\perc^2}{\ell_n^2} \sum_{v:w_v<\delta n^{\rho}}w_v^2 \leq C \delta^{3-\tau}n^{2\alpha - 2+1+(3-\tau)\rho}\perc^2 \leq C\delta^{3-\tau}, \\ 
\E[X_{ij}^{\sss \rm\mathrm{(III)}}(\delta)] &\leq \perc^2 \times \#\{v:w_v>\delta^{-1}n^{\rho}\} =C \delta^{\tau-1}.
\end{eq}
For non-negative integer-valued random variables $X,Y,Z$, with $X,Y$ being independent, by the triangle inequality,
\begin{align}
&\dTV(X+Y, Z)  = \sum_{k = 0}^\infty \big|\PR(X+Y = k) - \PR(Z=k)\big| \notag\\
& \leq \sum_{k = 0}^\infty \big|\PR(X = k) - \PR(Z=k)\big| +\sum_{k = 0}^\infty \big|\PR(X = k) - \PR(X=k,Y=0)\big| +\sum_{k = 0}^\infty  \PR(X=k,Y\geq 1)  \notag\\
& \leq \dTV(X,Z) + 2\PR(Y\geq 1) \leq  \dTV(X,Z) + 2\E[Y]\label{eq:total-var-bound},
\end{align}
where the last step uses Markov's inequality. 
Using \eqref{hub-connection-term-1}~and~\eqref{eq:total-var-bound}, in order to prove \eqref{eq:X-ij-asymp}, it suffices to show that
\begin{eq}\label{hub-connection-term-2}
\lim_{\delta\to 0}\lim_{n\to\infty}\dTV(X_{ij}^{\sss \rm\mathrm{(II)}}(\delta),P_{ij}) = 0, \quad \text{where}\quad P_{ij} \sim \poi(\lambda_{ij}).
\end{eq}
Define 
\begin{eq}\label{hub-connection-term-poisson}
P_{ij}^{\sss(n)}(\delta) \sim \poi(\lambda_{ij}^{\sss(n)}(\delta)), \quad \text{where }\lambda_{ij}^{\sss(n)}(\delta) = \sum_{v: w_v\in [\delta n^{\rho},\delta^{-1}n^{\rho}]} \perc^2 p_{iv}p_{vj}.
\end{eq}
Using standard inequalities from Stein's method \cite[Theorem 2.10]{RGCN1}, it follows that, as $n\to\infty$,
\begin{eq}\label{pois-total-var-stein}
&\dTV\big(X_{ij}^{\sss \rm\mathrm{(II)}}(\delta),P_{ij}^{\sss(n)}(\delta))\big) \leq \sum_{v: w_v\in [\delta n^{\rho},\delta^{-1}n^{\rho}]} \big(\perc^2 p_{iv}p_{vj}\big)^2
\leq Cn^{4\alpha - 4}\perc^4 \sum_{v: w_v\in [\delta n^{\rho},\delta^{-1}n^{\rho}]}w_v^4
 \\
 &\leq \frac{C}{\delta^2}n^{2\alpha - 2}\perc^4\sum_{v: w_v\in [\delta n^{\rho},\delta^{-1}n^{\rho}]}w_v^2= \frac{C}{\delta^2} n^{2\alpha-2}n^{-2(3-\tau)}n^{1+ (3-\tau) \rho} = \frac{Cn^{-(3-\tau)}}{\delta^2} \to 0.
\end{eq}
Further, 
\begin{eq}\label{eq:stein-mean}
	\lambda_{ij}^{\sss(n)}(\delta) 
	&\asymp \perc^2  \sum_{k=\delta^{\tau-1}n^{3-\tau}}^{\delta^{-(\tau-1)}n^{3-\tau}}  \Big(1- \e^{-\frac{\cf^2}{\mu} n^{\frac{3-\tau}{\tau -1}} i^{-\alpha} k^{-\alpha}}\Big)\Big(1- \e^{-\frac{\cf^2}{\mu} n^{\frac{3-\tau}{\tau -1}} j^{-\alpha} k^{-\alpha}}\Big) \\
	&  \asymp \perc^2  \int_{\delta^{\tau-1}n^{3-\tau}}^{\delta^{-(\tau-1)}n^{3-\tau}} \Big(1- \e^{-\frac{\cf^2}{\mu} n^{\frac{3-\tau}{\tau -1}} i^{-\alpha} y^{-\alpha}}\Big)\Big(1- \e^{-\frac{\cf^2}{\mu} n^{\frac{3-\tau}{\tau -1}} j^{-\alpha} y^{-\alpha}}\Big)\dif y\\
	& \asymp\lambda^2\int_{\delta^{\tau-1}}^{\delta^{-(\tau-1)}} \Big(1- \e^{-\frac{\cf^2}{\mu}  i^{-\alpha} x^{-\alpha}}\Big)\Big(1- \e^{-\frac{\cf^2}{\mu} j^{-\alpha} x^{-\alpha}}\Big) \dif x:= \lambda_{ij}(\delta).
	\end{eq}
As $\delta\to 0$, we have $\lambda_{ij}(\delta) \to \lambda_{ij}$.  Since the total variation distance between two Poisson distributions is at most the difference of their means, we conclude  \eqref{hub-connection-term-2}, and hence the proof of \eqref{eq:X-ij-asymp} also follows.
\end{proof}

\begin{remark}[No hubs connected via two hop paths in subcritical regime]\normalfont 
When $\percn = \lambda_n n^{-\seta}$ with $\lambda_n = o(1)$, we can use identical argument as above to show that for any $K\geq 1$,
\begin{eq}\label{eq:sub-hub-con}
\lim_{n\to\infty}\PR(X_{ij} \geq 1\text{ for some }1\leq i<j\leq K) =0.
\end{eq}
Indeed, the bounds in \eqref{hub-connection-term-1},  \eqref{pois-total-var-stein} and \eqref{eq:stein-mean} would all tend to zero as $n\to\infty$.
\end{remark}

\begin{proof}[Proof of asymptotic independence in  Proposition~\ref{prop:hub-connection}]
Fix $K\geq 1$. 
Note that for pairs $(i,j)$, and $(k,l)$ with $\{i,j\}\cap \{k,l\} = \varnothing$, $X_{ij}$ and $X_{kl}$ are independent due to the independence of the occupancy of edges in $\NRp$.
The only dependence between $X_{ij}$ and $X_{ik}$ arises due to potential connections $(i,v)$, $(v,j)$ and $(v,k)$. To simplify notation we give a full proof for the asymptotic independence of $(X_{12}, X_{13})$, and a minor adaptation of this proof holds for any general $K\geq 1$.  Fix $\delta >0$ and let $V_n(\delta) = \{v\colon \delta n^{\rho}\leq w_v \leq \delta^{-1} n^{\rho}\}$. Let $X_{12}^{\sss \rm\mathrm{(II)}}(\delta), X_{13}^{\sss \rm\mathrm{(II)}}(\delta)$ be the random variables as in \eqref{hub-connection-term-2}. Recall the definition of the constant $\lambda_{ij}(\delta)$ from \eqref{eq:stein-mean}.  Arguing as in the convergence of the marginals, it is enough to prove that as $n\to\infty$
\begin{eq}\label{eq:joint-poisson}
\dTV[(X_{12}^{\sss \rm\mathrm{(II)}}(\delta), X_{13}^{\sss \rm\mathrm{(II)}}(\delta)), (P_{12}(\delta), P_{13}(\delta)) ]\to 0,
\end{eq}
where $P_{12}(\delta), P_{13}(\delta)$ are  independent Poisson random variables with means $\lambda_{12}(\delta), \lambda_{13}(\delta)$ respectively. 

We need some additional  notation to prove this proposition. For $v\in V_n(\delta)$ and for $i\in \set{1,2,3}$, let $I_{iv}$ be the indicator representing presence of edge $\set{i,v}$ in $\NRp$, so that the two hop indicator equals $\xi_{ij}(v) = I_{iv} I_{jv}$. Fix two constants $p,q\in [0,1]$ and for each $v\in V_n(\delta)$, let $J_{2v}, {J}_{3v}$ be Bernoulli $p,q$ random variables, respectively independent of each other and all the other indicator random variables. Here the constants depend on $p,q$.
Write 
\begin{equation}
	R_n:= \sum_{v\in V_n(\delta)} [J_{2v} I_{1v}I_{2v} + {J}_{3v} I_{1v}I_{3v}] = \sum_{\beta \in \cI} \indicwo{\beta},
\end{equation}
where the index set $\cI$ is given by $\cI = \cup_{v\in V_n(\delta)}\set{(v,1,2), (v,1,3)}$ and $\indicwo{\beta} = J_{kv}I_{1 v}I_{kv}$ for $\beta =(v, 1, k)$. 
Our main tool is the Poisson Cram\'er-Wold device in \cite[Corollary 2.2]{AHH16}, which implies that in order to prove \eqref{eq:joint-poisson}, it is enough to show that, for every $p,q\in [0,1]$, as $n\to\infty$, 
 \begin{equation}
	 \label{eqn:indep-enough}
 	\dTV(R_n, P)\to 0, \qquad P \sim  \Poi(p\lambda_{12}(\delta) + q\lambda_{13}(\delta)). 
 \end{equation}
Letting $P^{\sss(n)}$ be a Poisson random variable with mean $p\lambda_{12}^{\sss(n)}(\delta) + q\lambda_{13}^{\sss(n)}(\delta)$ with $\lambda_{ij}^{\sss(n)}(\delta)$ as in~\eqref{eq:stein-mean}, it is enough to show that $\dTV(R_n, P^{\sss(n)})\to 0$. 
We aim to apply Poisson approximation via Stein's method~\cite[Theorem 6.23]{JLR00}. 
For any $\beta_1 = (v_1,1,k_1) \in \cI$ and $\beta_2 = (v_2,1,k_2) \in \cI$, $I_{\beta_1}$ and $I_{\beta_2}$ are not independent only if $v_1 = v_2$. 
Thus, \cite[Theorem 6.23]{JLR00} implies
\begin{equation}
	\label{eqn:stein-bound}
	\dTV(W, P^{\sss(n)}) \leq \sum_{\beta_1\in \cI} (\E[\indicwo{\beta_1}])^2 + \sum_{\beta_1,\beta_2\in \cI: v_1 = v_2, \beta_1\neq \beta_2}	\E[\indicwo{\beta_1}\indicwo{\beta_2}] := 2(b_1+b_2).  
\end{equation}
Thus it is enough to show $b_1, b_2\to 0$ as $n\to\infty$. Indeed, using $p_{ij} \leq \percn w_i w_j/\ell_n$, 
\begin{eq}
b_1\leq C \sum_{v_1\in V_n(\delta), k_1 = 2,3}\bigg(\pi_c^2  \frac{w_1w_{v_1}^2w_{k_1}}{\ell_n^2}\bigg)^2\leq  Cn^{4\alpha - 4}\perc^4 \sum_{v_1\in \cN_n(\delta)}w_{v_1}^4 \to 0,
\end{eq}
where the last step uses \eqref{pois-total-var-stein}. 
Similarly, 
\begin{eq}
b_2 &\leq C  \sum_{v_1\in V_n(\delta),k_1,k_2 = 2,3}  \pi_c^3 \frac{w_1w_{v_1}^3w_{k_1}w_{k_2}}{\ell_n^3}\leq C\frac{\perc^3w_1^3}{\ell_n^3} \sum_{v_1:w_{v_1}\leq \delta^{-1} n^{\rho}} w_{v_1}^3 \\
 &\leq C\frac{\perc^3w_1^3 n^{\rho}}{\delta \ell_n^3} \sum_{v:w_v\leq \delta^{-1}n^{\rho}} w_v^2 \leq C \perc^3 n^{3\alpha -3 +\rho +1+(3-\tau)\rho} = 	O(\perc).
\end{eq}
This completes the proof of \eqref{eqn:stein-bound} and thus we have proven the asymptotic independence stated in Proposition~\ref{prop:hub-connection} for $K=2$.
The proof of the asymptotic independence in Proposition~\ref{prop:hub-connection} for general $K$ follows the same line of argument, now using a $K(K-1)/2$-dimensional version of the Poisson Cram\'er-Wold device in \cite[Corollary 2.2]{AHH16}. We omit further details.
\end{proof}

\begin{remark}[Related rank-one models] \label{rem:related-2}\normalfont 
The proof of Proposition~\ref{prop:hub-connection} extends verbatim for the Chung-Lu model and generalized random graph with $\lambda_{ij}$ replaced by $\lambda_{ij}^{\sss \mathrm{CL}}(x)$ and $\lambda_{ij}^{\sss \mathrm{GRG}}(x)$, respectively, where 
\begin{eq}
\lambda_{ij}^{\sss \mathrm{CL}}(x) = \lambda^2 \int_0^\infty \Theta_i^{\sss \mathrm{CL}}(x) \Theta_j^{\sss \mathrm{CL}}(x) \dif x, \quad \lambda_{ij}^{\sss \mathrm{GRG}}(x) = \lambda^2 \int_0^\infty \Theta_i^{\sss \mathrm{GRG}}(x) \Theta_j^{\sss \mathrm{GRG}}(x) \dif x,
\end{eq}
where $\Theta_i^{\sss \mathrm{CL}}(x)$ and $ \Theta_i^{\sss \mathrm{GRG}}(x)$ are defined in \eqref{eq:theta-i-CL-GRG}. 
Indeed, all the asymptotic bounds only use the fact that $p_{uv} \leq \{w_uw_v/\ell_n,1\}$. The mean of the Poisson approximation changes depending on the model due to the computations in \eqref{eq:stein-mean}. 
\end{remark}

\subsection{Path-counting estimates} 
\label{sec:key-ingredients}
In this section, we prove path-counting estimates for $\NRp$ for $\lambda <\lambda_c$.  
Such estimates will play a pivotal role in showing that, when we start exploring from a hub, most vertices are found within a finite distance (see Proposition~\ref{prop:tail-particular-component} in the next Section). 
Similar estimates arise also in the context of  preferential attachment model for example~\cite[Lemma 2.4]{DHH10}.  
For two distinct vertices $i\neq j\in [n]$, let $f_{2k}(i,j)$ denote the probability that there exists a path of length $2k$ from $i$ to $j$ in $\NRp$. 
\begin{proposition}[Connection probabilities at even distance]\label{lem:geometric-decay}
Fix $\varepsilon>0$ and $\lambda < \lambda_c$. 
There exists $n_0 = n_0(\varepsilon)\geq 1$ and $b= b(\varepsilon)\in (\frac{1}{2},\alpha)$ such that for all $n\geq n_0$, $k\geq 1$, and $i\neq  j\in [n]$, 
\begin{eq}\label{eq:even-distance-path-count}
f_{2k}(i,j) \leq  (1+\varepsilon)^{2k}\Big(\frac{\lambda}{\lambda_c} \Big)^{2k}\frac{1}{(i\wedge j)^{1-b}(i\vee j)^b},
\end{eq}
where $\lambda_c$ is defined by \eqref{eqn:lambc-def}. 
\end{proposition}

\begin{proof}
Fix $\vep>0$. Without loss of generality, let $i<j$ so that $w_i> w_j$. 
Let us first relate the expected number of two-step connections to $h$ given in \eqref{eqn:hij-def}. 
We achieve this by showing that there exists $n_0 = n_0(\varepsilon)\geq 1$ such that for all $n\geq n_0$ and $i\neq j, i,j\in [n]$,
\begin{equation}\label{2hop-path-prob-pref-attachment-n}
p_{n}(i,j):=  \sum_{v\in [n]\setminus \{i,j\}} \perc^2 p_{iv} p_{vj} \leq (1+\varepsilon) \lambda^2 h(i,j).
\end{equation}
Using that $1-\e^{-x}\leq x$ for all $x>0$ and $\ell_n = (1+o(1))n\mu$, we can  bound
\begin{eq}\label{p-n-ij-ub-1}
p_n(i,j) &\leq (1+\varepsilon)\lambda^2 n^{-(3-\tau)} \frac{\cf^2n^{2\alpha-1}}{ \mu j^{\alpha}} \sum_{v\in [n]} \Big(1- \e^{-\frac{\cf^2}{\mu} n^{\frac{3-\tau}{\tau -1}} i^{-\alpha} v^{-\alpha}}\Big)v^{-\alpha}\\
&\leq (1+\varepsilon)\lambda^2  \frac{\cf^2}{ \mu j^{\alpha}} \int_0^\infty \Big(1- \e^{-\cf^2 i^{-\alpha} z^{-\alpha}/\mu}\Big)z^{-\alpha} \dif z.
\end{eq}
The final term is identical to the right hand side of \eqref{ub-rij}, and using the exact same argument following \eqref{ub-rij}, the proof of \eqref{2hop-path-prob-pref-attachment-n} follows.

We next investigate more general even-length paths. For any $k\geq 1$, define $\mathcal{I}_k = \mathcal{I}_k(i,j):=  \{\bld{v}=(v_j)_{j=0}^k\colon v_0 = i, v_{k} = j, \text{ and }v_j\text{'s are distinct}\}$, i.e., the set of possible self-avoiding paths of length $k$ started at $i$ and ending at $j$. 
Using \eqref{2hop-path-prob-pref-attachment-n},
\begin{eq}
f_{2k}(i,j) \leq  \big((1+\varepsilon)\lambda^2\big)^k\sum_{\bld{v}\in \mathcal{I}_k} \prod_{r=1}^k h(v_{r-1},v_r).
\end{eq}
Using $\lambda_c = \sqrt{\eta/4B_\alpha}$ from \eqref{eqn:lambc-def}, it is enough to show that, for any $k\geq 1$, 
\begin{eq}\label{eq:enough-k-path-count}
\gamma_{k} (i,j):=\sum_{\bld{v}\in \mathcal{I}_k} \prod_{r=1}^k \frac{1}{(v_{r-1}\wedge v_r)^{1-\alpha}(v_{r-1}\vee v_r)^{\alpha}} \leq (1+\varepsilon)^k\Big(\frac{4}{\eta} \Big)^k \frac{1}{(i\vee j)^{1-b}(i\wedge j)^b}.
\end{eq}
We use induction on $k$. For $k=1$, 
\begin{eq}
\gamma_1(i,j) = \frac{1}{i^{1-\alpha} j^\alpha} = \frac{1}{i} \Big(\frac{i}{j}\Big)^\alpha<\frac{1}{i} \Big(\frac{i}{j}\Big)^{b} < (1+\varepsilon)\frac{4}{\eta} \frac{1}{(i\vee j)^{1-b}(i\wedge j)^b},
\end{eq}
where the third step follows using $i<j$ and $b<\alpha$, and the final step follows using $\eta<4$.

Next, let us indicate the choice of $b$ that works. 
For $b\in (1-\alpha,\alpha)$, let 
    \[f(b) = \frac{1}{\alpha+b-1}+\frac{1}{\alpha-b},
    \]
which has a unique minimum at $b=\frac{1}{2}$ and $f(\frac{1}{2}) = \frac{4}{\eta}$. Since $f$ is continuous, we can choose $b = b(\varepsilon)>\frac{1}{2}$ such that $f(b) <(1+\varepsilon) \frac{4}{\eta}$. 
This will be the $b$ that we work with from now on.

The induction step for proving \eqref{eq:enough-k-path-count} is given by 
\begin{eq}
&\gamma_{k+1}(i,j) \leq \sum_{v < i}\frac{1}{i^{\alpha} v^{1-\alpha}} \gamma_k(v,j)+\sum_{v > i}\frac{1}{i^{1-\alpha} v^{\alpha}} \gamma_k(v,j)\\
&\leq (1+\varepsilon)^k\Big(\frac{4}{\eta} \Big)^k \bigg[\frac{1}{i^{\alpha} j^b}\sum_{v < i} \frac{1}{v^{2 -\alpha-b}}+ \frac{1}{i^{1-\alpha} j^b}\sum_{i<v <j} \frac{1}{v^{1-b +\alpha}}+\frac{1}{i^{1-\alpha} j^{1-b}}\sum_{v >j} \frac{1}{v^{b +\alpha}}\bigg] \\
& \leq (1+\varepsilon)^k\Big(\frac{4}{\eta} \Big)^k \bigg[\frac{1}{i^{\alpha} j^b}\int_0^i \frac{\dif v}{v^{2 -\alpha-b}}+ \frac{1}{i^{1-\alpha} j^b}\int_i^j \frac{\dif v}{v^{1-b +\alpha}}+\frac{1}{i^{1-\alpha} j^{1-b}}\int_j^\infty \frac{\dif v}{v^{b +\alpha}}\bigg] \\& = (1+\varepsilon)^k\Big(\frac{4}{\eta} \Big)^k \bigg[\frac{1}{i^{\alpha} j^b} \frac{i^{\alpha+b-1}}{\alpha+b-1}+ \frac{1}{i^{1-\alpha} j^b} \bigg(\frac{j^{b-\alpha}}{b-\alpha} - \frac{i^{b-\alpha}}{b-\alpha}\bigg) + \frac{1}{i^{1-\alpha} j^{1-b}} \frac{j^{1-\alpha-b}}{\alpha+b-1}\bigg]
\\
&= (1+\varepsilon)^k\Big(\frac{4}{\eta} \Big)^k \bigg[\frac{1}{i^{1-b} j^b}\bigg(\frac{1}{\alpha+b-1} + \frac{1}{\alpha-b}\bigg)+ \frac{1}{i^{1-\alpha} j^\alpha}\bigg(\frac{1}{\alpha+b-1} - \frac{1}{\alpha-b}\bigg)\bigg] \\
&\leq (1+\varepsilon)^{k+1}\Big(\frac{4}{\eta} \Big)^{k+1} \frac{1}{i^{1-b} j^b},
\end{eq}
where in the last step we have bounded the first term using our choice of $b$, and the second term is negative since $\alpha+b-1>\alpha-b$ for $b>\tfrac{1}{2}$.
Thus, the proof follows.
\end{proof}

\subsection{Negligible contributions to the total weight} \label{sec:negligible-contributions}
Let $\sC(i)$ denote the component in $\NRp$ containing vertex $i$ and $W_k(i) = \sum_{v\in \sC(i), \dst(v,i) = k}w_v$, where $\dst (\cdot, \cdot)$ is used in the rest of the paper for graph distance.
We will later see that $\sC(i)$, appropriately normalized, is close to $W(i) = \sum_{k=1}^\infty W_k(i)$. 
In this section, we identify the terms that provide negligible contributions to $W(i)$.  The next proposition states that the  contribution to the total weight arising from vertices in  odd neighborhoods is small. Moreover, the total weight outside a large, but finite, neighborhood of $i$ is also negligible. Intuitively, this is due to the hubs appearing only in finite even distances, and these hubs are unlikely to be at very large distance.
\begin{proposition} \label{prop:tail-particular-component}Suppose that $\lambda\in (0,\lambda_c)$.
For any fixed $i\geq 1$ and $\varepsilon^\prime > 0$,
	\begin{equation}
	\label{eq:contr-negligible-both}
	\lim_{K\to\infty}\limsup_{n\to\infty}\PR\bigg(\sum_{k>K}W_{2k}(i) > \varepsilon^\prime n^{\alpha}\bigg) = 0 \quad \text{and} 
	\quad \lim_{n\to\infty} \PR\bigg(\sum_{k=0}^\infty W_{2k+1}(i) > \varepsilon^\prime n^{\alpha} \bigg) = 0.
	\end{equation}
\end{proposition}

\begin{proof} We start by proving the result on even distances.
Recall the definition of $f_k(i,j)$ from Proposition~\ref{lem:geometric-decay}. 
Since $\lambda<\lambda_c$, we can choose $\varepsilon>0$ sufficiently small such that $\Lambda = (1+\varepsilon)^2(\lambda/\lambda_c)^2<1$.
Therefore, using Proposition~\ref{lem:geometric-decay},
\begin{eq}\label{eq:path-count-geom}
n^{-\alpha}\E[W_{2k}(i)] &\leq n^{-\alpha}\sum_{j\in [n]}w_{j} f_{2k}(i,j) \leq \cf \Lambda^k \bigg[\sum_{j \leq i} \frac{1}{ i^bj^{1-b +\alpha}}+\sum_{j >i} \frac{1}{i^{1-b}j^{\alpha+b}} \bigg]\leq \frac{C\Lambda^k }{i^b},
\end{eq}
for some constant $C>0$, where in the last step we have used that $b\in (\frac{1}{2},\alpha)$.
Since $\Lambda<1$, an application of Markov's inequality proves the first part of \eqref{eq:contr-negligible-both}. 

Next, we compute $\E[W_{2k+1}(i)]$. Using \eqref{eq:path-count-geom}, 
\begin{eq}
 n^{-\alpha}\E[W_{2k+1}(i)] &\leq n^{-\alpha} \sum_{v\in [n]} \PR(\{i,v\} \text{ is an edge}) \E[W_{2k}(v)] \leq  n^{-\alpha}\sum_{v\in [n]} \perc p_{iv} \frac{C\Lambda^k n^{\alpha}}{v^b}. 
\end{eq}
Let us split  the above sum in two terms by taking partial sums over $\{v\colon w_iw_v\leq \ell_n\}$ and $\{v \colon w_iw_v> \ell_n\}$, respectively. Denote the two terms by $\mathrm{(I)}$ and $\mathrm{(II)}$ respectively. 
Then, by Lemma~\ref{lem:order-estimates},
\begin{eq}
\mathrm{(I)} &\leq C \perc \frac{\Lambda^k n^{2\alpha - 1}}{i^{\alpha}}\sum_{v>Cn (w_i/\ell_n)^{\tau-1}} \frac{1}{v^{\alpha+b}} \leq C\perc \frac{\Lambda^kn^{2\alpha-1} }{i^{\alpha}}\Big(n\Big( \frac{w_i}{\ell_n}\Big)^{\tau-1}\Big)^{1-\alpha-b}\leq \frac{C \Lambda^k n^{-\varepsilon_0}}{i^{1-b}},
\end{eq}where $\varepsilon_0 = (3-\tau) (b- \frac{1}{2})>0$.
Similarly, 
\begin{eq}
\mathrm{(II)} \leq C\perc \Lambda^k \sum_{v\leq Cn (w_i/\ell_n)^{\tau-1}} v^{-b} \leq C\perc \Lambda^k\Big(n\Big( \frac{w_i}{\ell_n}\Big)^{\tau-1}\Big)^{1-b}\leq \frac{C \Lambda^k n^{-\varepsilon_0}}{i^{1-b}},
\end{eq}
and thus we conclude that, 
\begin{eq}\label{odd-path-weight-total}
\E[W_{2k+1}(i)] \leq \frac{C \Lambda^k n^{-\varepsilon_0}}{i^{1-b}}.
\end{eq}
The second assertion of \eqref{eq:contr-negligible-both}  again follows using Markov's inequality.
\end{proof}

The next proposition states that for each fixed $k\geq 1$, the primary contribution to $W_{2k}(i)$ arises only due to the hubs.
In its statement, we let $W_{k}^{\sss >R}(i):= \sum_{v\notin [R], \dst(v,i)=k}w_v$. 
\begin{proposition}[Weight of non-hubs at even distances]\label{prop-non-hub-contribution-K-nbd}
Suppose that $\lambda\in (0,\lambda_c)$. 
For any fixed $i\geq 1$, and $\varepsilon^\prime>0$, 
\begin{equation}
\lim_{R\to \infty}\limsup_{n\to\infty}\PR\bigg(\sum_{k=1}^{\infty} W_{2k}^{\sss >R}(i)>\varepsilon^\prime n^{\alpha} \bigg) =0.
\end{equation}
\end{proposition}

\begin{proof}
As before, in Proposition~\ref{lem:geometric-decay} choose $\varepsilon>0$ sufficiently small such that $\Lambda= (1+\varepsilon)^2(\lambda/\lambda_c)^2<1$. 
Choose $R$ large so that $i\in [R]$. Using Proposition~\ref{lem:geometric-decay},
\begin{equation}
\E[W_{2k}^{\sss >R}(i)] \leq C\sum_{v>R} \frac{n^\alpha}{v^\alpha} \frac{\Lambda^k}{i^{1-b}v^{b}} \leq \frac{Cn^{\alpha}\Lambda^k}{i^{1-b}} \sum_{v>R} \frac{1}{v^{\alpha+b}} \leq \frac{C\Lambda^kn^{\alpha}}{i^{1-b}R^{\alpha+b-1}},
\end{equation} 
where we have used that $\alpha+b>1$.
Therefore, 
\begin{equation} \label{eq:weight-nonhub-even}
n^{-\alpha}\E\bigg[\sum_{k=1}^{\infty} W_{2k}^{\sss >R}(i)\bigg] \leq \frac{C}{(1-\Lambda)i^{1-b}R^{\alpha+b-1}}.
\end{equation}
Once again an application of Markov's inequality completes the proof.
\end{proof}

\subsection{Sizes of components containing hubs}\label{sec:component-with-hubs}
In this section, we consider the asymptotic size of $\sC(i)$, the component containing vertex~$i$. 
Recall the asserted limit object  $\mathscr{G}_{\infty}(\lambda)$ from Section~\ref{sec-NR-crit-res}.
In $\mathscr{G}_{\infty}(\lambda)$, let $\sW_k(i) = \sum_{j:\dst(i,j) = k} \theta_j$. 
Thus the total weight of the component containing $i$ in $\cGinf(\lambda)$ is $\sW(i) = \sum_{k=0}^\infty\sW_k(i) $.
We start by relating the asymptotics of the total weight $W(i) = \sum_{k=1}^\infty W_k(i)$ in $\NRp$, defined in the previous Section to $\sW(i)$.  
\begin{theorem}[Total weight containing hub]\label{thm:total-weight-i}
Suppose that $\lambda\in (0,\lambda_c)$.
For each fixed $i\geq 1$, as $n\to\infty$, $n^{-\alpha}W(i) \xrightarrow{\sss d}\sW(i) $.
\end{theorem}
\begin{proof}
Let $W_{k}^{\sss \leq R}(i):= \sum_{j\in [R], \dst(i,j)=k}w_j$ and $\sW_{k}^{\sss \leq R} (i):= \sum_{j\in [R]\colon \dst(i,j) = k}\theta_j$. 
Proposition~\ref{prop:hub-connection} implies that, for any $K,R\geq 1$, 
\begin{eq} \label{eq:hub-total-weight-even}
n^{-\alpha}\sum_{k=1}^KW_{2k}^{\sss \leq R} (i)\dto \sum_{k=1}^K \sW_{k}^{\sss \leq R} (i). 
\end{eq}
Now, $\sum_{k=1}^K \sW_{k}^{\sss \leq R} (i) \nearrow \sum_{k=1}^K \sW_{k}(i)$ almost surely, as $R\to\infty$.
Thus, an application of Proposition~\ref{prop-non-hub-contribution-K-nbd} yields 
\begin{eq}\label{W-i-finite-K-lim}
n^{-\alpha}\sum_{k=1}^{K}W_{2k}(i) \dto \sum_{k=1}^K \sW_{k}(i).
\end{eq}
Finally, $\sum_{k=1}^K \sW_{k}(i) \nearrow \sW(i)$ almost surely, as $K\to\infty$, and thus we conclude the proof using Proposition~\ref{prop:tail-particular-component}. 
\end{proof}

\begin{theorem}[Component sizes of hubs]
\label{thm:comp-size-i}
Suppose that $\lambda\in (0,\lambda_c)$.
For each fixed $i\geq 1$, as $n\to\infty$, $
(n^{\alpha}\perc)^{-1}|\sC(i)| \dto \sW (i).$
\end{theorem}

We start by identifying the main contributions on the component sizes by proving analogues of Propositions~\ref{prop:tail-particular-component}--\ref{prop-non-hub-contribution-K-nbd} for cluster sizes instead of cluster weights.  
Define $\sC_k(i):= \{v\in \sC(i)\colon \dst(v,i) = k\}$. 
Thus $\sC_k(i)$ denotes the set of vertices at distance exactly $k$ from vertex~$i$.
Also, let $\sC_{k}^{\sss R}(i)\subset \sC_{k}(i)$ denote the vertices of $\sC_{k}(i)$ that are neighbors of some vertex in $\sC_{k-1}(i)\cap [R]$.

\begin{lemma}[Main contributions to cluster sizes]
\label{lem:odd-even-com-size}
Suppose that $\lambda\in (0,\lambda_c)$. 
For any fixed $i\geq 1$, and $\varepsilon>0$,
\begin{equation}
\label{eq:even-comp-size}
\lim_{n\to\infty} \PR\bigg( \sum_{k=0}^\infty |\sC_{2k}(i)| > \varepsilon n^{\alpha}\perc \bigg)=  0,\qquad \lim_{K\to\infty}\limsup_{n\to\infty}\PR\bigg( \sum_{k>K} |\sC_{2k+1}(i)| > \varepsilon n^{\alpha} \perc\bigg) = 0,
\end{equation}
and 
\begin{equation}\label{comp-nonhub-odd}
\lim_{R\to \infty}\limsup_{n\to\infty}\PR\bigg(\sum_{k=0}^\infty|\sC_{2k+1}(i)\setminus\sC_{2k+1}^{\sss R}(i)| >\varepsilon n^{\alpha} \perc \bigg) = 0.
\end{equation}
\end{lemma}

\begin{proof}
Note that 
\begin{eq}\label{odd-comp-even-weight}
\E\Big[|\sC_{k+1}(i)|\Big|\bigcup_{r=1}^k \sC_{r}(i)\Big] \leq \sum_{v_1\in \sC_{k}(i)} \sum_{v_2\in [n]}\perc p_{v_1v_2} \leq \perc W_{k}(i),
\end{eq}
and therefore $\E[|\sC_{k+1}(i)|] \leq \perc \E[ W_{k}(i)]$. Now the estimates in Proposition~\ref{prop:tail-particular-component} prove \eqref{eq:even-comp-size}.
Using an identical argument as in \eqref{odd-comp-even-weight} yields  
    \[\E[|\sC_{2k+1}(i)\setminus\sC_{2k+1}^{\sss R}(i)|] \leq \perc \E[W_{2k}^{\sss >R}(i)],\]
and \eqref{comp-nonhub-odd} follows from Proposition~\ref{prop-non-hub-contribution-K-nbd}.
\end{proof}

\begin{proof}[{Proof of Theorem~\ref{thm:comp-size-i}}]
Let us consider the breadth-first exploration of $\sC(i)$ starting from vertex~$i$.  
Let $F_k$ denote the sigma-algebra that contains information about the exploration when all vertices at depth $k$ have been explored. 
Thus,  $\cup_{r=1}^k \sC_{r}(i)$ is measurable with respect to $F_k$.  
Using Lemma~\ref{lem:odd-even-com-size}, and \eqref{W-i-finite-K-lim}, it is now enough to show that, for each fixed $i\geq 1$ and $k,R\geq 1$,
$|\sC_{2k+1}^{\sss R}(i)| = \perc W_{2k}^{\sss \leq R}(i) +\oP(n^{\alpha}\perc)$.
 This follows from Chebyshev's inequality if we can show that for any fixed $k,R\geq 1$,
 \begin{equation} \label{eq:estimate-expt-var-comp}
 \E\big[|\sC_{2k+1}^{\sss R}(i)| \ \big\vert\ F_{2k}\big] = \perc W_{2k}^{\sss \leq R}(i)+ \oP(n^{\alpha}\perc),  \quad \var{|\sC_{2k+1}^{\sss R}(i)| \ \big\vert\ F_{2k}} \leq E_{n},
 \end{equation}
 where  $\E[E_n] = o(n^{2\alpha}\perc^2)$. 
To this end, we first note that  
\begin{eq}
&\E[|\sC_{2k+1}^{\sss R}(i)| \mid F_{2k}] \leq \sum_{v\notin \cup_{r\leq 2k}\sC_{r}(i)} \sum_{u\in \sC_{2k}(i) \cap [R]}   \perc p_{uv}\leq \sum_{u\in \sC_{2k}(i) \cap [R]} \sum_{v\in [n]}  \perc p_{uv}. 
\end{eq}
Further, using inclusion-exclusion with respect to the union of $u\in \sC_{2k}(i)\cap [R]$ (for each $v\notin \sC_{2k}(i)$), it  follows that 
\begin{eq}
\label{incl-excl-CR}
\E[|\sC_{2k+1}^{\sss R}(i)| \vert F_{2k}] 
&\geq \sum_{v\notin \cup_{r\leq 2k}\sC_{r}(i)} \sum_{u\in \sC_{2k}(i) \cap [R]} \perc p_{uv} -
\sum_{v\notin \cup_{r\leq 2k}\sC_{r}(i)} \sum_{\substack{u_1,u_2\in \sC_{2k}(i) \cap [R],\\  u_1<u_2}} \perc^2p_{u_1v}p_{u_2 v}.
\end{eq} 
Let us denote the first and second term in \eqref{incl-excl-CR} by $\mathrm{(I)}$ and $\mathrm{(II)}$, respectively. 
Note that 
\begin{eq}
\mathrm{(II)} \leq  \sum_{u_1,u_2\in [R]} \sum_{v\in [n]} \perc^2p_{u_1v}p_{u_2v} = O(1) = o(n^{\alpha}\perc), 
\end{eq}almost surely, where the second step  follows using \eqref{hub-connection-term-1} and \eqref{hub-connection-term-2}. 
Further, we observe that
\begin{eq}
\frac{1}{n^{\alpha}\perc} \sum_{u\in \sC_{2k}(i) \cap [R]} \sum_{v\in \cup_{r\leq 2k}\sC_{r}(i)} \perc\Big(1-\e^{-w_uw_v/\ell_n}\Big)\leq \frac{1}{n^{\alpha}\ell_n} \bigg(\sum_{r\leq 2k}W_{ r}(i)\bigg)^2 = \OP(n^{\alpha-1}) = \oP(1),
\end{eq}
where in the second step, we have used Theorem~\ref{thm:total-weight-i}.
It thus follows that 
\begin{eq}\label{eq:first-approx-odd-comp}
\E[|\sC_{2k+1}^{\sss R}(i)|~ \big\vert F_{2k}] =  \sum_{u\in \sC_{2k}(i) \cap [R]} \sum_{v\in [n]} \perc p_{uv} +\oP(n^{\alpha}\perc).
\end{eq}
We now simplify the right hand side of \eqref{eq:first-approx-odd-comp}.
Fix $\varepsilon\in (0,\rho^2)$, and let us split the sum in two parts with $\{v\colon w_v\leq n^{\rho-\varepsilon}\}$, $\{v\colon w_v>n^{\rho-\varepsilon}\}$, and denote them by $\mathrm{(Ia)}$ and $\mathrm{(Ib)}$ respectively.
Using Lemma~\ref{lem:order-estimates}, and the fact that $-\rho(\tau-2) +\varepsilon (\tau-1)<0$ since $\varepsilon<\rho^2$,
\begin{eq}
\frac{\mathrm{(Ib)}}{n^{\alpha}\perc} \leq C R \frac{n^{1- (\tau-1)\rho+\varepsilon (\tau-1)}}{n^{\alpha}} \leq C R n^{-\rho(\tau-2) +\varepsilon (\tau-1)} = o(1), \qquad \text{almost surely},
\end{eq}
while 
\begin{eq}
\mathrm{(Ia)} = \perc \sum_{u\in \sC_{2k}(i) \cap [R]} \sum_{v\colon w_v \leq n^{\rho-\varepsilon}} \frac{w_uw_v}{\ell_n(1+o(1))} = \perc W_{2k}^{\sss \leq R}(i) (1+o(1)).
\end{eq}
The estimate for the expectation term in \eqref{eq:estimate-expt-var-comp} now follows. 

For $u\in \sC_{2k}(i)$, let $N_u$ denote the number of neighbors of $u$ in $\sC_{2k+1}(i)$. For the variance term, it follows using  the independence of edge occupancies  that
\begin{eq}
\var{|\sC_{2k+1}^{\sss R}(i)|\ \big\vert \ F_{2k}} &= \sum_{u\in \sC_{2k}(i) \cap [R]} \var{N_u} \leq \sum_{u\in \sC_{2k}(i) \cap [R]} \sum_{v\in [n]} \perc p_{uv} \leq \perc W_{2k}^{\sss \leq R}(i)=:E_n.
\end{eq} 
Using \eqref{eq:hub-total-weight-even} and the fact that $n^{-\alpha} W_{2k}^{\sss \leq R}(i)$ is bounded, we see that $\E[E_n] =O(n^{\alpha} \perc)= o(n^{2\alpha} \perc^2)$, which proves the required estimate in \eqref{eq:estimate-expt-var-comp}. Hence, the proof of Theorem~\ref{thm:comp-size-i} is complete.
\end{proof}

\subsection{Tightness of component sizes and weights: Proof of Theorem~\ref{thm:main-crit}} 
\label{sec:l2-tightness}
The goal of this section is to show that the vector of component sizes and their weights (appropriately normalized) is tight in~$\ell^2$. The proof will also show that the largest connected components correspond to those containing hubs. 
Then the proof of Theorem~\ref{thm:main-crit} will follow using Theorems~\ref{thm:total-weight-i}--\ref{thm:comp-size-i}. To this end, define 
\begin{eq}
\label{C-leq-def-2}
\sC_{\leq} (j) = 
\begin{cases}
\sC(j) \quad & \text{ if }j = \min \{v\colon v\in \sC(j)\}, \\
\varnothing \quad &\text{ otherwise,}
\end{cases}
\end{eq}
and let $\cW_{\sss \leq }(j):= \sum_{k\in \sC_{\leq} (j)}w_k$. 
The main ingredient is the following proposition: 
\begin{proposition}[Tightness in $\ell^2$]\label{prop-tight-l2}
Suppose that $\lambda\in (0,\lambda_c)$. 
For any $\varepsilon>0$, 
\begin{eq}\label{leq-comp-tight}
\lim_{K\to\infty}\limsup_{n\to\infty}\PR\bigg(\sum_{j>K} |\sC_{\leq} (j)|^2 > \varepsilon \perc^2n^{2\alpha}  \bigg) = 0,
\end{eq}
\begin{eq}\label{leq-weight-tight}
\lim_{K\to\infty}\limsup_{n\to\infty}\PR\bigg(\sum_{j>K} \big(\cW_{\sss \leq }(j)\big)^2 > \varepsilon n^{2\alpha}  \bigg) = 0.
\end{eq}
\end{proposition}

\begin{proof}
Recall that $\sCi$ is the $i$-th largest component of $\NRp$,  $W_{\sss (i)} = \sum_{v\in \sC_{\sss (i)}} w_v$ (we have suppressed the dependence of $\perc = 
\percl$ in the notation).  For a fixed $K\geq 1$, consider the graph $\NRp\setminus [K]$.
We augment a previously defined notation with a superscript $>K$ to denote the corresponding quantity for $\NRp\setminus [K]$.
Since the components $\{\sC_{\leq} (j)\colon j> K\}$ do not contain any vertices in~$[K]$, $\sum_{j>K} |\sC_{\leq} (j)|^2 \leq \sum_{i\geq 1} |\sC_{\sss (i)}^{\sss >K}|^2$. Therefore, it is enough to show that for any $\varepsilon>0$,
\begin{equation}
\label{eq:tightness-suff}
\lim_{K\to\infty} \limsup_{n\to\infty} \PR\bigg(\sum_{i\geq 1} |\sC_{\sss (i)}^{\sss >K}|^2 > \varepsilon \perc^2n^{2\alpha}\bigg) = 0, \quad \lim_{K\to\infty} \limsup_{n\to\infty} \PR\bigg(\sum_{i\geq 1} (W_{\sss (i)}^{\sss >K})^2 > \varepsilon n^{2\alpha}\bigg) = 0.
\end{equation}
Using the weight sequence $(w_i)_{i>K}$, let $V_n^{*, {\sss >K}}$ denote a vertex chosen in a size-biased manner from $[n]\setminus [K]$ chosen independently from $\NRp$ (i.e. for any $i>K$, $\PR(V_n^{*, {\sss >K}} = i) \propto w_i$). 
Let $\ell_n^{\sss >K}:=\sum_{i>K}w_i$. Then, $\ell_n^{\sss >K} \leq \ell_n$ for all $K\geq 1$.
Note that \eqref{odd-comp-even-weight} yields 
\begin{eq}
&\E\bigg[\sum_{i\geq 1}|\sC_{\sss (i)}^{\sss >K}|^2\bigg] = \E\bigg[\sum_{v\in [n]\setminus [K]}|\sC^{\sss >K}(v)|\bigg] \leq \perc \E\bigg[\sum_{v\in [n]\setminus [K]}W^{\sss >K}(v)\bigg] \\
&= \perc\E\bigg[\sum_{i\geq 1}|\sC_{\sss (i)}^{\sss >K}| \times W_{\sss (i)}^{\sss >K}\bigg] = \ell_n^{\sss >K} \perc \E[|\sC^{\sss >K}(V_n^{*, {\sss >K}})|] \leq \ell_n^{\sss >K} \perc^2 \E[W^{\sss >K}(V_n^{*, {\sss >K}})]. 
\end{eq}
Further,
\begin{eq}
\E\bigg[\sum_{i\geq 1}(W_{\sss (i)}^{\sss >K})^2\bigg]= \ell_n^{\sss >K}  \E[W^{\sss >K}(V_n^{*, {\sss >K}})].
\end{eq}
Now, by \eqref{eq:path-count-geom} and \eqref{odd-path-weight-total}, for any fixed $v\in [n]$,  $\E[W^{\sss >K} (v)] \leq Cn^{\alpha} (v^{-b}+ n^{-\varepsilon_0}v^{-(1-b)})$, where $C>0$ is independent of $K$,  and hence,
\begin{eq} \label{ss-simple-1}
\ell_n^{\sss >K} \E[W^{\sss >K}(V_n^{*, {\sss >K}})] \leq C n^{2\alpha} \sum_{v>K} \frac{1}{v^{b+\alpha}} + n^{2\alpha-\varepsilon_0} \sum_{v>K} \frac{1}{v^{1-b+\alpha}}.
\end{eq}
Since $b\in (\frac{1}{2},\alpha)$, both $\sum_{v>K} \frac{1}{v^{b+\alpha}}$ and $\sum_{v>K} \frac{1}{v^{1-b+\alpha}}$ go to zero as $K\to\infty$.
Therefore,
\begin{eq}
\lim_{K\to\infty}\limsup_{n\to\infty}(n^{\alpha} \perc)^{-2}\E\bigg[\sum_{i\geq 1}|\sC_{\sss (i)}^{\sss >K}|^2\bigg] = 0, \quad \lim_{K\to\infty}\limsup_{n\to\infty}n^{-2\alpha} \E\bigg[\sum_{i\geq 1}(W_{\sss (i)}^{\sss >K})^2\bigg] = 0. 
\end{eq}
Thus, \eqref{eq:tightness-suff} follows using Markov's inequality completing the proof of Proposition~\ref{prop-tight-l2}.
\end{proof}

\begin{proof}[Proof of Theorem~\ref{thm:main-crit}]
We give the proof for the component sizes. The proof for the weight follows similarly. 
Let $\sC_{\sss (i),K} $ be the $i$-th largest component among $\{\sC_{\leq} (j)\colon j\leq K\}$. 
For $K= \infty$, $\sC_{\sss (i),K}  = \sCi$. 
We first show that for each fixed $i\geq 1$,  $\sCi \approx \sC_{\sss (i),K}$ \c{if}{} $K$ is large. More precisely, for any fixed $ r\geq 1$ and $\varepsilon>0$,
\begin{eq}\label{eq:hubs-largest-relation}
\lim_{K\to\infty}\limsup_{n\to\infty}\PR\big(\exists i\leq r\colon  \big||\sC_{\sss (i),K} | - |\sC_{\sss (i)}| \big|>\varepsilon \perc n^{\alpha} \big) = 0. 
\end{eq}
Indeed, if $\sC_{\sss (1),K} \neq \sC_{\sss (1)} $, then $\sC_{\sss (1)}  = \max \{\sC_{\leq} (j)\colon j> K\}$, and therefore
\begin{eq}
\big||\sC_{\sss (1),K} | - |\sC_{\sss (1)}| \big| \leq \bigg(\sum_{j>K} |\sC_{\leq} (j)|^2\bigg)^{1/2}.
\end{eq}
Next, on the event  $\{\sC_{\sss (1),K} = \sC_{\sss (1)} \}$, we can similarly bound $\big||\sC_{\sss (2),K} | - |\sC_{\sss (2)}| \big| \leq (\sum_{j>K} |\sC_{\leq} (j)|^2)^{1/2}$ and in general on the event $\{\sC_{\sss (i),K} = \sC_{\sss (i)}, \ \forall i \in [r-1] \}$, we can also bound $\big||\sC_{\sss (i),K} | - |\sC_{\sss (i)}| \big| \leq (\sum_{j>K} |\sC_{\leq} (j)|^2)^{1/2}$. Thus \eqref{eq:hubs-largest-relation} follows using Proposition~\ref{prop-tight-l2}. 

Next, note that $(\sC_{\leq} (j))_{j\in [K]}$ is the collection of components $(\sC (j))_{j\in [K]}$ with multiplicities removed and replaced by empty sets (recall \eqref{C-leq-def-2}). Thus, $|\sC_{\sss (1), K}| = \max_{j\in [K]} |\sC(j)|$, and similar identities holds for $|\sC_{\sss (i), K}|$.
Thus, using \eqref{eq:hubs-largest-relation} and Theorem~\ref{thm:comp-size-i}, we conclude that  $((n^\alpha \perc)^{-1}|\sCi|)_{i\geq 1}$ converges to our desired limiting object in finite-dimensional sense. The $\ell^2_{\shortarrow}$-tightness follows by observing that $\sum_{j>K}|\sC_{\sss (j)}|^2\leq \sum_{j>K} |\sC_{\leq} (j)|^2$. 
\end{proof}

\subsection{Sub-critical behavior: proof of Theorem \ref{thm:barely-subcrit-single-edge}}
\label{sec-proof-barely-sub-GRG}
The proof of Theorem~\ref{thm:barely-subcrit-single-edge} can be completed by modifying the arguments for the critical regime. 
In fact, if $\percn = \lambda_n n^{-(3-\tau)/2}$ for some $\lambda_n\to 0$, then the hub-connection probabilities tend to zero as shown in \eqref{eq:sub-hub-con}. 
Moreover, we can follow identical arguments as in Proposition~\ref{prop:tail-particular-component} and Lemma~\ref{lem:odd-even-com-size} to show that $W(i) = w_i (1+\oP(1))$ and $|\sC(i)| = \percn w_i (1+\oP(1))$.
To successfully apply Chebyshev's inequality to get these asymptotics, we need $\percn w_i \to \infty$, which is true since $\percn\gg n^{-\alpha}$ by the assumptions of Theorem~\ref{thm:barely-subcrit-single-edge}. 
Finally, we can use identical arguments as in Proposition~\ref{prop-tight-l2} to deduce the $\ell^2_{\sss\shortarrow}-$tightness of the vector of component sizes and weights. 
Thus, the proof of Theorem~\ref{thm:barely-subcrit-single-edge} follows.
\qed

\section{The giant in the embedded inhomogeneous random graph}
\label{sec:small-giant} 
Henceforth, we consider the supercritical case, i.e., $\percn = \lambda n^{-\seta}$ for $\lambda > \lambda_c$.
In this section, we proceed to set up the main conceptual ingredients for the emergence of the giant for $\lambda>\lambda_c$. 
Fix a parameter $a>0$, and define 
	\eqn{
	\label{Nna-def}
	N_n(a)=\lfloor a n^{(3-\tau)/2} \rfloor.
	}
We also denote
	\eqn{
	\label{Nn1-def}
	N_n=N_n(1).
	}
By \eqref{theta(i)-def}, we note that,  for $i\in \lceil N_n u \rceil $	and $u \in (0,a]$ 
\begin{eq}\label{eq:asymp-w-N}
	w_{\lceil N_n u \rceil}
	=
	\cf u^{-\alpha}\Big(\frac{n}{N_n}\Big)^{\alpha}
	=\cf u^{-\alpha}\big(n^{(\tau-1)/2}\big)^{\alpha} \asymp \sqrt{n} \cf u^{-\alpha},
\end{eq}
and thus  $[N_n(a)]$ consists of vertices with weight at least of order $\sqrt{n}a^{-\alpha}$. 

The key conceptual step is that, if $a$ is large enough, then a giant component emerges inside $[N_n(a)]$ that forms the core connectivity structure of the giant component in the whole graph. In turn, this graph is an inhomogeneous random graph, for which the critical value can be determined exactly, as we explain in more detail now.

To this end, consider the percolated graph  $\rNR(\bw,\percn)$, restricted to $[N_n(a)]$, and denote this subgraph by $\cG_{\sss N_n(a)}$. Then, $\cG_{\sss N_n(a)}$ is distributed as an inhomogeneous random graph that is {\em sparse} in that the number of edges grows linearly in the number of vertices in the graph.
Thus, the emergence of the giant component within $\cG_{\sss N_n(a)}$ can be studied using the general setting of inhomogeneous random graphs developed by Bollob\'as, Janson and Riordan in \cite{BJR07}. 
In particular, the results of 
\cite{BJR07} gives a critical value $\lambda_c(a)$, such that, for $\lambda>\lambda_c(a)$,
a unique and highly concentrated giant exists inside $[N_n(a)]$, that is {\em stable} to the addition of a small proportion of edges. The stability result is used later in Section~\ref{sec-size-tiny-giant} below to understand the perturbation on this giant after adding all the edges outside $[N_n(a)]$.
In Section~\ref{sec:size-core-giant}, we make the connection with the key results from \cite{BJR07} explicit and state the relevant results for our proof. 
 The rest of the section is devoted to analysis of the limiting quantities as $a\to \infty$. 
In Section~\ref{sec:equality-crit-val}, we first show that $\lim_{a\to\infty}\lambda_c(a)  = \lambda_c$, where $\lambda_c$ is given by \eqref{eqn:lambc-def}. 
The connection between $\lambda_c(a)$ and $\lambda_c$ is quite remarkable given the vastly different descriptions of these quantities. We prove this fact by an explicit computation. The convergence of $\lambda_c(a)$ is also a key conceptual step, since it shows that, whenever $\lambda>\lambda_c$, one can choose $a$ to be large enough to make a tiny giant appear inside~$[N_n(a)]$. 
Finally, the asymptotics for functionals of the giant inside $\cG_{\sss N_n(a)}$ are given by survival probabilities of certain multitype branching processes that depend sensitively on $a$. 
In Section~\ref{sec-MBP}, we analyze these survival probabilities as $a\to\infty$.
This sets the stage for Section~\ref{sec-size-tiny-giant}, where we identify the primary contributions to the size of the giant in the whole graph using the giant inside $[N_n(a)]$, for $a$ large enough.

\subsection{Size and weight of the giant core}
\label{sec:size-core-giant}
Consider the measure space $\mathcal{S}_a = ((0,a], \sB((0,a]), \Lambda_a)$, where $\sB((0,a])$ denotes the Borel sigma-algebra on $(0,a]$, and $\Lambda_a(\dif x) = \frac{\dif x}{a}$ is the normalized Lebesgue measure on $(0,a]$.
Recall from~\eqref{eq:p-ij-NR-defn} that the probability that there is an edge between $i$ and $j$ after percolation equals 
	$p_{ij} =\percn [1-\e^{-w_{i}w_{j}/\ell_n}].$
For $u,v\in(0,a]$, define the kernel 
	\eqn{
	\label{kappa-Nn-def}
	\kappa_{\sss N_n}^{\sss (a)}(u,v)=N_n(a)  p_{\lceil N_n u \rceil \lceil N_n v \rceil} /\lambda.
	}
Then putting $u_i^n = i/N_n$, we have that for all $i\in [N_n(a)]$, $p_{ij} = \lambda \kappa_{\sss N_n}^{\sss (a)} (u_i^n,u_j^n)/N_n(a) $. 
Obviously, the empirical measure $\Lambda_{n,a}$ of $(u_i^n)_{i\in [N_n(a)]}$ converges in the weak topology, with limiting measure $\Lambda_a$.
This verifies \cite[(2.2)]{BJR07}, and thus  $(\cS_a, ((u_i^n)_{i\in [N_n(a)]})_{n\geq 1})$ is a vertex space according to the definition in \cite[Section 2]{BJR07}. 

Next, we verify that $(\kappa_{\sss N_n}^{\sss (a)})_{n\geq 1}$ is a sequence of graphical kernels on  $\cS_a$
according to \cite[Definition 2.9]{BJR07}. 
For any $(u_n)_{n\geq 1}, (v_n)_{n\geq 1} \subset (0,a]$ and $u,v\in (0,a]$ with $u_n \to u$ and $v_n\to v$, it follows using \eqref{eq:asymp-w-N} that 
\eqn{
	\label{kappa-a-def}
	\kappa_{\sss N_n}^{\sss (a)}(u_n,v_n)\rightarrow \kappa^{\sss(a)}(u,v):= a [1-\e^{-\cf^{2}(uv)^{-\alpha}/\mu}] \quad \text{for all }u,v\in (0,a].
	}
Note that $\kappa^{\sss (a)}$ is bounded and continuous, and thus the first two conditions of \cite[Definition~2.7]{BJR07} are satisfied. 
Next, note that Lemma~\ref{lem:order-estimates} yields
	\eqan{
	\frac{1}{N_n(a)}\sum_{\substack{i,j\in [N_n(a)]\\ i<j}} \percn [1-\e^{-w_iw_j/\ell_n}]
	&\to \frac{\lambda}{2a}  \int_0^a \int_0^a [1-\e^{-\cf^{2}(uv)^{-\alpha}/\mu}]\dif u\dif v\nn\\
	&=\frac{1}{2} \int_0^a \int_0^a  \lambda \kappa^{\sss(a)}(u,v) \Lambda_a(\dif u)\Lambda_a(\dif v),
	}
which verifies \cite[(2.11)]{BJR07}, and thus all the conditions of \cite[Definition 2.9]{BJR07} have now been verified. 
Finally, $\kappa^{(a)} >0$, so that it is irreducible according to \cite[Definition 2.10]{BJR07}. 
Hence we have verified that $\cG_{\sss N_n(a)}$ is an inhomogeneous random graph with kernels $(\kappa^{\sss(a)}_{\sss N_n})_{n\geq 1}$ satisfying all the requisite good properties in \cite{BJR07}.

To describe the phase transition,  define the integral operator ${\bf T}_{\kappa^{\sss(a)}}: L^2(\mathcal{S}_a) \mapsto L^2(\mathcal{S}_a)$ by 
	\eqn{ \label{eq:integral-operator-defn}
	({\bf T}_{\kappa^{\sss(a)}} f)(u) = \int_{0}^a \kappa^{\sss(a)}(u,v)f(v)\Lambda_a(\dif v) = \int_0^a[1-\e^{-\cf^{2}(uv)^{-\alpha}/\mu}] f(v) \dif v,
	}
and let $\|{\bf T}_{\kappa^{\sss(a)}}\|$ denote its operator norm. 
Let $\sCa_{\sss (i)}$ denote the size of $i$-th largest component of the graph $\cG_{\sss N_n(a)}$. 
Also, let $\cT_{\sss \geq k}^a$ denote the set of vertices that belong to some component of size at least $k$ in $\cG_{\sss N_n(a)}$.

Throughout this section, we suppress $\percn$  in the notation.
To describe the size of the giant component in $[N_n(a)]$,  let  $\cX_{a}^\lambda (u)$ be a multi-type branching process with type space $\cS_a$, where we start from one vertex with type $u\in \cS_a$, and a particle of type $v\in \cS_a$ produces progeny in the next generation according to a Poisson process on $\cS_a$ with intensity $\lambda \kappa^{\sss (a)}(v,x) \Lambda_a(\dif x)$.
Let $\rho_a^\lambda(u)$ be the survival probability of $\cX_{a}^\lambda (u)$, and $\rho_{a,{\sss \geq k}}^\lambda$ denote the probability that $\cX_{a} (u)$ has at least $k$ individuals. 
Define 
\begin{eq}\label{defn:survival-prob}
\rho_a^\lambda = \int_0^a \rho_a^\lambda (u) \Lambda_a(\dif u) = \frac{1}{a}\int_0^a \rho_a^\lambda (u) \dif u, \quad  \rho_{a,{\sss \geq k}}^\lambda = \int_0^a \rho_{a,{\sss \geq k}}^\lambda (u) \Lambda_a(\dif u) = \frac{1}{a}\int_0^a\rho_{a,{\sss \geq k}}^\lambda (u) \dif u.  
\end{eq}
The following proposition describes the emergence of the giant component for $\cG_{\sss N_n(a)}$: 
\begin{proposition}[Emergence of giant in $\mathcal{G}_{\sss N_n(a)}$]\label{prop:giant-restricted}
Under {\rm Assumption~\ref{assumption-NR}}, the following hold for any $a>0$:
\begin{enumerate}[(i)]
\item For $\lambda>\|{\bf T}_{\kappa^{\sss(a)}}\|^{-1}$, $|\sC_{\sss (1)}^{a}| = N_n(a)\rho_{a}^\lambda (1+\oP(1))$, 
and $|\sC_{\sss (2)}^{a}| = \OP(\log(n))$. Further, for each fixed $k\geq 1$, $|\cT_{\sss \geq k}^a| = N_n(a) \rho_{a,{\sss \geq k}}^\lambda (1+\oP(1))$. Finally, $\cT_{\sss \geq k}^a$ is {\em stable}, in the sense that, for every $\vep>0$, there exists a $\delta>0$ such that, with high probability, removing at most $\delta n$ edges from $\mathcal{G}_{\sss N_n(a)}$ changes $\cT_{\sss \geq k}^a$ by at most $\vep n$ vertices.
\item[(ii)] For $\lambda<\|{\bf T}_{\kappa^{\sss(a)}}\|^{-1}$, $|\sC_{\sss (1)}^{a}| = \OP(\log(n))$.
\end{enumerate}
\end{proposition}

\proof
The asymptotics of $|\sCa_{\sss (1)}|$  follow  directly by applying \cite[Corollary 3.2]{BJR07} and
\cite[Theorem 3.12]{BJR07}, and further noting that $\sup_{n,x,y} \kappa_{N_n}^{\sss(a)}(x,y) <\infty$. 
The asymptotics of $|\cT_{\sss \geq k}^a|$ follows using \cite[Theorem 9.1]{BJR07}. 
The stability of the giant in part (i) is proved in \cite[Theorem 11.1]{BJR07}.
\qed
\medskip

We conclude this section by providing the asymptotics of the total weight inside  $\sCa_{\sss (1)}$: 
\begin{proposition}[Weight of the giant in $\cG_{\sss N_n(a)}$]\label{prop-weight-giant-core}
Under {\rm Assumption~\ref{assumption-NR}}, for any fixed $a>0$, as $n\to\infty$, 
\begin{eq}
 \sum_{i\in \sCa_{\sss (1)} } \frac{\percn w_i}{\sqrt{n}}  \pto \zeta_a^\lambda,
\end{eq}
where $\zeta_a^{\lambda}:=\lambda \int_0^a \cf u^{-\alpha} \rho_a^\lambda(u) \dif u$.
\end{proposition}

\begin{proof}
We apply \cite[Theorem 9.10]{BJR07}. 
First the contribution due to $i\leq N_n(\vep)$ can be  almost surely bounded by
	\eqn{
	\frac{1}{N_n(a)}\sum_{i\in [N_n(\vep)]} \frac{w_i}{\sqrt{n}}
	\leq \frac{\cf n^{\alpha}}{N_n(a)\sqrt{n}} \sum_{i\leq \vep N_n} i^{-\alpha}\leq \frac{C\vep^{1-\alpha}}{a}.
	}
Further, for all $i\in [N_n(a)]\setminus [N_n(\vep)]$, the function $i\mapsto \frac{w_i}{\sqrt{n}}$ is bounded. Thus,  \cite[Theorem 9.10]{BJR07} is applicable and we have 
\begin{eq}
\sum_{i\in \sCa_{\sss (1)} } \frac{\percn w_i}{\sqrt{n}} = \frac{\lambda a}{N_n(a)}\sum_{i\in \sCa_{\sss (1)} } \frac{w_i}{\sqrt{n}} \pto a\lambda \int_0^a \cf u^{-\alpha} \rho_a^\lambda(u) \Lambda_a(\dif u) = \zeta_a^\lambda,
\end{eq}
and the proof follows. 
\end{proof}

\subsection{Equality of the critical values}
\label{sec:equality-crit-val}
In this section, we relate the critical values in the inhomogeneous random graph $\mathcal{G}_{\sss N_n(a)}$, for $a$ large, to the critical value $\lambda_c$ defined in \eqref{eqn:lambc-def}. Let us denote $\lambda_c(a) = \|\bfT_{\sss \kappa^{\sss (a)}}\|^{-1}$.
We start by observing a monotonicity of $\lambda_c(a)$:
\begin{lemma}[Monotonicity of $a\mapsto \lambda_c(a)$]
\label{lem-mon-lambdaca}
The function $a\mapsto \lambda_c(a)$ is non-increasing on $[0,\infty)$.
\end{lemma}

\begin{proof} 
Fix $b>a$, and let $N_n(\cdot)$ be as in \eqref{Nna-def}. Fix $\lambda>\lambda_c(a)$. We will prove that then also $\lambda>\lambda_c(b)$, which proves that $\lambda_c(a)\geq \lambda_c(b),$ as required.

Since $\lambda>\lambda_c(a)$, Proposition~\ref{prop:giant-restricted} implies that the graph $\cG_{\sss N_n(a)}$ on vertex set $[N_n(a)]$ has a giant component of size $\rho_{a}^\lambda N_n(a)(1+\oP(1))$, where $\rho_{a}^\lambda>0$ since $\lambda>\lambda_c(a)$.  Denote this component by $\sC_{\sss (1)}^{a}$. 
Since $[N_n(a)]\subseteq [N_n(b)]$, and since the edge probabilities in $\cG_{\sss N_n(a)}$ and~$\cG_{\sss N_n(b)}$ are equal on $[N_n(a)]$,
we can find a coupling under which $\cG_{\sss N_n(a)}$ is a subgraph of $\cG_{\sss N_n(b)}$ with probability one.
Under this coupling, there exists a component of $\sC \subset\cG_{\sss N_n(b)}$ such that $\sC_{\sss (1)}^{a}\subseteq \sC$. 
For any $q>0$, if $|\sC_{\sss (1)}^{a}|\geq q N_n(a)$, then  $|\sC_{\sss (1)}^{b}|\geq |\sC| \geq q N_n(a)\geq q(a/b) N_n(b)$. Thus, as $n\to\infty$,  $\PR(|\sC^b_{\sss (1)}|/N_n(b) \geq \frac{a}{b} \rho_a^\lambda ) \to 1,$
and therefore $\lambda>\lambda_c(b)$ by Proposition~\ref{prop:giant-restricted}, as required.
\end{proof}
Lemma~\ref{lem-mon-lambdaca} implies that $\lim_{a\to\infty} \lambda_c(a)$ exists and is finite. Let
	\eqn{
	 \label{lambda_IRM}
	 \lambda_c^{\sss \mathrm{IRG}}:= \lim _{a\to\infty} \lambda_c(a) =  \inf_{a>0} \lambda_c(a).
	 }
We next show that $\lambda_c^{\sss \mathrm{IRG}}=\lambda_c$:

\begin{lemma}[Equality of critical values]
\label{lem-lambdacs-equal} $\lambda_c^{\sss \mathrm{IRG}}=\lambda_c,$ with $\lambda_c, \lambda_c^{\sss \mathrm{IRG}}$ defined in \eqref{eqn:lambc-def}, \eqref{lambda_IRM} respectively.
\end{lemma}
\begin{proof} Fix $a>0$.  For two functions $f,g\colon [0,a]^2\to [0,\infty)$, we define the operation
	\eqn{
	\label{conv-def}
	(f\star g)(x,y)=\int_{0}^{a} f(x,v)g(v,y) \Lambda_a(\dif v).
	}
We also recursively define $f^{\star (n+1)}=f\star f^{\star n}$, with $f^{\star 1}=f$. We claim that 
	\eqn{
	\|\bfT_{\sss \kappa^{\sss (a)}}\| = \lim_{k\rightarrow \infty} \Big(\int_0^a\int_{0}^{a} \big(\kappa^{\sss (a)}\big)^{\star 2k}(u,v)\Lambda_a(\dif u) \Lambda_a(\dif v)\Big)^{1/(2k)}.
	}
Indeed, $\kappa^{\sss (a)}$ is a bounded function, so that the integral operator $\bfT_{\kappa^{\sss (a)}}$
defined on $L^2([0,a], \Lambda_a)$ given by \eqref{eq:integral-operator-defn}
is Hilbert-Schmidt and thus compact \cite[Theorem 4 in Chapter 22]{Lax02}.  
Further, it is a positive and self-adjoint operator, since $\kappa^{\sss (a)}$ is positive and symmetric. Thus, the largest eigenvalue of $\bfT_{\kappa^{\sss (a)}}$ is positive and separated from the second largest in absolute value \cite[Theorem 1 in Chapter 23]{Lax02}. Finally, as a compact and self-adjoint operator, it has an othonormal basis of eigenfunctions \cite[Theorem 3 in Chapter 28]{Lax02}, so that the claim follows by an expansion in terms of the eigenfunctions.
We can rewrite this with $\kappa^{\sss (a)}_2(u,v)=(\kappa^{\sss (a)}\star \kappa^{\sss (a)})(u,v)$ as
	\eqn{
	\|\bfT_{\sss \kappa^{\sss (a)}}\| = \lim_{k\rightarrow \infty} \Big(\int_0^a\int_0^a \big(\kappa^{\sss (a)}_2\big)^{\star k}(u,v)\Lambda_a(\dif u)\Lambda_a(\dif v)\Big)^{1/(2k)}.
	}
As a result,
	\eqn{
	\|\bfT_{\sss \kappa^{\sss (a)}}\|={\|\bfT_{\sss \kappa^{\sss (a)}_2}\|}^{1/2}.
	}
Next, note that, for any $u,v\in(0,a]$,
	\eqan{
	\kappa^{\sss (a)}_2(u,v)&=\int_0^a \kappa^{\sss (a)}(u,x)\kappa^{\sss (a)}(x,v)\Lambda_a(\dif x)
	=a\int_0^a [1-\e^{-\cf^{2}(ux)^{-\alpha}/\mu}]
	[1-\e^{-\cf^{2}(vx)^{-\alpha}/\mu}]\dif x\nn\\
	&=a \int_0^a \Theta_{u}(x)\Theta_{v}(x) \dif x,\nn
	}
where we recall \eqref{defn:lambda-ij}. When the integral is evaluated on $[0,\infty)$,  $\kappa_2^{\sss (a)}(u,v)$ would be equal to $a \lambda_{uv}/\lambda^2$. Thus,
	\eqn{
	\|\bfT_{\sss \kappa_2^{\sss (a)}}\|
	=\|\bfT_{\sss \bar\lambda^{\sss (a)}}\|_{L^2(0,a)}=\|\bfT_{\sss \bar\lambda^{\sss (a)}}\|_{L^2(0,\infty)},
	}
where, for $u,v\in (0,a]$, we let
	\eqn{
	\bar\lambda^{\sss (a)}(u,v)=\indic{u,v\in (0,a]}\int_0^a \Theta_{u}(x)\Theta_{v}(x) \dif x.
	}
Obviously, $a\mapsto \bar\lambda^{\sss (a)}(u,v)$ is increasing, and it converges pointwise to $\lambda_{uv}/\lambda^2$. As a result, also 
	\eqn{
	\|\bfT_{\sss \kappa^{\sss (a)}}\|
	=\|\bfT_{\sss \bar\lambda^{\sss (a)}}\|_{L^2(0,\infty)}\nearrow \|\bfT_{\sss \lambda}\|_{L^2(0,\infty)},
	}
where $\lambda(u,v)=\lambda_{uv}/\lambda^2$. 
Next, recall $h(\cdot,\cdot)$ from \eqref{eqn:hij-def}. An argument identical to \eqref{lambda-ij-ratio} yields   $\lim_{u\to\infty}\lim_{v\to\infty}\lambda(u,v)/h(u,v) = 1$. 
Thus, by \cite[Lemma 1]{DK90}, 
	\eqn{ \label{operator-norm}
	\|\bfT_{\sss \lambda}\|_{L^2(0,\infty)}=\int_0^{\infty} \frac{h(1,u)}{\sqrt{u}}\dif u={4B_{\alpha}}/{\eta},
	}
where the last step follows using \eqref{crit-vallue-deduction}.
Therefore, 
	\eqn{
	\lambda_c^{\sss \mathrm{IRG}}=\lim_{a\rightarrow \infty} \|\bfT_{\sss \kappa^{\sss (a)}}\|^{-1}
	=\lim_{a\rightarrow \infty} \frac{1}{\|\bfT_{\sss \kappa_2^{\sss (a)}}\|^{1/2}}
	=\sqrt{\frac{\eta}{4B_{\alpha}}}=\lambda_c,
	}
as required.
\end{proof} 

\subsection{Survival probability of the  multi-type branching process}
\label{sec-MBP}
In this section, we analyze the asymptotics in Proposition~\ref{prop-weight-giant-core} as $a\to\infty$. Recall the multi-type Poisson branching process $\cX_a^\lambda(u)$, and its survival probability $\rho_a^\lambda(u)$ from Section~\ref{sec:size-core-giant}. 
Recall the definition of $\zeta_a^\lambda$ from Proposition~\ref{prop-weight-giant-core}. 
The following is the main result of this section: 
\begin{proposition}[Large $a$ asymptotics of one-neighborhood giant] \label{prop-large-a-original}
For any $\lambda>\lambda_c$, as $a\to\infty$,
	\eqn{
	\label{a-zeta-a-lim}
	\zeta_a^\lambda := \lambda\int_0^a \cf u^{-\alpha} \rho_a^\lambda(u) \dif u \to \zeta^\lambda \in (0,\infty).
	} 
\end{proposition}
Before starting with the proof, we give some background on the object in \eqref{a-zeta-a-lim}. $\rho_a^\lambda(u)$ is the survival probability of a vertex of type $u$, which in the pre-limit corresponds to vertex $\lceil uN_n\rceil$. The factor $\cf u^{-\alpha}$ then corresponds to the rescaled version of $w_{\lceil uN_n\rceil},$ recall \eqref{eq:asymp-w-N}. Thus, $\zeta_a^\lambda$ can be viewed as the rescaled total weight or the rescaled size of the one-neighborhood of the giant in $[N_n(a)]$. Since, for $a$ large, this one-neighborhood is approximately the entire connected component of this giant in $[n],$ as shown  in Section \ref{sec-size-tiny-giant}, this explains the relevance of Proposition \ref{prop-large-a-original}.

We would like to stress some subtleties. First, $u\mapsto \cf u^{-\alpha}$ is not integrable, so we cannot think of $\zeta_\lambda$ as a survival probability of a branching process starting with a type chosen in a size-biased manner.
 Further, $\kappa(u,v)= 1- \e^{-\cf^2(uv)^{-\alpha}/\mu}$ is not integrable on  $((0,\infty), \dif x \otimes \dif y)$. 
As a result, we cannot express the limit of survival probabilities $(\rho_a^\lambda (u))_{u\geq 0}$ in terms of a maximum fixed point equation, as a survival probability would be expressed. 
This is reflected in the fact that the maximum solution of the previous fixed point equation $f = 1-\e^{-\bfT_\kappa f}$ is always~1 for non-integrable $\kappa$. 
However, the limit of $\zeta_a^\lambda$ still exists, and we can prove this using alternative arguments.

The proof is organised as follows. We start by stating an upper bound on our random graph in terms of an unpercolated Norros-Reittu model. This upper bound is also useful in Section~\ref{sec-size-tiny-giant}. Then, we perform a limiting argument on the survival probabilities to prove Proposition \ref{prop-large-a-original}.

\paragraph{Upper bound by an unpercolated  Norros-Reittu model.} We next discuss a 
Norros-Reittu model without percolation, which contains the graph $\cG_{\sss N_n(a)}$ as a subgraph.  
The nice thing about unpercolated Norros-Reittu models is that the total progeny can be coupled to a branching process as shown in \cite{NR06}, and it is possible to do direct computations on the limiting branching process as we will see below in Lemma~\ref{lem-surv-prob}.
This will be useful in showing finiteness of  limiting quantities such as $\zeta^\lambda$ in \eqref{a-zeta-a-lim}. 
Note that
	\eqn{
	\label{NR-bound}
	\percn p_{ij} =\percn[1-\e^{-w_iw_j/\ell_n}]\leq 1-\e^{-\percn w_iw_j/\ell_n}. 
	}
Indeed, the inequality in the second step of \eqref{NR-bound} is equivalent to the fact that, for every $p\in[0,1]$ and $x\geq 0$,
	\eqn{
	1-\e^{-x}\leq \frac{1}{p}[1-\e^{-px}].
	}
For $x=0$, both sides are equal. Differentiating with respect to $x$ gives that $\e^{-x}\leq \e^{-px}$,
which is true since $p\in[0,1], x\geq 0$. 
Now, recall the connection probabilities in the original model from~\eqref{eq:p-ij-NR-defn}. 
Then \eqref{NR-bound} shows that there exists a coupling such that $\rNR(\bw,\percn)$ is a subgraph of $\rNR(\percn\bw)$ with probability one. Henceforth, we will always work under this coupling.

For $\rNR(\percn\bw)$, it is known that, starting from any vertex $j$, the size of the connected component of $j$ can be bounded from above by the total progeny of a branching process, where the root has offspring distribution that is Poisson$(\percn w_j)$, while for all other vertices, the offspring distribution is mixed Poisson with mixing distribution $\percn W^{\star}_n$, where $W_n^\star$ has a size-biased distribution,  
i.e.,
	 \eqn{
	\label{size-biased-full}
	\PR(W_n^\star\leq x)=\frac{\sum_{i\in [n]}w_i\indic{w_i\leq x}}{\sum_{i\in [n]}w_i}.
	}
This is proved by Norros and Reittu in \cite{NR06}. 
Similar results can be proven when we restrict connected components to fixed subsets of~$[n]$, as we will frequently rely on below. 
In particular, we can use this observation to the restricted set $[N_n(a)]$ when considering the graph $\cG_{\sss N_n(a)}$. 
For this, we start by introducing some notation. For $A\subseteq [n]$, denote
	\eqn{
	w(A)=\sum_{a\in A}w_a.
	}
Then, we note that when restricting to $[N_n(a)]$, the parameter of the Poisson random variable of the root when starting from vertex $j\in [N_n(a)]$ is replaced with $\percn w_j w([N_n(a)])/\ell_n$, and that, for other vertices, the offspring becomes Poisson with mixing distribution
    \eqn{
    \label{W-lambda-def}
    W_{\sss [N_n(a)]}^\lambda:=\percn W^\star_{[N_n(a)]} w([N_n(a)])/\ell_n,
    }
where now
	 \eqn{
	\label{size-biased-[Nn]}
	\PR(W_{[N_n(a)]}^\star\leq x)=\frac{\sum_{i\in [N_n(a)]}w_i\indic{w_i\leq x}}{w([N_n(a)])}.
	}
This is formalized in the following lemma, which we state more generally, as we will rely upon it in various parts of the proof as well:
\begin{lemma}[Branching process upper bound on components restricted to subsets]
\label{lem-BP-UB-comp}
Let $A\subseteq [n]$, and consider the connected component of $\rNR(\percn \bw)$ of a vertex $j\in A$ restricted  to $A$. The size of this connected component is stochastically upper bounded by the total progeny of a mixed-Poisson branching process, where the root has Poisson offspring with parameter $\percn w_j w(A)/\ell_n$, and all other vertices have mixed-Poisson offspring with mixing distribution $\percn W_{A}^\star  w([N_n(a)])/\ell_n$ with
    \eqn{
	\label{size-biased-A}
	\PR(W_{A}^\star\leq x)=\frac{\sum_{i\in A}w_i\indic{w_i\leq x}}{w(A)}.
	}
\end{lemma}

\begin{proof}
Fix $A\subseteq [n]$. 
In $\rNR(\percn \bw)$, two vertices $i$ and $j$ with $i,j\in A$ share at least one edge  with probability $1-\e^{-\percn w_iw_j/\ell_n}$, and all edges are independent.
We now present another way to generate such independent edges.

For $j\in A$, we draw a Poisson random variable with parameter $\percn w_iw(A)/\ell_n$. We consider these to be the {\em potential neighbors} of $j$. Then we assign a label to each of these potential neighbors, and this label equals $i$ with probability
    \eqn{
    q_A(i)=\frac{w_i}{w(A)}, \qquad i\in A.
    }
Retain an edge between $i$ and $j$ when there is at least one potential neighbor of $j$ with label~$i$. Then, for fixed $j$, the numbers of neighbors with label $i$ are {\em independent} Poisson random variables with parameters
    \eqn{
    q_A(i)\frac{\percn w_jw(A)}{\ell_n}
    =\frac{\percn w_iw_j}{\ell_n},
    }
so that the probability that there is at least one potential neighbor with label $j$ equals $1-\e^{-\percn w_iw_j/\ell_n}$, as required.

The above shows how the neighbors of a vertex $i$ can be chosen. In order to obtain the stochastic upper bound on the connected components in Lemma~\ref{lem-BP-UB-comp}, we explore the connected component in a breadth-first way. Then, it follows that the connected components with edge probabilities $1-\e^{-\percn w_iw_j/\ell_n}$ are obtained through a {\em thinning} of the above construction, where vertices in the tree are ordered in the breadth-first manner, and repetitions of the labels (as well as all their offspring) are removed.

Finally, we note that the above process of potential neighbors is a Poisson branching process with mixing distribution given by $W_{A}^\star$ in \eqref{size-biased-A}. Indeed, we explore a single potential neighbors by {\em first} drawing its mark, and, given that its mark equals $i$, drawing a Poisson random variable with parameter $\percn w_iw(A)/\ell_n$ of potential neighbors. Then, the collection of potential neighbors (which includes the percolation component, due to the thinning) is a mixed-Poisson branching process where the root (which corresponds to the vertex with label $j$) has a Poisson offspring with parameter $\percn w_jw(A)/\ell_n$, while all other vertices have offspring of a mixed-Poisson distribution with mixing parameter $W_{A}^\star$ in \eqref{size-biased-A}. Thus the proof of Lemma~\ref{lem-BP-UB-comp} is complete.
\end{proof}

Next, let us investigate the survival probabilities of the above branching process for $A=[N_n(a)]$. Let $\bar{\rho}_{n,a}^{\star,\lambda}$ denote the survival probability of the above branching process with root also having the mixed Poisson offspring distribution with $W_{\sss [N_n(a)]}^\lambda$ in \eqref{W-lambda-def}. 
Also, let $\bar{\rho}_{n,a}^\lambda(u)$ denote the survival probability when we start with vertex $j=\lceil uN_n\rceil$. 
The following lemma investigates the asymptotics of these survival probabilities when $n\rightarrow \infty$:

\begin{lemma}[Survival probability for upper bounding branching process]
\label{lem-surv-prob}
For any $\lambda>\lambda_c(a)$, as $n\to\infty$, $\bar{\rho}_{n,a}^{\star,\lambda}\to \bar{\rho}_a^{\star,\lambda}$, where $\bar{\rho}_a^{\star,\lambda}$ is the maximum solution satisfying
	\begin{eq}
	\label{rho-a-def}
	 \bar{\rho}_a^{\star,\lambda}=  (1-\alpha)a^{\alpha-1}\int_0^a  u^{-\alpha} [1-\e^{-\lambda \bar{c}_{\sss F} u^{-\alpha} a^{1-\alpha} \bar{\rho}_a^{\star,\lambda}}]\dif u,
	\end{eq}
with $\bar{c}_{\sss F}=\frac{ \cf^2}{(1-\alpha)\mu}.$
Moreover, for all sufficiently large $n$,
    \begin{eq}
    \label{eq:rho-bar-u}
    \bar{\rho}_{n,a}^\lambda(u) \leq  C \min\{1,u^{-\alpha}\},
    \end{eq}
for some constant $C = C(\lambda,\alpha) >0$ independent of $a$.
\end{lemma}

\begin{proof} 
We write 
	\eqn{
	\widetilde{w}_i=\frac{w_i\percn w([N_n(a)])}{\ell_n}.
	}
 Note that $\E[(1-t)^X] = \e^{-ct}$ for $X\sim \mathrm{Poisson}(c)$.
Now, conditioning on the type of the root, the branching process dies out precisely when all the progeny of generation one dies out. Equating these probabilities, we get
	\eqn{
	\label{surv-prob-comp}
	\bar{\rho}_{n,a}^{\star,\lambda}=1-\sum_{i\in [N_n(a)]}\frac{w_i}{w([N_n(a)])}\e^{-\widetilde{w}_i \bar{\rho}^{\star,\lambda}_{n,a}}
	=\sum_{i\in [N_n(a)]}\frac{w_i}{w([N_n(a)])}[1-\e^{-\widetilde{w}_i \bar{\rho}^{\star,\lambda}_{n,a}}].
	}
Recalling $N_n$ from \eqref{Nn1-def}, we rewrite the sum in an integral to obtain
	\eqn{ \label{surv-prob-star-n}
	\bar{\rho}_{n,a}^{\star,\lambda}=N_n \int_0^a \frac{w_{\lceil uN_n\rceil }}{w([N_n(a)])}[1-\e^{-\widetilde{w}_{\lceil uN_n\rceil } \bar{\rho}^{\star,\lambda}_{n,a}}]\dif u.
	}		
We further simplify 
	\eqn{ \label{eq:sum-weight-N-a}
	w([ N_n(a) ] ) =   \cf\sum_{j=1}^{N_n(a)} (n/j)^{\alpha}\asymp \frac{\cf}{1-\alpha} n^{\alpha} N_n(a)^{1-\alpha} =  \frac{\cf}{1-\alpha} \sqrt{n} N_n a^{1-\alpha},
	}
while $w_{\lceil uN_n \rceil} \asymp \cf \sqrt{n} u^{-\alpha}$ by \eqref{eq:asymp-w-N}.
We then conclude that 
	\eqn{\label{w-asymp-tild}
	\tilde{w}_{\lceil uN_n \rceil }=w_{\lceil uN_n\rceil} \percn  \frac{w([N_n(a)])}{\ell_n}
	\asymp \frac{\cf\sqrt{n}}{u^{\alpha}} \bigg(\frac{\cf}{1-\alpha} \sqrt{n} N_n a^{1-\alpha}\bigg)\frac{\percn}{\mu n}=\lambda \bar{c}_{\sss F} u^{-\alpha}  a^{1-\alpha}.
	}
Thus, by \eqref{size-biased-[Nn]}, and using $\tilde{w}_i$ defined in \eqref{w-asymp-tild},
	\eqan{
	\label{CDF-converge}
	\PR(W_{\sss [N_n(a)]}^\lambda\leq x) &= \frac{1}{w([N_n(a)])} \sum_{i\in N_n(a)} w_i \ind{\tilde{w}_i \leq x} \asymp \frac{N_n}{w([N_n(a)])}  
	\int_0^a w_{\lceil uN_n \rceil }\ind{\tilde{w}_{\lceil uN_n \rceil} \leq x} \dif u \\
	&\asymp (1-\alpha) a^{\alpha-1} \int_0^a u^{-\alpha} \ind{\lambda \bar{c}_{\sss F} u^{-\alpha} a^{1-\alpha} \leq x} \dif u.\nn
	}
Let $W_{\infty,a}^\lambda$ be a random variable with distribution function given by the right hand side of \eqref{CDF-converge}. Then $W_{\sss [N_n(a)]}^\lambda\xrightarrow{\sss d} W_{\infty,a}^\lambda$. 
Thus, if $\bar{\rho}_a^{\star,\lambda}$ denotes the survival probability of the branching process with starting distribution and progeny distribution given by a mixed-Poisson random variable with parameter $W_{\infty,a}^\lambda$, also
    \begin{eq}
    \bar{\rho}_{n,a}^{\star,\lambda} \to \bar{\rho}_a^{\star,\lambda}.
    \end{eq}
We conclude \eqref{rho-a-def} by taking limit as $n\to\infty$ in \eqref{surv-prob-star-n}. 

For \eqref{eq:rho-bar-u}, let us start with vertex $j=\lceil uN_n\rceil$. 
The limit of 
$\bar{\rho}_{n,a}^\lambda(u)$ exists using \eqref{CDF-converge}. 
 By the fact that the branching process is i.i.d.~after the first generation, using a union bound,  this survival probability is at most the expected offspring of $j=\lceil uN_n\rceil$ times $\bar{\rho}^{\star,\lambda}_{n,a}$.  The expected offspring is 
	\eqn{
	\percn w_{\lceil uN_n \rceil}  \frac{w_{[N_n(a)]}}{\ell_n}= \frac{\cf^2}{\mu} u^{-\alpha} \percn n^{-1/2} \sum_{j=1}^{N_n(a)} \Big(\frac{n}{j}\Big)^{\alpha}=C u^{-\alpha} \percn N_n(a)^{1-\alpha} n^{\alpha-1/2}=C u^{-\alpha} a^{1-\alpha}.
	}
Thus, for all sufficiently large $n$,
\begin{eq}\label{upper-bound-sur-u-dep}
\bar{\rho}_{n,a}^\lambda(u)\leq  C u^{-\alpha} a^{1-\alpha} \bar{\rho}^{\star,\lambda}_a. 
\end{eq}
The proof of \eqref{eq:rho-bar-u} follows if we can show that $\limsup_{a\to\infty}a^{1-\alpha} \bar{\rho}^{\star,\lambda}_a <\infty$. 
Using \eqref{rho-a-def}, and writing $z_a=a^{1-\alpha} \bar{\rho}^{\star,\lambda}_a $, 
\begin{eq}
z_a&=  (1-\alpha)\int_0^a  u^{-\alpha} [1-\e^{-\lambda \bar{c}_{\sss F} u^{-\alpha} z_a}]\dif u \leq (1-\alpha)\int_0^a  u^{-\alpha} \min\{1,\lambda \bar{c}_{\sss F} u^{-\alpha} z_a\}\dif u 
\\
&\leq C_1 \bigg[\int_0^{C_2 z_a^{\frac{1}{\alpha}}} u^{-\alpha} \dif u+z_a\int_{C_2 z_a^{\frac{1}{\alpha}}}^{\infty} u^{-2\alpha} \dif u\bigg] \leq C z_a^{\frac{1-\alpha}{\alpha}}.
\end{eq}
Since $\frac{1-\alpha}{\alpha}<1$, it follows that  $\limsup_{a\to\infty} z_a = \limsup_{a\to\infty}a^{1-\alpha} \bar{\rho}^{\star,\lambda}_a <\infty$. The proof of \eqref{eq:rho-bar-u} is now completed using~\eqref{upper-bound-sur-u-dep}. 
\end{proof}

\begin{proof}[Proof of Proposition~\ref{prop-large-a-original}] 
First, note that $\zeta^\lambda_a$ is non-decreasing in $a$. 
Indeed, for $b>a$, there exists a coupling under which $\cG_{\sss N_n(a)} \subset \cG_{\sss N_n(b)}$. Under this coupling, by Proposition~\ref{prop:giant-restricted}, the component of $\cG_{\sss N_n(b)}$ containing $\sC^a_{\sss (1)}$ has size $\thetaP(N_n)$, and since $\sC^b_{\sss (2)} = \OP(\log n)$, it must be the case that $\sC^a_{\sss (1)}\subset \sC^b_{\sss (1)}$ with high probability. 
By Proposition~\ref{prop-weight-giant-core}, it now follows that $\zeta_a^\lambda \leq \zeta_b^\lambda$.  
Thus $\lim_{a\to\infty} \zeta^\lambda_a$ exists and is positive. 

Next, using \eqref{NR-bound}, we have that, for any $a>0$, $\rho_a^\lambda (u) \leq \limsup_{n\to\infty} \bar{\rho}^\lambda_{n,a} (u)$. Therefore, an application of \eqref{eq:rho-bar-u} yields that 
\begin{eq}
\lim_{a\to\infty}\zeta_a^\lambda \leq C \int_0^\infty  u^{-\alpha} \min\{1, u^{-\alpha}\} \dif u<\infty,
\end{eq}
and the proof of Proposition~\ref{prop-large-a-original} follows. 
\end{proof}

\section{Size of the tiny giant}
\label{sec-size-tiny-giant}
In this section, we complete the proof of Theorem~\ref{thm:supcrit-bd}.
To this end, fix $\lambda>\lambda_c$. 
By Lemmas~\ref{lem-mon-lambdaca} and~\ref{lem-lambdacs-equal},   $\lambda>\lambda_c(a)$ for all sufficiently large $a$. Therefore, by Proposition~\ref{prop:giant-restricted},
the graph restricted to $[N_n(a)]$ has a giant component $\sC_{\sss (1)}^{a}$ of approximate size $N_n(a) \rho_{a}^\lambda$, where $\rho_{a}^\lambda>0$.
We denote by $\sC_{\sss (i)}^{a,\star}$ the component of $\rNR(\bw,\percn)$ containing $\sC_{\sss (i)}^{a}$.
The main idea is to show that the component $\sC_{\sss (1)}^{a,\star}$ is the unique giant component $\sC_{\sss (1)}$ of $\rNR(\bw,\percn)$, in the iterated limit as first $n\to\infty$, followed by $a\to\infty$. 

Let us now explain in more detail how we aim to approach the proof. For $j \in [N_n(a)]$, define
$\spn (\{j\})$ to be the set of vertices  $v\in [N_n(a)]^c$ such that there exists a path between $j$ and  $v$ that lies entirely in $[N_n(a)]^c$.
For $V\subseteq [N_n(a)]$, we write $\spn (V) = \cup_{j\in V} \spn (\{j\})$. 
We have that $\spn(\sC_{\sss (1)}^{a})\subseteq \sC_{\sss (1)}^{a,\star}$, but $\sC_{\sss (1)}^{a,\star}$ may be larger since $\spn(\sC_{\sss (1)}^{a})$ may intersect with $\spn(\sC_{\sss (j)}^{a})$ for some $j\geq 2$, in which case $\sC_{\sss (1)}^{a,\star}$ gets merged with $\sC_{\sss (j),*}^a$.
To study the effect of such mergers, let us say that there is a \emph{return path} between $i,j\in [N_n(a)]$ if a path exists between $i$ and $j$ with at least one intermediate vertex in $[N_n(a)]^c$. 
In other words, the existence of a return path between $i,j\in [N_n(a)]$ means that $i,j$ become part of the same component  only after adding the edges in $[N_n(a)]^c$.

Let $\cR_{\sss (1)}^a$ denote the set of vertices $v\in [N_n(a)]$ such that $v$ is  connected to some $j\in \sC_{\sss (1)}^{a}$ only via a return path.
Then,
	\eqn{
	\label{giant-equality}
	\sC_{\sss (1)}^{a,\star}=\sC_{\sss (1)}^{a} \cup \spn(\sC_{\sss (1)}^{a}) \cup \cR_{\sss (1)}^a \cup \spn(\cR_{\sss (1)}^a).
	}
Our objective is to show that, for $\lambda>\lambda_c(a)$ and large enough $a$, the main contribution in $|\sC_{\sss (1)}^{a,\star}|$ comes from $\spn(\sC_{\sss (1)}^{a}) $. An important ingredient to such a proof is that  $|\spn(\sC_{\sss (1)}^{a}) |$ is asymptotically close to the size of the one-neighborhood of $\sC_{\sss (1)}^a$ (see also Proposition \ref{prop-large-a-original} and the intuition below it).

The remainder of this section is organised as follows:
We start by proving a lower and an upper bound on the
span of $\sC_{\sss (1)}^a$ in Sections~\ref{sec-lower-bound}~and~\ref{sec:span-comp} respectively.  
In Section~\ref{sec:return- path}, we show that the contributions to the spans due to return paths is asymptotically negligible. 
In fact, we will show that the span of small subsets of vertices is small uniformly over the choice of the vertex sets (see Lemma~\ref{prop:uniform-small-span}).
In Section~\ref{sec-no-outside-Nn} we show that with high probability there is no large component outside of $[N_n(a)]$.
We conclude with the proof of Theorem \ref{thm:supcrit-bd} in Section \ref{sec:proof-main-thm}.

\subsection{Concentration of the spans: lower bound}
\label{sec-lower-bound}
Fix $a>0$ and recall the definitions of  $\sCa_{\sss (i)}$ from Section~\ref{sec:small-giant}, and that of $\spn(V)$ for $V\subseteq [N_n(a)]$ above \eqref{giant-equality}.
In this section, we obtain a lower bound on the asymptotic size of  $\spn(\sC_{\sss (1)}^{a})$, by proving a sharp approximation for the  1-neighborhood of $\sC_{\sss (1)}^{a}$:

\begin{proposition}[Lower bound for the span]
\label{prop-LB-size-span-C1}
Fix $\lambda>\lambda_c$. 
For any $\vep >0$, there exists $a_1 = a_1(\vep)>0$ such that, for all $a\geq a_1$,
\begin{eq}
\lim_{n\to\infty} \PR\bigg( \frac{|\spn(\sC_{\sss (1)}^{a})|}{\sqrt{n}} \geq \zeta^\lambda - \vep \bigg) = 1,
\end{eq}
where $\zeta^\lambda$ is as in \eqref{a-zeta-a-lim}.
\end{proposition}

For $V\subseteq [N_n(a)]$, let $\mathcal{N}_l(V)$ denote the vertices in $\spn(V)$ that are at distance $l$ from $V$, and let $\mathcal{N}_{\sss \geq l}(V) = \cup_{l'\leq l} \mathcal{N}_{l'}(V)$.
Thus,
	\eqn{ 
	\label{eq:split-span}
	\spn(V) = {\mathcal N}_1(V)\cup {\mathcal N}_{\sss \geq 2}(V).
	}
Lemma~\ref{prop:2nbd-giant} below identifies the asymptotics of the first term in \eqref{eq:split-span}. 
In the next section, 
where we analyze the upper bound on $|\spn(\sC_{\sss (1)}^{a})|$,  
we show that the second term in \eqref{eq:split-span} gives a negligible contribution (see Lemma~\ref{prop:3nbd-negligible} below), but this is not needed for the lower bound in Proposition \ref{prop-LB-size-span-C1}:

\begin{lemma}[Direct neighbors of {$[N_n(a)]$}]
\label{prop:2nbd-giant}
Let $V\subseteq [N_n(a)]$ be such that $\sum_{i\in V} \percn w_i \geq c_0 \sqrt{n}$ for some constant $c_0>0$.  Then, for any fixed $a>0$, and $\vep >0$, as $n\to\infty$,
	\begin{eq}
	\label{eq:1-nbd-prob-convergence}
	\PR\bigg(\Big| |{\mathcal N}_1(V)| - \sum_{i\in V} \percn w_i \Big| > \vep\sqrt{n} \ \bigg\vert\ \cG_{\sss N_n(a)} \bigg) \pto 0. 
	\end{eq} 
\end{lemma}

\begin{proof}
Let $\PR_1$ and $\E_1$ denote the conditional probability and expectation, respectively, conditionally on $\cG_{\sss N_n(a)}$. 
Let us first show that 
	\begin{eq}\label{eq:expt-weight-main-contribution}
	\E_1\big[\big|{\mathcal N}_1(V)\big|\big] =   (1+o(1))\sum_{i\in V}  \percn w_i+o(\sqrt{n}).
	\end{eq}
Note that 
	\eqn{\label{one-neighbor-sum-indic}
	\big|{\mathcal N}_1(V)\big|= \sum_{j\notin [N_n(a)]} \ind{(i,j) \text{ create an edge for some }i\in V}.
	}
Thus, by a union bound,  
    \begin{eq}
    \label{eq:sim-split-01}
    \E_1\big[\big|{\mathcal N}_1(V)\big|\big] \leq \sum_{i\in V}  \sum_{j\in [n]} \percn \big( 1- \e^{-w_iw_j/\ell_n}\big).
    \end{eq}
Moreover, using inclusion-exclusion,  the expectation in \eqref{eq:expt-weight-main-contribution} is at least 
	\begin{eq}\label{eq:sim-split-1}
	&\sum_{i\in V}  \sum_{j\notin [N_n(a)]} \percn \big( 1- \e^{-w_iw_j/\ell_n}\big)   - \sum_{i_1,i_2\in V } \sum_{j\notin [N_n(a)]} \percn^2 \big( 1- \e^{-w_{i_1}w_j/\ell_n}\big)\big( 1- \e^{-w_{i_2}w_j/\ell_n}\big).
	\end{eq}
Now, by \eqref{eq:sum-weight-N-a},  $\sum_{i\in V} w_i \leq \sum_{i\in [N_n(a)]} w_i \leq C a^{1-\alpha} N_n \sqrt{n}$, and thus, using $1-\e^{-x} \leq x$,  the second term is at most 
    \begin{eq}
    \percn^2 \bigg(\sum_{i\in V} w_i\bigg)^2 \frac{1}{\ell_n^2} \sum_{j\notin [N_n(a)]} w_j^2&\leq Ca^{2-2\alpha}\percn^2 (\sqrt{n} N_n)^2 n^{-2+2\alpha} \sum_{j>N_n(a)} j^{-2\alpha} \\
    &\leq C a^{3-4\alpha }n^{(3-\tau)/2} = o(\sqrt{n}).
    \end{eq}
Moreover, 
    \begin{eq}
    \sum_{i\in V} \sum_{j\in [N_n(a)]} \percn \big( 1- \e^{-w_iw_j/\ell_n}\big) \leq \percn N_n(a)^2 \leq C a^2 n^{(3-\tau)/2} = o(\sqrt{n}).
    \end{eq}
Thus, \eqref{eq:sim-split-01} and \eqref{eq:sim-split-1} together imply that 
    \begin{eq}
    \label{eq:one-step-simple}
    \E_1\big[\big|{\mathcal N}_1(V)\big|\big] =  \sum_{i\in V} \sum_{j\in [n]}\percn  \big( 1- \e^{-w_iw_j/\ell_n}\big) + o(\sqrt{n}).
    \end{eq}
Let $\varepsilon_n$ be such that $\vep_n\searrow 0$ sufficiently slowly (to be specified later). 
Let us split the first term of \eqref{eq:one-step-simple} in two parts by restricting the sum over $j\in[n]$ to $\{j\colon w_iw_j\leq \varepsilon_n \ell_n\}$ and $\{j\colon w_iw_j >  \varepsilon_n \ell_n\}$, respectively. Denote the two terms by $\mathrm{(I)}$ and $\mathrm{(II)}$, respectively.
Note that 
	\begin{eq}
	\mathrm{(II)} &\leq \percn \sum_{i\in [N_n(a)]}\#\{j\colon w_iw_j \leq \varepsilon_n\ell_n\} = C\percn \varepsilon_n^{-(\tau -1)} n^{2-\tau} \sum_{i\in [N_n(a)]} w_i^{\tau-1}\\
	& = Cn^{(3-\tau)/2} \varepsilon_n^{-(\tau -1)} \sum_{i\leq an^{(3-\tau)/2}} \frac{1}{i} = Cn^{(3-\tau)/2} \varepsilon_n^{-(\tau -1)}\log(a n^{(3-\tau)/2}) =o(n^{1/2}),
	\end{eq}
where in the second step we have used Lemma~\ref{lem:order-estimates}, and the choice of $\vep_n$ is such that the final step holds.
Moreover, since $1-\e^{-x} = x(1+o(1))$ as $x\to 0$,
	\begin{eq}
	\mathrm{(I)} &\geq (1+o(1))\sum_{i\in V} \sum_{j\colon w_j \leq \varepsilon_n\ell_n/w_i} \percn \frac{w_iw_j}{\ell_n} 
	\\
	&\geq (1+o(1))\sum_{i\in V} \sum_{j\colon w_j \leq C\varepsilon_n n^{\rho}} \percn \frac{w_iw_j}{\ell_n} = (1+o(1))\sum_{i\in V} \percn w_i .
	\end{eq}
Also, $\mathrm{(I)} \leq \sum_{i\in V} \percn w_i $. 
We conclude that
	\begin{eq}\label{eq:sim-split-3}
	\sum_{i\in V} \sum_{j\in [n]} \percn  \big( 1- \e^{-w_iw_j/\ell_n}\big) = (1+o(1))\sum_{i\in V} \percn w_i + o(\sqrt{n}), 
	\end{eq}
and thus \eqref{eq:expt-weight-main-contribution} follows by combining  \eqref{eq:one-step-simple} and \eqref{eq:sim-split-3}.

To complete the proof of \eqref{eq:1-nbd-prob-convergence}, we apply Chebyshev's inequality for which we need to bound the variance of $|{\mathcal N}_1(V)|$. Let $\mathrm{Var}_1$ denote the variance conditionally on  $\cG_{\sss N_n(a)}$.
Note that \eqref{one-neighbor-sum-indic} is a sum of conditionally independent indicators, given $\cG_{\sss N_n(a)}$. 
Therefore,
	\begin{eq}
	\mathrm{Var}_1\big(\big|{\mathcal N}_1 (V)\big| \big) \leq \E_1\big[\big|{\mathcal N}_1 (V)\big|  \big], 
	\end{eq}
and an application of Chebyshev's inequality completes the proof. 
\end{proof}
Now we are ready to complete the proof of Proposition~\ref{prop-LB-size-span-C1}:

\begin{proof}[Proof of Proposition~\ref{prop-LB-size-span-C1}]
Clearly, $|\spn(V)| \geq |\cN_1(V)|$ by \eqref{eq:split-span}. 
We apply Lemma \ref{prop:2nbd-giant} with $V=\sC_{\sss (1)}^{a}$, and rely on Proposition \ref{prop-weight-giant-core} to estimate $|\cN_1(\sC_{\sss (1)}^{a})|$. Finally, by Proposition~\ref{prop-large-a-original}, we can take $a>1$ sufficiently large, so that $\zeta_a^\lambda\geq \zeta^\lambda-\vep/2$.
Thus, Proposition~\ref{prop-LB-size-span-C1} follows.
\end{proof}

\subsection{Concentration of the spans: upper bound}
\label{sec:span-comp}
Fix $a>0$ and recall the definition of $\cT_{\sss \geq k}^a$ from Section~\ref{sec:small-giant}, and that of $\spn(V)$ for $V\subseteq [N_n(a)]$ above \eqref{giant-equality}.
In this section, we obtain an upper bound on $\spn(\cT_{\sss \geq k}^a)$.

\begin{proposition}[Upper bound on the span of large clusters]
\label{prop:size-span-C1}
Fix $\lambda>\lambda_c$. 
For any $\vep >0$, there exists $a_1 = a_1(\vep)>0$ such that for all $a\geq a_1$ there exists $k_0 = k_0(\vep,a)$ such that, for all $k\geq k_0$, 
\begin{eq}
\lim_{n\to\infty} \PR\bigg( \frac{|\spn(\cT_{\sss \geq k}^a)|}{\sqrt{n}} \leq \zeta^\lambda + \vep \bigg) = 1,
\end{eq}
where $\zeta^\lambda$ is as in \eqref{a-zeta-a-lim}.
\end{proposition}

Together with Proposition~\ref{prop-LB-size-span-C1}, Proposition~\ref{prop:size-span-C1} provides the following law of large numbers on 
$\spn(\sC_{\sss (1)}^{a})$:

\begin{corollary}[Law of large numbers for $\spn(\sC_{\sss (1)}^{a})$]
\label{cor-LLN-span-C1}
Under the conditions of {\rm Proposition \ref{prop:size-span-C1}}, with high probability
$\zeta^\lambda - \vep\leq |\spn(\sC_{\sss (1)}^{a})|/\sqrt{n}\leq \zeta^\lambda + \vep$.
\end{corollary}

Our goal will be to first show that, given any arbitrary $V\subseteq [N_n(a)]$, the $\spn(V)$ is predominantly carried by the one-neighborhood of $V$, when $a$ is large. 
Recall the notation $\cN_l(V)$, $\mathcal{N}_{\sss \geq 2}(V)$ before \eqref{eq:split-span}, and that $\spn(V) = \mathcal{N}_1(V)\cup \mathcal{N}_{\sss \geq 2}(V)$. 
Lemma \ref{prop:2nbd-giant} has studied ${\mathcal N}_1(V)$ in detail, and now we focus on studying ${\mathcal N}_{\sss \geq 2}(V)$ for $a>1$ large:
\begin{lemma}[Additional neighborhood of $\protect{[N_n(a)]}$]
\label{prop:3nbd-negligible} 
Let $V\subseteq [N_n(a)]$ be such that $\sum_{i\in V} \percn w_i \leq C_0 \sqrt{n},$
for some constant $C_0>0$ (independent of $a$). Then, for any $\varepsilon>0$, there exists $a_0 = a_0(\vep)>0$ such that for any $a>a_0$, as $n\to\infty$,
	\eqn{
   	 \PR\big(|{\mathcal N}_{\sss \geq 2}(V)|>\varepsilon \sqrt{n} \ \big| \ \cG_{\sss N_n(a)} \big) \pto 0.
	} 
\end{lemma}
\begin{proof}
Recall that $\PR_1$ and $\E_1$ denote the conditional probability and expectation, respectively, given $\cG_{\sss N_n(a)}$. We first show that there exists $a_0 = a_{0}(\vep)>0$ such that for all $a>a_0$,
	\begin{eq}
	\label{asym-path-i-j}
	\frac{\E_1[|{\mathcal N}_{\sss \geq 2}(V)|]}{\sqrt{n}} < \frac{\vep}{2}, \quad \text{with high probability.}
	\end{eq}
For any $i,j\in [n]$,  let $\cA_{l}(i,j)$ denote the event that there exists $l-1$ vertices $i_1,\dots, i_{l-1}\in [N_n(a)]^c$ such that 
$(i, i_1, \dots, i_{l-1},  j)$ is a path in $\rNR(\bw,\percn)$.
In words, $\cA_{l}(i,j)$ is the event that there exists a path of length $l$ between $i$ and $j$ in $\rNR(\bw,\percn)$ with all the intermediate vertices in~$[N_n(a)]^c$.
Now, with $i_0=i, i_{l}=j$, note that 
	\begin{eq} \label{eq:path-count prob-large-small}
	\PR_1(\cA_l(i,j))&\leq \sum_{i_1, \ldots, i_{l-1}\notin [N_n(a)]} \prod_{s=1}^l \frac{\percn w_{i_{s-1}}w_{i_s}}{\ell_n}
	\\
	&\leq \percn \frac{w_iw_j}{\ell_n} \Big(\sum_{v\notin [N_n(a)]}  \frac{\percn w_{v}^2}{\ell_n}\Big)^{l-1}
	=\percn \frac{w_iw_j}{\ell_n} \bar{\nu}_n(a)^{l-1},	
	\end{eq}
where 
	\begin{eq}
	\label{bar-nu-a-bd}
	\bar{\nu}_n(a) &= \frac{1}{\ell_n} \sum_{v \notin [N_n(a)]} \percn w_v^2 = \frac{Cn^{2\alpha}\percn}{n} \sum_{v> an^{(3-\tau)/2}} v^{-2\alpha} \\
	&= \frac{Cn^{2\alpha}\percn}{n}\big(an^{(3-\tau)/2}\big)^{1-2\alpha}  = Ca^{-(3-\tau)/(\tau-1)}.\nn
	\end{eq}
Thus, using $\sum_{i\in V} \percn w_i \leq C_0 \sqrt{n}$,
	\eqan{
	\label{eq:neighborhood-sum-split}
	\E_1[|{\mathcal N}_{\sss \geq 2}(V)|] 
	&\leq \sum_{l\geq 2} \sum_{i\in V}\sum_{j\notin [N_n(a)]} \PR_1(\cA_{l}(i,j))\\
	&\leq C\bar{\nu}_n(a) \frac{\percn}{\ell_n}\sum_{i\in V} w_i \sum_{j\notin [N_n(a)]} w_j \leq C\bar{\nu}_n(a) \sqrt{n},\nn
	}    
and \eqref{asym-path-i-j} follows using \eqref{bar-nu-a-bd}.

Next, we compute $\Var_1(|{\mathcal N}_{\sss \geq 2}(V)|)$, where $\Var_1$ denotes the conditional variance given $\cG_{\sss N_n(a)}$. Let $I_{ij}(a)$ be the indicator of the event $\cup_{l\geq 2} \cA_l(i,j)$.
Note that 
\begin{eq}\label{eq:n-geq-2-2ndmoment}
\E_1[|{\mathcal N}_{\sss \geq 2}(V)|^2] \leq  \sum_{i_1,i_2\in V} \sum_{j_1,j_2\in [N_n(a)]^c} \PR_1(I_{i_1j_1}(a) = 1,I_{i_2j_2}(a) = 1). 
\end{eq}
We split the sum over possible choices of $i_1,i_2,j_1,j_2$. 
If $i_1= i_2$ and $j_1 = j_2$, then we get the same bound as in \eqref{eq:neighborhood-sum-split}.
Let $i_1 = i_2=i$ and $j_1\neq j_2$. If $I_{ij_1}(a) = 1$ and $I_{ij_2}(a) = 1$, then we have two cases. 
\begin{itemize}
    \item[$\rhd$] \textbf{Case 1:} There are two vertex-disjoint paths $[i,j_1]$ and $[i,j_2]$ with all intermediate vertices in $[N_n(a)]^c$. 
    \item[$\rhd$] \textbf{Case 2:} There exists a vertex $k \in [N_n(a)]^c$ such that there are three vertex-disjoint paths $[i,k]$, $[k,j_1]$ and $[k,j_2]$ with all intermediate vertices in $[N_n(a)]^c$. 
\end{itemize}
Since the paths described above are vertex-disjoint, we can apply the BK-inequality \cite[Theorem 3.3]{BK85}. 
Let $\cA(i,k) = \cup_{l\geq 1} \cA_l(i,k)$. 
By \eqref{eq:path-count prob-large-small}, $\PR_1(\cA(i,k)) \leq C\percn w_i w_k/\ell_n $ for any $i,k$. 
Thus, when Case 1 occurs, we can bound the term in \eqref{eq:n-geq-2-2ndmoment} by 
\begin{eq}\label{BK-appl-two-path}
&\sum_{i\in V}\sum_{j_1,j_2\in [N_n(a)]^c} \PR_1(I_{ij_1}(a) =I_{ij_2}(a) = 1, \text{ and Case 1 occurs})\\
& \leq \sum_{i\in V}\sum_{j_1,j_2\in [N_n(a)]^c} \PR_1(\cA(i,j_1)) \PR_1(\cA(i,j_2)) \leq \frac{C\percn^2}{\ell_n^2} \sum_{i\in V} w_i^2 \bigg(\sum_{j\in [N_n(a)]^c} w_j \bigg)^2\\
& \leq C\percn^2   \sum_{i\in [n]} w_i^{2}\leq C (n^\alpha \percn)^2 = o(n).
\end{eq}
Again, using an union bound over the choices of $k$ and applying  the BK-inequality, we obtain
\begin{eq}\label{BK-multipath-N-geq-2}
&\sum_{i\in V}\sum_{j_1,j_2\in [N_n(a)]^c}\PR_1(I_{ij_1}(a)= I_{ij_2}(a) = 1, \text{ and Case 2 occurs})\\
&\leq \sum_{i\in V}\sum_{j_1,j_2\in [N_n(a)]^c} \sum_{k\in [N_n(a)]^c}\PR_1(\cA(i,k)) \PR_1(\cA(j_1,k))\PR_1(\cA(j_2,k)) \\ 
&\leq \frac{C\percn^3}{\ell_n^3} \sum_{i\in V}\sum_{j_1,j_2\in [N_n(a)]^c}\sum_{k\in [N_n(a)]^c} w_iw_{j_1} w_{j_2} w_k^3  \leq C\sqrt{n}\frac{\percn^2}{\ell_n}\sum_{k\in [N_n(a)]^c}  w_k^3 \leq C a^{1-3\alpha} n \percn,
\end{eq}
where in the one-but-last step we have used our assumption that $\sum_{i\in V} \percn w_i \leq C_0\sqrt{n}$, and the final step follows by using 
    \eqan{
    \label{sum-wk3}
    \sum_{k\in [N_n(a)]^c}  w_k^3 
    &\leq \sum_{k>N_n(a)}  \frac{ \cf^3 n^{3\alpha}}{k^{3\alpha}}=
    Cn^{3\alpha} (a N_n)^{1-3\alpha} = C a^{1-3\alpha} N_n n^{3/2}.
    }
We can similarly treat the case $i_1\neq i_2$ and $j_1 = j_2$.
In that case, we no longer have to split in two cases as above, since $k$ may be equal to $j$. Thus, the same argument as \eqref{BK-multipath-N-geq-2} shows that
\begin{eq}\label{BK-multipath-N-geq-2-2}
&\sum_{i_1,i_2\in V}\sum_{j\in [N_n(a)]^c}\PR_1(I_{i_1j}(a) = I_{i_2j}(a) = 1) \leq C a^{1-3\alpha} \sqrt{n}.
\end{eq}
Note also that $|{\mathcal N}_{\sss \geq 2}(V)|= \sum_{j\notin [N_n(a)]} \ind{\cA(i,j) \text{occurs for some }i\in V},$ and thus \eqref{BK-multipath-N-geq-2-2} also implies that
\begin{eq} \label{eq:lb-geq-2}
&\E_1\big[|{\mathcal N}_{\sss \geq 2}(V)|\big] \\
&\geq \sum_{i\in V} \sum_{j\in [N_n(a)]^c} \PR_1(I_{ij}(a) = 1) - \sum_{i_1,i_2\in V} \sum_{j\in [N_n(a)]^c} \PR_1(I_{i_1j}(a) = I_{i_2j}(a) = 1)\\
&=\sum_{i\in V} \sum_{j\in [N_n(a)]^c} \PR_1(I_{ij}(a) = 1)-o(\sqrt{n}). 
\end{eq}

Next, consider the case where $i_1,i_2,j_1,j_2$ are all distinct. 
Let $\cB(i_1,j_1,i_2,j_2)$ denote the event that the paths $[i_1,j_1]$ and $[i_2,j_2]$ are disjoint. 
By the BK-inequality 
\begin{eq}\label{eq:N-geq-2-2ndmoment-2}
&\sum_{\substack{i_1,i_2\in V,j_1, j_2\in [N_n(a)]^c\\
i_1\neq i_2, j_1\neq j_2}} \PR_1 (\cB(i_1,j_1,i_2,j_2) ) \\
&\leq \sum_{\substack{i_1,i_2\in V,j_1, j_2\in [N_n(a)]^c\\
i_1\neq i_2, j_1\neq j_2}} \PR_1(I_{i_1j_1}(a) = 1)\PR_1(I_{i_2j_2}(a) = 1) \leq \big(\E_1[|{\mathcal N}_{\sss \geq 2}(V)|]+o(\sqrt{n})\big)^2,
\end{eq}
where we have used \eqref{eq:lb-geq-2} in the last step.

Let $\cB'(i_1,j_1,i_2,j_2)$ denote the event that $[i_1,j_1]$ and $[i_2,j_2]$ intersect.
If $\cB'(i_1,j_1,i_2,j_2)$ occurs, then there are two vertices $k_1,k_2\in [N_n(a)]^c$ in $[i_1,j_1]$ such that $[i_1,k_1]$, $[k_1,k_2]$, $[k_2,j_1]$, $[i_2,k_1]$ and $[j_2,k_2]$ are edge-disjoint. 
There are two cases depending on whether $k_1 = k_2$ (we denote this event by $\cB_1'(i_1,j_1,i_2,j_2)$) or $k_1\neq k_2$ (and we denote this event by $\cB_2'(i_1,j_1,i_2,j_2)$). 
The BK-inequality implies that 
\begin{eq}\label{eq:N-geq-2-2ndmoment-3}
&\sum_{i_1,i_2\in V} \sum_{j_1,j_2\in [N_n(a)]^c} \PR (\cB'_1(i_1,j_1,i_2,j_2) ) \\
&\leq \sum_{i_1,i_2\in V} \sum_{j_1,j_2\in [N_n(a)]^c}  \sum_{k\in [N_n(a)]^c} \PR(\cA(i_1,k)) \PR(\cA(j_1,k))\PR(\cA(i_2,k))\PR(\cA(j_2,k))\\
&\leq  \frac{C\percn^4}{\ell_n^4} \sum_{i_1,i_2\in V}\sum_{j_1,j_2\in [N_n(a)]^c}\sum_{k\in [N_n(a)]^c} w_{i_1}w_{i_2} w_{j_1}w_{j_2} w_k^4  \\
&\leq Cn\frac{\percn^2}{\ell_n^2}\sum_{k\in [N_n(a)]^c}  w_k^4 \leq C a^{1-4\alpha} n\percn,
\end{eq}
where we have used that $\sum_{k\in [N_n(a)]^c}  w_k^4 \leq C a^{1-4\alpha} N_n n^{2}$, as can be derived similarly as in \eqref{sum-wk3}.

To compute $\PR (\cB'_2(i_1,j_1,i_2,j_2))$, we again apply the BK-inequality, and \eqref{sum-wk3} again implies that 
\begin{eq}\label{eq:N-geq-2-2ndmoment-4}
&\sum_{i_1,i_2\in V} \sum_{j_1,j_2\in [N_n(a)]^c} \PR (\cB'_2(i_1,j_1,i_2,j_2) ) \\
&\leq \sum_{i_1,i_2\in V} \sum_{j_1,j_2\in [N_n(a)]^c} \sum_{k_1,k_2\in [N_n(a)]^c} \PR(\cA(i_1,k_1))\PR(\cA(k_1,k_2)) \PR(\cA(j_1,k_2))\PR(\cA(i_2,k_1))\PR(\cA(j_2,k_2))\\
&\leq  \frac{C\percn^5}{\ell_n^5}\sum_{i_1,i_2\in V}\sum_{j_1,j_2\in [N_n(a)]^c}\sum_{k\in [N_n(a)]^c} w_{i_1}w_{i_2} w_{j_1}w_{j_2} w_{k_1}^3 w_{k_2}^3  \\
&\leq Cn\frac{\percn^3}{\ell_n^3}\bigg(\sum_{k\in [N_n(a)]^c}  w_k^3 \bigg)^2 \leq C a^{2-6\alpha} n\percn.
\end{eq}
Finally, we conclude from \eqref{BK-appl-two-path},  \eqref{BK-multipath-N-geq-2}, \eqref{BK-multipath-N-geq-2-2}, \eqref{eq:N-geq-2-2ndmoment-2}, \eqref{eq:N-geq-2-2ndmoment-3} and \eqref{eq:N-geq-2-2ndmoment-4} that $\Var_1 (|{\mathcal N}_{\sss \geq 2}(V)|) = o(n)$ for each fixed $a>0$.
Thus, on the event that $\E_1[|{\mathcal N}_{\sss \geq 2}(V)|]\leq \vep\sqrt{n}/2$, which occurs with high probability, \eqref{asym-path-i-j} and  the Chebychev inequality imply that 
\eqan{
\PR_1\big(|{\mathcal N}_{\sss \geq 2}(V)| > \vep \sqrt{n}\big)&\leq 
\PR_1\Big(\Big||{\mathcal N}_{\sss \geq 2}(V)|-\E_1[|{\mathcal N}_{\sss \geq 2}(V)|]\Big| > \vep \sqrt{n}/2\Big)\\
&\leq \frac{4\Var_1 (|{\mathcal N}_{\sss \geq 2}(V)|)}{\vep^2n}
\pto 0,\nn
}
and thus the proof of Lemma~\ref{prop:3nbd-negligible} follows.
\end{proof}
Next we bound the total weight of small sets of vertices which will be required in the proof of Proposition~\ref{prop:size-span-C1}:
\begin{lemma}[Small sets have small weight]\label{lem:w-small-set}
Fix any $\delta>0$, and $V\subset [N_n(a)]$ such that $|V| \leq \delta N_n$. Then $\frac{1}{\sqrt{n}}\sum_{k\in V} \percn w_k \leq \frac{\cf \delta ^{1-\alpha}}{1-\alpha}$.
\end{lemma}

\begin{proof}
Note that 
    $\frac{1}{\sqrt{n}}\sum_{k\in V} \percn w_k \leq  
    \frac{1}{\sqrt{n}}\sum_{k\leq  \delta N_n} \percn w_k. $
Using \eqref{theta(i)-def}, we conclude that 
\begin{eq}
    \frac{1}{\sqrt{n}}\sum_{k\leq \delta N_n} \percn w_k \leq \cf n^{-\frac{3-\tau}{2} - \frac{1}{2} + \alpha }\sum_{k \leq \delta N_n} k^{-\alpha}  \leq \frac{\cf}{1-\alpha} \delta^{1-\alpha}. 
\end{eq}
\end{proof}

We are now ready to prove Proposition~\ref{prop:size-span-C1}:
\begin{proof}[Proof of Proposition~\ref{prop:size-span-C1}]
Let $\lambda>\lambda_c$ and fix any $\vep >0$. 
Using Propositions~\ref{prop-weight-giant-core}~and~\ref{prop-large-a-original}, there exists $a_1 = a_1(\vep)>0$ such that for all $a\geq a_1$ there exists $k_0 = k_0(\vep,a)$ such that for all $k\geq k_0$
\begin{eq}\label{eq:1-nbd-mass-ub-lb}
\lim_{n\to\infty} \PR\bigg( \frac{1}{\sqrt{n}}\sum_{i\in \cT_{\sss \geq k}^a} \percn w_i \leq \zeta^\lambda + \frac{\vep}{2} \bigg) = 1,
\end{eq}
Next, we take $\delta = (\vep/4C_0)^{1/(1-\alpha)} $, where $C_0 = \cf/(1-\alpha)$ as in Lemma~\ref{lem:w-small-set}, i.e., $\sum_{k\in V} \percn w_k \leq \vep \sqrt{n}/4$, whenever $|V| \leq \delta N_n$.
Recall the notation $\rho_{a,{\sss \geq k}}^\lambda$ from \eqref{defn:survival-prob}. 
Since $\rho_{a,{\sss \geq k}}^\lambda \searrow \rho_{a}^\lambda$ as $k\to\infty$, we can choose $k_0 = k_0(\vep,a)$ such that, for all $k\geq k_0$, $\rho_{a,{\sss \geq k}}^\lambda \leq \rho_{a}^\lambda+ \delta /2a$. 
Using Proposition~\ref{prop:giant-restricted},
with high probability, 
\begin{eq}
|\cT_{\sss \geq k}^a| \leq |\sC_{\sss (1)}^a| + \delta N_n \quad \implies\quad   |\cT_{\sss \geq k}^a \setminus \sC_{\sss (1)}^a| \leq \delta N_n, 
\end{eq}where the last implication uses that $\sC_{\sss (1)}^a \subset \cT_{\sss \geq k}^a$ with high probabilitiy, since $|\sC_{\sss (1)}^a| = \Theta_{\sss \PR}(N_n(a)) $.
By our choice of $\delta$, and Lemma~\ref{lem:w-small-set}, with high probability 
\begin{eq}
\frac{1}{\sqrt{n}}\sum_{i\in \cT_{\sss \geq k}^a} \percn w_i \leq \frac{1}{\sqrt{n}}\sum_{i\in \sC_{\sss (1)}^a} \percn w_i + \frac{\vep}{4}\leq \zeta^\lambda +\frac{\vep}{2}.
\end{eq}
This concludes the proof of \eqref{eq:1-nbd-mass-ub-lb}.

Using \eqref{eq:1-nbd-mass-ub-lb}, we can now apply Lemmas~\ref{prop:2nbd-giant}~and~\ref{prop:3nbd-negligible} for  $\cT_{\sss \geq k}^a$, to conclude that, for any $a>\max\{a_0,a_1\}$, 
	\begin{eq}\label{eq:span-1nbd-approx} 
	\frac{|\spn(\cT_{\sss \geq k}^a)|}{\sqrt{n}} = (1+\oP(1))\sum_{j\in \cT_{\sss \geq k}^a} \frac{\percn w_j}{\sqrt{n}},
	\end{eq}
where $a_0$ is as in Lemma~\ref{prop:3nbd-negligible}. 
This concludes the proof of Proposition~\ref{prop:size-span-C1}. 
\end{proof}

\subsection{Negligible contribution due to return paths} 
\label{sec:return- path}
Let us start by constructing the graph $\bar{\cG}_{\sss N_n(a)}$ as follows:  $\{i,j\}$ is an edge of $\bar{\cG}_{\sss N_n(a)}$ if and only if $\{i,j\}$ is an edge of $\cG_{\sss N_n(a)}$, {\em or} there exists a path from $i$ to $j$ with all intermediate vertices in~$[N_n(a)]^c$.
We will term the additional edges in~$\bar{\cG}_{\sss N_n(a)}$ as \emph{return edges}.
Henceforth, we augment a previously notation with bar to denote the corresponding quantity for $\bar{\cG}_{\sss N_n(a)}$. For example,  $\bar{\sC}_{\sss (i)}^a$ and $\bar{\cT}_{\sss \geq k}^a$ respectively denote the $i$-th largest component and the number of vertices in components of size~$i$.

Our {\em candidate} giant component in the whole graph $\rNR (\bw,\percn)$ is $\bar{\sC}_{\sss (1)}^{a} \cup \spn(\bar{\sC}_{\sss (1)}^{a})$ for large~$a$. 
Note that the vertices in $\bar{\sC}_{\sss (1)}^{a} \setminus \sC_{\sss (1)}^a$ added due to the  \emph{return edges} are precisely the return vertices, as explained before~\eqref{giant-equality}. In particular, $\bar{\sC}_{\sss (1)}^{a}=\sC_{\sss (1)}^{a} \cup \cR_{\sss (1)}^a,$ so that also
    \eqn{
    \bar{\sC}_{\sss (1)}^{a} \cup \spn(\bar{\sC}_{\sss (1)}^{a})
    =\sC_{\sss (1)}^{a} \cup \spn(\sC_{\sss (1)}^{a}) \cup \cR_{\sss (1)}^a \cup \spn(\cR_{\sss (1)}^a)=\sC_{\sss (1)}^{a,\star}.
    }
The goal of this section is to show that the addition of the return edges can only increase the asymptotics of the span by a negligible amount:

\begin{proposition}[Span with return vertices]
\label{prop:size-span-C2}
There exists $\varepsilon_0>0$ such that for any $\vep \in (0,\varepsilon_0)$, there exists $a_2 = a_2(\vep)>0$ such that, for all $a\geq a_2$, there exists $k_1 = k_1(\vep,a)$ such that, for all $k\geq k_1$,
\begin{eq}
\lim_{n\to\infty} \PR\bigg( \frac{|\spn(\bar{\cT}_{\sss \geq k}^a)|}{\sqrt{n}} \leq \zeta^\lambda + \vep \bigg) = 1,
\end{eq}
where $\zeta^\lambda$ is as in \eqref{a-zeta-a-lim}. 
\end{proposition}
Let us explain the intuition behind the proof.
The main idea is that $\cT_{\sss \geq k}^a$ is \emph{robust} in the sense that its size does not change too much by adding edges to the graph arbitrarily, as long as the number of added edges is small (see Lemma~\ref{lem:robustness} below). 
For this reason, the span of the added vertices is also small (see Lemma~\ref{prop:uniform-small-span} below).
In order to make use of this idea, we later show that
there are not many return edges for large $a$ (see Lemma~\ref{lem:rna-whp-bound} below).
To make these ideas precise, we start with the following elementary fact from \cite[Lemma 9.4]{BJR07}:
\begin{lemma}[{\cite[Lemma 9.4]{BJR07}}] \label{lem:robustness}
Let $G_1,G_2$ be two graphs on the same set of vertices and the edge set of $G_1$ is contained in that of $G_2$. Let $k\geq 1$ and $N_{\geq k}(G_i)$ be the set of vertices with component size at least $k$ in $G_i$ for $i=1,2$. Then $N_{\geq k}(G_1) \leq N_{\geq k}(G_2) \leq N_{\geq k}(G_1) + k \Delta$, where $\Delta$ is the difference between the number of edges in $G_1,G_2$.
\end{lemma}
The next lemma shows that spans of small subsets of $[N_n(a)]$ are uniformly small: 
\begin{lemma}[{Span of small sets in $[N_n(a)]$}] \label{prop:uniform-small-span}
Given any $\varepsilon_1>0$, there exists $a_0=a_0(\varepsilon_1)>0$ such that, for all $a\geq a_0$, 
\begin{eq}
\lim_{n\to\infty}\PR \bigg( \max_{V\subset [N_n(a)]: |V| \leq \varepsilon_1 N_n}  \spn (V) \leq 
C_0\varepsilon_1^{1-\alpha}\sqrt{n}\bigg) =1, 
\end{eq}for some absolute constant $C_0>0$.
\end{lemma}
\begin{proof}
    Recall that  $\bar{\nu}_n(a) = \frac{1}{\ell_n}\sum_{v\notin [N_n(a)]} \percn w_v^2$, and $\cA_{l}(i,j)$ is the event that there exists a path of length $l$ between $i$ and $j$ in $\rNR(\bw,\percn)$ with all the intermediate vertices in~$[N_n(a)]^c$. 
    Let $Z_L:=\sum_{i\in [N_n(a)], j\in [N_n(a)]^c} \sum_{l>L} \mathbbm{1}_{\cA_{l}(i,j)} $. Fix any $\delta>0$. Recall that $\sum_{i\in [N_n(a)]} \percn w_i = Ca^{1-\alpha} \sqrt{n}$. By Markov's inequality, 
    \begin{eq}\label{tail-path-prob}
    \PR\Big(Z_L& > \frac{C_0}{2}\varepsilon_1^{1-\alpha}\sqrt{n}\Big) \leq \frac{2}{C_0\varepsilon_1^{1-\alpha}\sqrt{n}} \sum_{i\in [N_n(a)], j\in [n]}\sum_{l>L }\percn \frac{w_iw_j}{\ell_n} (\bar{\nu}_n(a))^{l-1}\\
    &=  \frac{(\bar{\nu}_n(a))^{L}}{1-\bar{\nu}_n(a)} \frac{C}{\varepsilon_1^{1-\alpha}\sqrt{n}}\sum_{i\in [N_n(a)]}\percn w_i \leq C_1\e^{-C_2L \log a +C_3 \log a + C_4 \log \frac{1}{\varepsilon_1}} \leq \delta, 
    \end{eq}for some 
    $L =L_0 = L_0(\delta)$ (the choice of $L_0$ does not depend on $\varepsilon_1,a$ as long as $a \geq \min\{2, \frac{1}{\varepsilon_1}\}$).

    Next, fix any $V\subset [N_n(a)]$ such that $|V| \leq \varepsilon_1 N_n$.
    We claim that,
    for any $1\leq l \leq L_0$, and any choice of $V$ above,
    \begin{eq}\label{eq:bound-N-k}
    \PR\bigg(|\cN_l(V)| > \frac{C_0  }{2} \varepsilon_1^{1-\alpha} \sqrt{n}\bigg)\leq \e^{-C N_n a^{\frac{1-\alpha}{2}}}, 
    \end{eq}
    where $C$ may depend only on $\varepsilon$, and the inequality holds for all sufficiently large $n$. We first check that \eqref{eq:bound-N-k} implies Lemma~\ref{prop:uniform-small-span}, and then prove \eqref{eq:bound-N-k}.
    Indeed, by \eqref{tail-path-prob}, 
    \begin{eq}
    &\PR \bigg( \max_{V\subset [N_n(a)]: |V| \leq \varepsilon_1 N_n}  \spn (V) > C_0 \varepsilon_1^{1-\alpha} \sqrt{n}\bigg) \\
    &\leq \delta + \binom{N_n(a)}{\floor{\varepsilon_1 N_n}}\max_{V\subset [N_n(a)]: |V| \leq \varepsilon_1N_n} \PR \bigg(  \spn (V) > C_0 \varepsilon_1^{1-\alpha} \sqrt{n},\  Z_L \leq  \frac{C_0}{2}\varepsilon_1^{1-\alpha}\sqrt{n}\bigg) \\
    &\leq \delta + L_0 \e^{\varepsilon_1N_n \log \frac{a}{\varepsilon_1\e}} \max_{V\subset [N_n(a)]: |V| \leq \varepsilon_1 N_n} \max_{l\leq L_0} \PR \bigg(  |\cN_l(V)| > \frac{C_0  }{2} \varepsilon_1^{1-\alpha} \sqrt{n}\bigg) \\
    &\leq \delta+n \e^{\varepsilon_1N_n \log \frac{a}{\varepsilon_1\e} - C N_n a^{\frac{1-\alpha}{2}}}= \delta+o(1),
    \end{eq}
    for all large enough $a$, 
    where in the third step we have used Stirling's approximation 
    \begin{eq}
    \binom{N_n(a)}{\varepsilon_1 N_n}\leq\frac{(N_n(a))^{\varepsilon_1 N_n}}{(\floor{\varepsilon_1 N_n}) ! } \sim \bigg(\frac{N_n(a)}{\floor{\varepsilon_1 N_n}/\e}\bigg)^{ \varepsilon_1 N_n} \sim \e^{\varepsilon_1N_n \log (a/\varepsilon_1\e)}.
    \end{eq}
    Since $\delta>0$ is arbitrary, this completes the proof of Lemma~\ref{prop:uniform-small-span}.
    \medskip
    
    It remains to prove \eqref{eq:bound-N-k}.
    We will prove \eqref{eq:bound-N-k} inductively, along also with the companion estimate
    \begin{eq}\label{vol-conc-1}
    \PR\bigg(\sum_{j\in \cN_{l}(V) } \percn w_j > \frac{C_0  }{14} \varepsilon_1^{1-\alpha} \sqrt{n} \bigg) \leq \e^{-C N_n a^{\frac{1-\alpha}{2}}}.
    \end{eq}
    For $l=0$, \eqref{eq:bound-N-k} holds trivially and \eqref{vol-conc-1} holds by Lemma~\ref{lem:w-small-set}.
    At step $l\geq 1$, let $\cE_l$ 
    denote the good event that the events in \eqref{eq:bound-N-k} and \eqref{vol-conc-1} do not occur.
    Then, 
    \begin{eq}\label{expt-bound-N_l}
    \E\big[|\cN_{l+1}(V)| \ \big\vert\  \cN_{l}(V), \cE_l \big]  
    \leq \sum_{i\in \cN_{l}(V)} \sum_{j\in [n]}  \percn \frac{w_iw_j}{\ell_n} =\sum_{i\in \cN_{l}(V)} \percn w_i \leq \frac{C_0  }{14} \varepsilon_1^{1-\alpha} \sqrt{n},
    \end{eq}
where the last step uses \eqref{vol-conc-1}. 
Note that $\cN_{l+1}(V)$ is, conditionally on $\cup_{r\leq l}\cN_{r}(V)$, a sum of independent indicators. 
Thus, standard concentration inequalities \cite[Corollary 2.4, Theorem 2.8]{JLR00} imply 
    \begin{eq}
    \PR\bigg(|\cN_{l+1}(V)| > \frac{C_0 }{2} \varepsilon_1^{1-\alpha}\sqrt{n}\bigg)\leq \PR(\cE_l^c) + \e^{-C' \sqrt{n}}.
    \end{eq}
Thus \eqref{eq:bound-N-k} follows. To inductively verify \eqref{vol-conc-1}, note that 
    \begin{eq}\label{ex-upper-bound}
    &\E\bigg[\sum_{j\in \cN_{l+1}(V) } \percn w_j \ \Big\vert \ \cN_{l}(V), \cE_l\bigg] \leq \sum_{j\in [N_n(a)]^c} \sum_{i\in \cN_l(V)} \percn w_j \frac{\percn w_iw_j}{\ell_n} \\
    &\leq \frac{C_0  }{14} \varepsilon_1^{1-\alpha} \sqrt{n} \times  \sum_{j\in [N_n(a)]^c} \frac{\percn w_j^2}{\ell_n}  \leq \frac{C_0  }{14}  \frac{\lambda^2 \cf^2}{\mu} a^{1-2\alpha} \varepsilon_1^{1-\alpha} \sqrt{n} \leq \frac{C_0  }{14}   a^{-\alpha+\frac{1-\alpha}{2}} \varepsilon_1^{1-\alpha} \sqrt{n},
    \end{eq}
    for all large enough $a$.
For the concentration, we will use the following elementary fact: 
\begin{fact}\label{fact:concentration}
Fix $k\geq 1$, let $X_i\sim \mathrm{Bernoulli}(p_i)$ independently for $i\in [k]$, and let $a_i$ be such that $\max_i a_i>0$ and $\sum_i a_i p_i \leq x$. Then,
\begin{eq}
\PR\bigg(\sum_{i}a_iX_i > 3x\bigg) \leq \e^{-\frac{x}{\max_i a_i}}. 
\end{eq}

\end{fact}
\begin{proof}
      For any $t\leq 1/\max_i a_i$, Markov's inequality and the independence of $(X_i)_{i\geq 1}$ imply that
  \begin{eq}
  \PR\bigg(\sum_{i \in [k]}a_iX_i > 3x\bigg) &\leq \e^{-3tx} \prod_{i\in [k]} \E[\e^{ta_iX_i}] = \e^{-3tx} \prod_{i\in [r]} \big(1-p_i+p_i\e^{ta_i}\big)\\
  &\leq \e^{-3tx}  \e^{\sum_{i\in [k]}p_i(\e^{ta_i}-1)} \leq \e^{-tx}, 
  \end{eq}  
where in the third step we have used that $1+x\leq \e^{x}$ for any $x\geq 0$, and in the final step we have used that $\e^x-1\leq 2x$ for all $x\in [0,1]$. The proof follows by taking $t=1/\max_i a_i$. 
\end{proof}
Using Fact~\ref{fact:concentration}, and $w_i \leq \cf a^{-\alpha}\sqrt{n} $ for $i\in [N_n(a)]^c$, \eqref{ex-upper-bound} now yields \eqref{vol-conc-1} at step $l+1$. This completes the proof of Lemma \ref{prop:uniform-small-span}.
\end{proof}

We now wish to use Lemmas \ref{lem:robustness} and \ref{prop:uniform-small-span}. The most direct approach would be to apply these lemmas to large $a$, using the fact that with increasing $a$, only few return edges are added. However, this is not possible, as the choice of $\delta$ in Lemma \ref{lem:robustness} {\em also} depend on $a$. Therefore, we introduce an additional parameter $b\gg a$, and apply Lemmas \ref{lem:robustness} and \ref{prop:uniform-small-span} to all the return edges that arise due to paths also touching $[N_n(b)]^c$. 

Let us now present the details of this argument. Fix $b>a$. We say that $i$ and $j$ have a return path \emph{touching} $[N_n(b)]^c$ when there is a path between $i$ to $j$ with intermediate vertices in $[N_n(a)]^c$, and {\em at least 
one} of the intermediate vertices in $[N_n(b)]^c$.
Let $r_{n} (a,b)$ total number of such paths between vertices in $[N_n(a)]$. The following lemma shows that, given $a$, we can choose $b$ so large that the number of return paths touching $[N_n(b)]$ can be made arbitrarily small by choosing $b$ sufficiently large:

\begin{lemma}[{Return touching $[N_n(b)]^c$}]\label{lem:rna-whp-bound}
There exists $a_1>0$ such that for any $\delta>0$ and $a>a_1$, 
\begin{eq}\label{eq:prop-rna}
\lim_{b\to\infty}\lim_{n\to\infty}\PR(r_n(a,b) \leq \delta N_n) = 1.
\end{eq}
\end{lemma}

\begin{proof}
By the Markov inequality, it is enough to show that, for every $a>0$ fixed,
    \begin{eq}\label{eq:exp-rna}
    \lim_{b\to\infty}\limsup_{n\to\infty} \frac{\E[r_n(a,b)]}{N_n}  = 0. 
    \end{eq}
For $i,j\in [N_n(a)]$, let $I_{ij}(a,b)$ denote the indicator that there is a return path from $i$ to $j$ touching~$[N_n(b)]^c$. 
We recycle some notation from the proof of Lemma~\ref{prop:3nbd-negligible}.
We write $\cA_{l}(i,j)$ to denote the event that there exists $l-1$ vertices $i_1,\dots i_{l-1}$, with $i_k\notin [N_n(a)]$ for all $k\leq l-1$ and $i_j \notin [N_n(b)]$ for at least one $j$, such that 
$(i, i_1, \dots i_{l-1},  j)$ is a path in $\rNR(\bw,\percn)$.
In words, $\cA_{l}(i,j)$ is the event that there exists a path of length $l$ between $i$ and $j$ in $\rNR(\bw,\percn)$ with all the intermediate vertices in $[N_n(a)]^c$ and at least one intermediate vertex in $[N_n(b)]^c$.

A return path has minimum length two, so that $l\geq 2$. 
Thus, 
\begin{eq}\label{exp-return-1}
\E[r_n(a,b)] &= \sum_{i,j\in [N_n(a)], i<j} \PR(I_{ij}(a,b) = 1) \leq   \sum_{l\geq 2} \sum_{i,j\in [N_n(a)]} \PR(\cA_l(i,j)).
\end{eq}
We first consider the sum with  $l = 2$. Recall the notation from  \eqref{kappa-a-def}
that 
	$\kappa(u,v) = 1-\e^{-\cf^{2} (uv)^{-\alpha}/\mu}$.
Let $i = \lceil uN_n\rceil$ and $j = \lceil vN_n\rceil$, where $u,v\in (0,a]$. 
Consider another vertex $\lceil xN_n\rceil$ with $x>b$, which corresponds to the vertex outside $[N_n(b)]$ from which the return happens. Note that
    \begin{eq}\label{hub-hub-return-2}
    \PR(\cA_2(i,j)) &= \sum_{k\in [N_n(b)]^c}  \percn^2\big(1-\e^{-w_iw_k/\ell_n}\big)\big(1-\e^{-w_jw_k/\ell_n}\big) \leq \frac{\lambda^2}{N_n} \int_b^\infty \kappa(u,x) \kappa(v,x)\dif x, 
    \end{eq}
and thus 
\begin{eq}\label{exp-return-2}
&\frac{1}{N_n}\sum_{i,j\in [N_n(a)]}\PR(\cA_2(i,j)) \leq  \lambda^2  \int_{0}^a\int_0^a\int_b^\infty \kappa(u,x) \kappa(v,x)\dif x \dif u \dif v, 
\end{eq}
which tends to zero in the iterated limit where  $\lim_{b\to\infty}\limsup_{n\to\infty}$ since the above integral over $x\in [0, \infty)$ is finite for all $a$ fixed.

For $l\geq 3$, the path is of the form $(i_0, i_1, \dots, i_{l-1},  i_l)$ with $i_0 = i$, $i_l = j$.
We split these sums in three cases. 
We say that $\cA_l(i,j,1)$ happens if $i_1\in [N_n(b)]^c$, $\cA_l(i,j,2)$ happens if $i_{l-1}\in [N_n(b)]^c$, and $\cA_l(i,j,3)$ happens if $i_j\in [N_n(b)]^c$ for some $1< j<l-1 $.
We compute 
    \begin{eq}\label{hub-hub-return-gen}
    &\PR(\cA_l(i,j,1))\\
    &\leq \sum_{i_2, \ldots, i_{l-1}\in [N_n(a)]^c,i_1\in [N_n(b)]^c} \prod_{s=0}^{l-1} \percn\big(1-\e^{-w_{i_{s}}w_{i_{s+1}}/\ell_n}\big)
	\\
	&\leq \frac{\percn}{\ell_n} \bigg(\sum_{k\in [N_n(a)]^c}  \frac{\percn w_{k}^2}{\ell_n}\bigg)^{l-3} \sum_{i_1\in [N_n(b)]^c}  \percn w_{i_1}\big(1-\e^{-w_iw_{i_1}/\ell_n}\big)\sum_{i_{l-1}\in [N_n(a)]^c}  \percn w_{i_{l-1}}\big(1-\e^{-w_jw_{i_{l-1}}/\ell_n}\big)\\
	& \leq (\bar{\nu}_n(a))^{l-3} \frac{\lambda^2\cf^2}{\mu} \percn \int_b^\infty x^{-\alpha}\kappa(u,x) \dif x\int_a^\infty y^{-\alpha}\kappa(v,y) \dif y.
\end{eq}
Similarly,
\begin{eq}
&\PR(\cA_l(i,j,2)) \leq (\bar{\nu}_n(a))^{l-3} \frac{\lambda^2\cf^2}{\mu} \percn \int_a^\infty x^{-\alpha}\kappa(u,x) \dif x\int_b^\infty y^{-\alpha}\kappa(v,y) \dif y,
\end{eq}
and 
\begin{eq}\label{exp-return-4}
&\PR(\cA_l(i,j,3)) \leq \bar{\nu}_n(b)(\bar{\nu}_n(a))^{l-4} \frac{\lambda^2\cf^2}{\mu} \percn \int_a^\infty x^{-\alpha}\kappa(u,x) \dif x\int_a^\infty y^{-\alpha}\kappa(v,y) \dif y.
\end{eq}
Taking $a$ large enough so that $\bar{\nu}_n(a) < 1$, it follows that 
\begin{eq}\label{exp-return-3}
&\frac{1}{N_n}\sum_{l\geq 3}\sum_{i,j\in [N_n(a)]}\PR(\cA_l(i,j)) \\
&\leq C \bigg[\bigg(\int_{0}^a\int_b^\infty x^{-\alpha}\kappa(u,x) \dif x \dif u \bigg)^2+ \bar{\nu}_n(b)\bigg(\int_{0}^a\int_a^\infty x^{-\alpha}\kappa(u,x) \dif x \dif u \bigg)^2\bigg].
\end{eq}
Since $\int_0^a\int_b^\infty x^{-\alpha}\kappa(u,x) \dif x \dif u<\infty$ for every fixed $a>0$, the expression in \eqref{exp-return-3} tends to zero in the iterated limit $\lim_{b\to\infty}\limsup_{n\to\infty}$.
Thus, the proof of \eqref{eq:exp-rna} follows by combining \eqref{exp-return-1}, \eqref{exp-return-2} and \eqref{exp-return-3} and the proof of Lemma~\ref{lem:rna-whp-bound} is thus complete.
\end{proof}

We need one final fact before completing the proof of Proposition~\ref{prop:size-span-C2}. 
Fix $b>a$, and define $V_0(b)\subset[N_n(b)]^c$ to be the collection of $j\in [N_n(b)]^c$ such that there is a path from $j$ to some vertex $i\in [N_n(a)]$ with all intermediate vertices in $[N_n(a)]^c$ and at least one intermediate vertex in $[N_n(b)]^c$. The following lemma proves an upper bound on the size of $V_0(b)$:

\begin{lemma}[{Span touching $[N_n(b)]^c$}]\label{lem:intermediate-span}
There exists $a_1$ such that for any $\varepsilon>0$ and $a>a_1$, 
\begin{eq}
\lim_{b\to\infty}\limsup_{n\to\infty} \PR(|V_0(b)| > \varepsilon \sqrt{n}) =0. 
\end{eq}
\end{lemma}
\begin{proof}
Fix any $i\in [N_n(a)]$ and $j\in [N_n(b)]^c$. Let $\cA_l'(i,j)$ be the event that 
there is a path from $j$ to $i$ of length $l$ with all intermediate vertices in $[N_n(a)]^c$ and at least one intermediate vertex in $[N_n(b)]^c$. Using identical computations as \eqref{eq:path-count prob-large-small}, for any $l\geq 2$, 
\begin{eq}
\PR(\cA_l'(i,j)) \leq \percn \frac{w_iw_j}{\ell_n} (\bar{\nu}_n(a))^{l-2}\bar{\nu}_n(b). 
\end{eq}
The $\bar{\nu}_n(b)$ term comes due to one intermediate vertex in $[N_n(b)]^c$. Take $a_1$ to be large enough such that $\bar{\nu}_n(a)<1$ for all $a>a_1$.
  Then, 
 \begin{eq}
 \E[V_0(b)] \leq \sum_{i\in [N_n(a)], j\in [N_n(b)]^c}\sum_{l\geq 2} \PR(\cA_l'(i,j)) \leq \frac{\bar{\nu}_n(b)}{1-\nu_n(a)}\sum_{i\in [N_n(a)]} \percn w_i,
 \end{eq}
 where in the last step we have used that $\sum_{j\in [N_n(b)]^c}  w_j \leq \ell_n$. Using \eqref{eq:sum-weight-N-a}, $\sum_{i\in [N_n(a)]} \percn w_i \leq Ca^{1-\alpha} \sqrt{n}$ and $\bar{\nu}_n(b) \leq C b^{-(3-\tau)/(\tau-1)}$ by \eqref{bar-nu-a-bd}.  Therefore, $\lim_{b\to\infty}\lim_{n\to\infty}\E[V_0(b)] = 0$, and the proof follows using Markov's inequality. 
\end{proof}

Let us close this section by completing the proof of Proposition~\ref{prop:size-span-C2}: 
\begin{proof}[Proof of Proposition~\ref{prop:size-span-C2}]
Fix $\vep>0$ small enough. 
Take $\varepsilon_1 =  [\varepsilon / C_0]^{1/(1-\alpha)}$, so that the bound on the span from Lemma~\ref{prop:uniform-small-span} is $\varepsilon$. 
Next, choose $a_2$ such that  Lemmas~\ref{prop:uniform-small-span},~\ref{lem:rna-whp-bound},~\ref{lem:intermediate-span} and Propositions~\ref{prop-LB-size-span-C1}, \ref{prop:size-span-C1} hold. 
Fix $a\geq a_2$ and let $k_0=k_0(\varepsilon,a)$ be such that the above results work. 
Also, the perturbation $k\Delta$ in Lemma~\ref{lem:robustness} will be taken to be at most $\varepsilon_1 N_n$. 
Fix any $\delta>0$ (sufficiently small) such that Lemma~\ref{lem:rna-whp-bound} holds.
This sets the stage for our proof, and fixes the necessary parameters.

Fix $b>a$ large. We add the additional edges due to return paths leaving $[N_n(a)]$ in two stages, by first adding the edges due to return paths not touching $[N_n(b)]^c$ (i.e., with all intermediate vertices in $[N_n(b)]\setminus[N_n(a)]$), and then adding edges due to return paths  touching $[N_n(b)]^c$.

\paragraph*{Stage 1: } Suppose that we first add the edges due to return paths not touching $[N_n(b)]^c$ to $\cG_{\sss N_n(a)}$. Let $\cG_{\sss N_n(a)}^{\sss +}$ be the graph obtained by starting with $\cG_{\sss N_n(a)}$, and additionally creating an edge between two vertices if such a return path exists between them. 
Define $\cT_{\sss \geq k}^{\sss a, +}$ for the set of vertices with components size at least $k$ in $\cG_{\sss N_n(a)}^{\sss +}$. 
We seek to upper bound $\spn(\cT_{\sss \geq k}^{\sss a, +})$. 

Let $v\in \spn(\cT_{\sss \geq k}^{\sss a, +})$. By definition, there exists a path $(v,i_1,\dots,i_l, u)$ such that $i_{l'}\in [N_n(a)]^c$ for all $l'\in [l]$, and $u\in \cT_{\sss \geq k}^{\sss a, +}$. We write $P$ as a shorthand for $(i_1,\dots,i_l)$.
Consider the following set of exhaustive cases (where in fact several cases can occur at the same time, due to the fact that $P$ is not necessarily unique): 
\begin{enumerate}[(1)]
    \item If $v \in [N_n(b)]$, then there are at most $N_n(b) = o(\sqrt{n})$ choices of $v$ for any fixed $b$. 
    \item If $v \in [N_n(b)]^c$ and $P =\varnothing$, then also $v \in \rspn_b(\cT_{\sss \geq k}^{\sss b})$ (in fact $v$ lies in the one-neighborhood of $\cT_{\sss \geq k}^{\sss b}$). 
    \item If $v \in [N_n(b)]^c$ and $P \neq \varnothing$, then we have the following sub-cases: 
    \begin{enumerate}[(a)]
        \item If $P \subset [N_n(b)]^c$, then $v \in \rspn_b(\cT_{\sss \geq k}^{\sss b})$, using that $\cT_{\sss \geq k}^{\sss a,+} \subset \cT_{\sss \geq k}^{\sss b}$;
        \item If $P \subset [N_n(b)]\setminus [N_n(a)]$, then $v \in \rspn_b(\cT_{\sss \geq k}^{\sss b})$ (since $v$ lies in the one-neighborhood of $\cT_{\sss \geq k}^{\sss b}$);
        \item If $P$ intersects both $[N_n(b)]^c$ and $[N_n(b)]\setminus [N_n(a)]$, then $v\in V_0(b)$, where $V_0(b)$ is defined in Lemma~\ref{lem:intermediate-span}. 
    \end{enumerate}
\end{enumerate}
The above shows that 
\begin{eq}
|\spn(\cT_{\sss \geq k}^{\sss a, +})| \leq o(\sqrt{n}) + |\rspn_b(\cT_{\sss \geq k}^{\sss b})|+ |V_0(b)|. 
\end{eq}
Using Proposition~\ref{prop:size-span-C1} and Lemma~\ref{lem:intermediate-span}, for all $a\geq a_2$,
\begin{eq}\label{span-s1}
\lim_{b\to\infty} \lim_{n\to\infty} \PR\bigg(\frac{|\spn(\cT_{\sss \geq k}^{\sss a, +})|}{\sqrt{n}} \leq  \zeta^\lambda+\varepsilon \bigg) = 1.
\end{eq}
\paragraph*{Stage 2:} Next, we add the return paths  touching $[N_n(b)]^c$ to $\cG_{\sss N_n(a)}^{\sss +}$. 
On top of $\cG_{\sss N_n(a)}^{\sss +}$, if we additionally create an edge between two vertices if  a return path touching $[N_n(b)]^c$ exists between them, then the resulting graph will be $\bar{\cG}_{\sss N_n(a)}$ defined above Proposition~\ref{prop:size-span-C2}. 
By Lemma~\ref{lem:robustness},
$|\bar{\cT}_{\sss \geq k}^{\sss a} \setminus \cT_{\sss \geq k}^{\sss a, +}| \leq \varepsilon_1 N_n$ on the event that $\{r_n(a,b)\leq \delta N_n\}$. 
Thus, Lemmas~\ref{prop:uniform-small-span} and~\ref{lem:rna-whp-bound} show that for all 
$a\geq a_2$ and $k\geq k_0$
\begin{eq}\label{eq:span-return-1}
\lim_{b\to\infty}\lim_{n\to\infty} \PR\big(|\spn(\bar{\cT}_{\sss \geq k}^{\sss a} \setminus \cT_{\sss \geq k}^{\sss a, +}) | \leq \varepsilon \sqrt{n}\big) = 1.
\end{eq}
The proof of Proposition~\ref{prop:size-span-C2} now follows by combining \eqref{span-s1} and \eqref{eq:span-return-1}. 
\end{proof}

\subsection{No large components outside of \texorpdfstring{$[N_n(a)]$}{TEXT}}
\label{sec-no-outside-Nn}
So far, we have studied the maximal component involving vertices from $[N_n(a)]$.
We are left to study the maximal size of clusters that are completely outside of $[N_n(a)]$. 
Recall from \eqref{C-leq-def-2} that $\sC_{\sss \leq}(j)$ is empty when $j\neq \min\{i\colon i\in \sC(j)\}$ and equals $\sC(j)$ otherwise. The main estimate on the cluster size outside of $[N_n(a)]$ is the following lemma:

\begin{lemma}[No large components outside  \protect{$[N_n(a)]$}]
\label{lem:small-comps-outside-core}
For each fixed $a>0$, as $n\to \infty$,
	\eqn{
	(\percn n^{\alpha})^{-1}\max_{j\in [n]\setminus [N_n(a)]} |\sC_{\sss \leq}(j)|\pto 0.
	}
\end{lemma}

\begin{proof} It suffices to prove the statement for $|\sC_{\sss \leq}(j)\setminus \{j\}|$. 
Let $D_j$ denote the degree of $j$. Fix $\vep>0$. We bound
	\eqn{
	\label{split-outside-Nn(a)}
	\PR\Big(\max_{j\in [n]\setminus [N_n(a)]} |\sC_{\sss \leq}(j)\setminus \{j\}|\geq 2\vep \percn n^{\alpha}\Big)
	\leq 
	\sum_{j>N_n(a)} \Big[\PR(D_j\geq \vep \percn n^{\alpha})+\PR(|\cN_{\sss \geq, \geq 2}(j)|\geq \vep \percn n^{\alpha})\Big],
	}
where $\cN_{\sss \geq, \geq 2}(j)$ is the part of the cluster of $\sC_{\sss \leq}(j)$ at distance at least 2 away from $j$. 
Now $D_j$ is a sum of independent $\mathrm{Bernoulli} \big(\percn (1-\e^{-w_jw_k/\ell_n} )\big)$ random variables, and $\E[D_j] \leq \percn w_j$.  
Since $\percn w_j=o(\percn n^{\alpha})$, standard concentration arguments \cite[Corollary 2.4, Theorem 2.8]{JLR00} show that 
\begin{eq}
\PR(D_j\geq \vep \percn n^{\alpha}) \leq \e^{-\percn n^{\alpha}},
\end{eq}
for all sufficiently large $n$.

For the second summand in \eqref{split-outside-Nn(a)}, we use the Markov inequality to bound
	\eqn{
	\PR(|\cN_{\sss \geq, \geq 2}(j)|\geq \vep \percn n^{\alpha})\leq (\vep \percn n^{\alpha})^{-2} \E[|\cN_{\sss \geq, \geq 2}(j)|^2].
	}
The expectation can be computed using path counting again similar to \eqref{eq:n-geq-2-2ndmoment}.
Since $j$ is the minimum index of $\sC_{\sss \leq}(j)$, we will need the paths to have vertices with indices higher than $j$ only.
Let $J_{ij}$ denote the indicator that $i$ and $j$ are connected via a path of length at least 2 with all intermediate vertices having index  at least $j$. 

Thus
	\eqn{
	\E[|\cN_{\sss \geq, \geq 2}(j)|^2]= \sum_{i_1,i_2\geq j} \PR(I_{i_1j} =1, I_{i_2j} = 1),
	}where $I_{ij}$ is the indicator that there is a path from $i$ to $j$ with all intermediate vertices having index at least $j$.
We can decompose the above in two cases depending on whether the paths $[i_1,j]$ and $[i_2,j]$ are disjoint or not. Denote the two cases by $\text{(I)}$ and $\text{(II)}$, respectively. 
Using the BK-inequality \cite[Theorem 3.3]{BK85} again yields
\begin{eq}
\text{(I)} \leq \bigg(\sum_{i_1\geq j} \PR(I_{i_1j} =1)\bigg)^2 &\leq \bigg(\sum_{i_1>j}\sum_{l\geq 2}\sum_{\substack{k_1, \ldots, k_{l-1}\geq j\\
k_0 = i_1, k_l = j}} \prod_{s=1}^l \frac{\percn w_{k_{s-1}}w_{k_s}}{\ell_n} \bigg)^2\\
&\leq (\percn w_j)^2 \bigg(\sum_{l\geq 2}\nu_n(j)^{l-1}\bigg)^2,
\end{eq}
where now
	\eqn{
	\nu_n(j)=\percn \sum_{k\geq j}\frac{w_k^2}{\ell_n}
	\leq C\percn n^{2\alpha-1}  j^{1-2\alpha}.
	}
If the paths $[i_1,j]$ and $[i_2,j]$ are not disjoint, then three disjoint paths exist $[i_1,k]$, $[i_2,k]$ and $[k,j]$ for some $k>j$. 
Therefore, applying the BK-inequality \cite[Theorem 3.3]{BK85} once again,
\begin{eq}
\text{(II)} &\leq \sum_{i_1,i_2, k>j}\prod_{k_0 \in \{i_1,i_2,j\}} \sum_{l\geq 2}\sum_{\substack{k_1, \ldots, k_{l-1}\geq j, k_l = k}} \prod_{s=1}^l \frac{\percn w_{k_{s-1}}w_{k_s}}{\ell_n} \\
	&\leq (\percn w_j) \frac{\percn^2 \sum_{k\geq j} w_k^3}{\ell_n}\bigg(\sum_{l\geq 2}\nu_n(j)^{l-1}\bigg)^3 \leq (\percn w_j) \frac{\percn^2 \sum_{k\geq j} w_k^3}{\ell_n},
\end{eq}
since $\nu_n(j) \leq C \perc n (n/j)^{2\alpha-1} \leq \tfrac{1}{2}$ for $j>N_n(a)$. 
Using $\sum_{k>j} w_k^3/\ell_n= O(n^{3\alpha-1} j^{1-3\alpha})$, we conclude that	
	\eqan{
	\PR(|\cN_{\sss \geq, \geq 2}(j)|\geq \vep \percn n^{\alpha})&\leq O(1)(\vep \percn n^{\alpha})^{-2}    \Big[\percn^4 n^{6\alpha-2} j^{2-6\alpha}
	+\percn^3 n^{4\alpha-1} j^{1-4\alpha}\Big]\\
	&=O(1) \Big[\percn^2 n^{4\alpha-2} j^{2-6\alpha}
	+\percn n^{2\alpha-1} j^{1-4\alpha}\Big],\nn
	}
so that
	\eqan{
	\sum_{j>N_n(a)} \PR(|\cN_{\sss \geq, \geq 2}(j)|&\geq \vep \percn n^{\alpha})\leq O(1) \sum_{j>N_n(a)}\Big[\percn^2 n^{4\alpha-2} j^{2-6\alpha}
	+\percn n^{2\alpha-1} j^{1-4\alpha}\Big]\\
	&=O(1) \Big[\percn^2  n^{4\alpha-2} N_n^{3-6\alpha}+\percn n^{2\alpha-1} N_n^{2-4\alpha}\Big]=O(1)N_n^{1-2\alpha}=o(1),\nn
	}
since $\alpha>\tfrac12$. This proves Lemma~\ref{lem:small-comps-outside-core}. 
\end{proof}

\subsection{Completing the proof of Theorem~\ref{thm:supcrit-bd}}
\label{sec:proof-main-thm}
We now have all the ingredients to complete the proof of Theorem~\ref{thm:supcrit-bd}.
First, by Lemma~\ref{lem:small-comps-outside-core}, the giant component $\sC_{\sss (1)}$ for the whole graph is one of the components of vertices in $[N_n(a)]$ with high probability. 
Recall from the beginning of Section~\ref{sec-size-tiny-giant} that $\sC_{\sss (i)}^{a,\star}$ the component of $\rNR(\bw,\percn)$ containing $\sC_{\sss (i)}^{a}$. 
Also, recall the definition of $\bar{\cG}_{\sss N_n(a)}$ and its functionals from Section~\ref{sec:return- path}. Fix $\varepsilon>0$. 
Since $\sC_{\sss (1)}^a\subset \bar{\sC}_{\sss (1)}^a \subset \bar{\cT}_{\sss \geq k}^a$ with high probability, Propositions~\ref{prop-LB-size-span-C1}~and~\ref{prop:size-span-C2} show that we can choose $a_0>0$ so large that, for $a>a_0$,
\begin{eq}\label{small-giant-asymp}
\lim_{n\to\infty} \PR \big(\sqrt{n}(\zeta^\lambda - \vep) \leq |\sC_{\sss (1)}^{a,\star}|\leq  \sqrt{n}(\zeta^\lambda + \vep) \big) =1.
\end{eq}
Moreover, Proposition~\ref{prop:size-span-C2} and \eqref{small-giant-asymp} also show that there exists a large enough $k_0 = k_0(\varepsilon,a)$ such that for all $a>a_0$ and $k\geq k_0$, 
\begin{eq}\label{small-giant-asymp-2}
\lim_{n\to\infty} \PR \bigg(\max_{\sC \subset   \bar{\cT}_{\sss \geq k}^a\setminus \bar{\sC}_{\sss (1)}^a } \big|\sC \cup \spn(\sC)\big| \leq \varepsilon \sqrt{n}\bigg) = 1.
\end{eq}
Further, since all components outside $\bar{\cT}_{\sss \geq k}^a$ have size at most $k_0$, an application of
by Lemma~\ref{prop:uniform-small-span} shows that
\begin{eq}\label{small-giant-asymp-3}
\lim_{n\to\infty} \PR \bigg(\max_{\sC \subset   (\bar{\cT}_{\sss \geq k}^a)^c} \big|\sC \cup \spn(\sC)\big| \leq \varepsilon \sqrt{n}\bigg) = 1,
\end{eq}
where the maximum runs over all connected components $\sC \subset (\bar{\cT}_{\sss \geq k}^a)^c$.
Finally, in Lemma~\ref{lem:small-comps-outside-core} we have shown that the components not involving vertices in $[N_n(a)]$ have size at most $o(\sqrt{n})$. 
Thus, with high probability, $\sC_{\sss (1)}^{a,\star}$ is the unique giant component of $\rNR(\bw, \percn)$ with size given by \eqref{small-giant-asymp}, and the second largest component has size at most $\varepsilon \sqrt{n}$. This proves the statements in Theorem~\ref{thm:supcrit-bd} about the uniqueness of the  giant component.

We complete the proof by showing that hubs are very likely to be in the newly born giant. Fix $\delta>0$, and consider the set of hubs given by $H=\{h\colon w_h\geq n^{1/2+\delta}\}$. We will show that $H\subseteq \sC_{\sss (1)}(\percn)$ with high probability.
Remove the set of hubs $H$ from the graph. The giant $\sC_{\sss (1)}^{a,H}$ in $[N_n(a)]\setminus H$ has all the same characteristics as the original giant in $[N_n(a)]$, since the removal of a small fraction of vertices has hardly any effect on the giant, as shown by our previous analysis.

Fix $h\in H$. We will condition on $\sC_{\sss (1)}^{\sss a,H}$, and consider the two-hop paths between $h$ and $\sC_{\sss (1)}^{\sss a,H}$ consisting of paths $h\longrightarrow j\longrightarrow \sC_{\sss (1)}^{\sss a,H}$ for $j\in [n]\setminus [N_n(a)]$. Write
    \eqn{
    M_{n,h}=\sum_{j\in [n]\setminus [N_n(a)]}\indic{h\longrightarrow j\longrightarrow \sC_{\sss (1)}^{\sss a,H}}
    }
for the number of $j$ that are forming the two-hop paths. For fixed $h$, and conditionally on $\sC_{\sss (1)}^{a,H}$, the indicators are {\em independent}. We next consider their success probabilities.

Denote the conditional probability given $\sC_{\sss (1)}^{\sss a,H}$ by $\mathbb{P}_{1,H}$. Note that, for $v\in \sC_{\sss (1)}^{\sss a,H}$, the probability that $v$ is {\em not} connected to $j$ after percolation equals $1-\pi_n+\pi_n\e^{-w_vw_j/\ell_n}$. Thus, the conditional probability given $\sC_{\sss (1)}^{\sss a,H}$ that $j$ is connected to some $v\in\sC_{\sss (1)}^{\sss a,H}$ equals
    \eqn{
    \mathbb{P}_{1,H}(j\longrightarrow \sC_{\sss (1)}^{\sss a,H})
    =1-\prod_{v\in \sC_{\sss (1)}^{\sss a,H}}\big(1-\pi_n+\pi_n\e^{-w_vw_j/\ell_n}\big).
    }
Since the event $\{v\conn j\}$ is independent of the event $
\{j\longrightarrow \sC_{\sss (1)}^{\sss a,H}\},$ this leads to
    \eqan{
    \mathbb{P}_{1,H}(h\longrightarrow j\longrightarrow \sC_{\sss (1)}^{\sss a,H})&=\pi_n [1-\e^{-w_hw_j/\ell_n}]
    \Big[1-\prod_{v\in \sC_{\sss (1)}^{\sss a,H}}\big(1-\pi_n+\pi_n\e^{-w_vw_j/\ell_n}\big)\Big].
    }
Restrict the product over $v$ to $v\in \sC_{\sss (1)}^{a,H}\setminus [N_n(1)]$. Then, $w_vw_j/\ell_n\leq \tfrac{1}{2}$ for $a$ large, since $j>N_n(a)$, and, in turn, $\e^{-x}\leq 1-x/2$ when $x\leq \tfrac{1}{2}$. This leads to the upper bound
    \eqn{
    1-\pi_n+\pi_n\e^{-w_vw_j/\ell_n}\leq 1-\frac{w_vw_j\pi_n}{2\ell_n}
    \leq \e^{-\frac{\pi_nw_vw_j}{2\ell_n}},
    }
which in turn implies the lower bound
    \begin{eq}
    \mathbb{P}_{1,H}(h\longrightarrow j\longrightarrow \sC_{\sss (1)}^{\sss a,H})&\geq \pi_n [1-\e^{-w_hw_j/\ell_n}]
    \bigg[1-\prod_{v\in \sC_{\sss (1)}^{\sss a,H}\setminus [N_n(1)]}\e^{-\frac{\pi_n w_vw_j}{2\ell_n}}\bigg]\\
    &=\pi_n [1-\e^{-w_hw_j/\ell_n}]\bigg[1- \exp\bigg(-\pi_n w_j \sum_{v\in \sC_{\sss (1)}^{\sss a,H}\setminus [N_n(1)]} \frac{w_v}{2\ell_n}\bigg)\bigg].
    \end{eq}
Since Proposition \ref{prop-weight-giant-core} implies that $\pi_n \sum_{v\in \sC_{\sss (1)}^{\sss a,H}\setminus [N_n(1)]}w_v\geq \vep \ell_n/\sqrt{n}$ whp for $\vep>0$ sufficiently small, we obtain, whp and for some $\vep'>0$ small,
    \eqan{
    \mathbb{P}_{1,H}(h\longrightarrow j\longrightarrow \sC_{\sss (1)}^{a,H})
    &\geq 
    \pi_n [1-\e^{-w_hw_j/\ell_n}][1-\e^{-\vep w_j/\sqrt{n}}]
    \geq \vep' \frac{w_j}{\sqrt{n}} \pi_n [1-\e^{-w_hw_j/\ell_n}].
    }
This is true for all $j>N_n(a)$. By independence, we conclude that
    \eqn{
    \mathbb{P}_{1,H}(M_{n,h}=0)
    \leq \prod_{j>N_n(a)} \Big(1-\vep' \frac{w_j}{\sqrt{n}} \pi_n [1-\e^{-w_hw_j/\ell_n}]\Big).
    }
We compute
    \eqan{
    \sum_{j>N_n(a)}\frac{w_j}{\sqrt{n}} \pi_n [1-\e^{-w_hw_j/\ell_n}]
    &\geq 
    \sum_{j>N_n(a)}\frac{w_j}{\sqrt{n}} \pi_n [1-\e^{-w_j n^{1/2+\delta}/\ell_n}]\\
    &\geq [1-\e^{-n/\ell_n}]\sum_{j>N_n(a)\colon w_j\leq n^{1/2-\delta}}\frac{w_j \pi_n}{\sqrt{n}},\nn
    }
and
    \eqn{
    \sum_{j>N_n(a)\colon w_j\leq n^{1/2-\delta}}\frac{w_j \pi_n}{\sqrt{n}}
    =(1+o(1))\frac{\ell_n \pi_n}{\sqrt{n}}\rightarrow \infty
    }
faster than any power of $\log{n}$. Therefore, $\mathbb{P}_{1,H}(M_{n,h}=0)=o(1/n)$, and thus,     
    \eqn{
    \mathbb{P}_{1,H}(\exists h\in H\colon M_{n,h}=0)\rightarrow 0,
    }
which completes the proof of the fact that all the hubs are in the giant component. 
Hence, the proof of Theorem~\ref{thm:supcrit-bd} is also complete.
\qed

\paragraph*{Acknowledgments.} SB was partially supported by NSF grants DMS-1613072, DMS-1606839 and ARO grant W911NF-17-1-0010. SD was partially supported by 
Vannevar Bush Faculty Fellowship ONR-N00014-20-1-2826. The work of RvdH is supported in part by the Netherlands Organisation for Scientific Research (NWO) through the Gravitation {\sc NETWORKS} grant no. 024.002.003.

{\small \bibliographystyle{abbrv}
\bibliography{ultradensebib-single-edge}
}

\end{document}